\documentclass[onefignum,onetabnum]{siamonline190516}

\usepackage{amssymb}
\usepackage{graphicx}
\usepackage{epstopdf}
\ifpdf
  \DeclareGraphicsExtensions{.eps,.pdf,.png,.jpg}
\else
  \DeclareGraphicsExtensions{.eps}
\fi
\usepackage{amsfonts,extarrows,leftidx,cancel}
\usepackage{bbm}
\usepackage[mathscr]{euscript}
\usepackage{algpseudocode}
\usepackage{algorithm}
\usepackage{enumitem}
\usepackage{subfig}
\graphicspath{{./figures/}}

\newcommand{\Rb}{\mathbb{R}}

\newcommand{\Sb}{\mathbb{S}}

\newcommand{\sfC}{\mathsf{C}}
\newcommand{\sfU}{\mathsf{U}}
\newcommand{\sfV}{\mathsf{V}}

\newcommand{\sfL}{\mathsf{L}}
\newcommand{\sfK}{\mathsf{K}}

\newcommand{\sfTheta}{\mathsf{\Theta}}

\newcommand{\MM}{{\mathrm{M}}}
\newcommand{\NN}{{\mathrm{N}}}

\newcommand{\eps}{\varepsilon}

\newcommand{\algrule}[1][.2pt]{\hspace*{-.6in}\hrulefill}
\algdef{SE}[SUBALG]{Indent}{EndIndent}{}{\algorithmicend\ }%
\algtext*{Indent}
\algtext*{EndIndent}

\DeclareMathOperator*{\argmax}{argmax}
\DeclareMathOperator*{\argmin}{argmin}

\newcommand{\rmd}{\mathrm{d}}

\newsiamremark{remark}{Remark}
\newsiamremark{hypothesis}{Hypothesis}
\crefname{hypothesis}{Hypothesis}{Hypotheses}
\newsiamthm{claim}{Claim}
\newtheorem{assumption}{Assumption}[section]
\headers{On optimal bases for multiscale PDEs and Bayesian homogenization}{S. Chen, Z. Ding, Q. Li, and S. J. Wright}

\title{On optimal bases for multiscale PDEs and Bayesian homogenization\thanks{Submitted to the editors DATE.
\funding{This work was funded by ONR-N00014-21-1-214, NSF DMS-2023239 and CCF-2224213, NSF-CAREER-1750488, AFOSR Grant FA9550-21-1-0084, and Office of the Vice Chancellor for Research and Graduate Education at the University of Wisconsin Madison with funding from the Wisconsin Alumni Research Foundation. }}}

\author{Shi Chen\thanks{Department of Mathematics, University of Wisconsin-Madison, Madison, WI
  (\email{schen636@wisc.edu}, \url{https://simonchenthu.github.io/}).}
  \and Zhiyan Ding\thanks{Department of Mathematics, University of California-Berkeley, Berkeley, CA
  (\email{zding.m@math.berkeley.edu}, \url{https://math.berkeley.edu/\string~zding.m/}).}
  \and Qin Li\thanks{Department of Mathematics, University of Wisconsin-Madison, Madison, WI
  (\email{qinli@math.wisc.edu}, \url{https://people.math.wisc.edu/\string~qinli/}).}
  \and Stephen J. Wright\thanks{Department of Computer Sciences, University of Wisconsin-Madison, Madison, WI
  (\email{swright@cs.wisc.edu}, \url{https://pages.cs.wisc.edu/\string~swright/}).}}

\usepackage{amsopn}
\DeclareMathOperator{\diag}{diag}

\ifpdf
\hypersetup{
  pdftitle={On optimal bases for multiscale PDEs and Bayesian homogenization},
  pdfauthor={S. Chen, Z. Ding, Q. Li, and S. J. Wright}}
\fi

\begin{document}

\maketitle

\begin{abstract}
PDE solutions are numerically represented by basis functions. Classical methods employ pre-defined bases that encode minimum desired PDE properties, which naturally cause redundant computations.
What are the best bases to numerically represent PDE solutions? From the analytical perspective, the Kolmogorov $n$-width~\cite{Kolmogoroff1936UberDB} is a popular criterion for selecting representative basis functions. From the Bayesian computation perspective, the concept of optimality  selects the modes that, when known, minimize the variance of the conditional distribution of the rest of the solution~\cite{doi:10.1137/140974596}.
We show that these two definitions of optimality are equivalent.
Numerically, both criteria reduce to solving a Singular Value Decomposition (SVD), a procedure that can be made numerically efficient through randomized sampling.
We demonstrate computationally the effectiveness of the  basis functions so obtained on several linear and nonlinear problems. In all cases, the optimal accuracy is achieved with a small set of basis functions.
\end{abstract}



\section{Introduction}

Given an arbitrary PDE system, how do we find the best basis functions to represent its solutions?

When ``best" is defined in terms of the  relative distance between a solution and its projection onto the space spanned by the basis functions, the Kolmogorov $n$-width achieves the best measure over all possible collections of $n$ basis functions~\cite{Kolmogoroff1936UberDB}.
Traditional numerical methods use basis functions determined ``manually,'' for example, the orthogonal polynomials / Fourier bases used in spectral methods~\cite{GoOr:1977numerical,Be:1997spectral,Tr:2000spectral}, piecewise polynomials used in Continuous Galerkin (CG)~\cite{Ci:2002finite,Br:2008mathematical} or Discontinuous Galerkin (DG) methods~\cite{He:2007nodal,CoKaSh:2012discontinuous}, and classical finite differences~\cite{Le:2007finite,Th:2013numerical}.
When information about the equations is not fully utilized in the construction of the bases, it can contain much redundancy: Considerably more than $n$ functions may be required to achieve the $n$-width. For elliptic PDEs, especially in the context of multi-scale computing, the design of optimal basis was pioneered in~\cite{BaLi:2011optimal}. Relevant approaches include the generalized finite element method~\cite{BaLiSiSt:2020multiscale,SmPaAn:2016optimal,MaSc:2021novel,MaSc:2022error}, optimal basis with hierarchical structure~\cite{Ow:2017multigrid,OwZh:2017gamblets,OwSc:2019operator}, nonlocal multicontinua upscaling~\cite{ChEfLe:2018constraint,ChEfLeVaWa:2018non}, methods based on random SVD~\cite{CaEfGaLi:2016randomized,BuSm:2018randomized,ChLiLuWr:2020randomized,ChLiLuWr:2020random} and the exponentially convergent multiscale finite element method~\cite{ChHoWa:2021exponential,ChHoWa:2023exponentially}.
For PDEs of general types, techniques are not known for determining the optimal basis accurately and efficiently. For this reason, the Kolmogorov $n$-width is believed to be useful conceptually but not especially relevant to practice.

Another perspective on this question was explored in the pioneering paper~\cite{doi:10.1137/140974596}.
There it was argued that one can view the PDE as a map from a source term modeled as a Gaussian process to the PDE solution that is also viewed as a Gaussian field with a shifted mean and stretched covariance that encodes information about the Green's functions.
When the solution is tested on some basis functions, the whole solution becomes a conditioned Gaussian.
From this perspective, the ``best" basis would be made up of the test functions that are most revealing --- the ones that minimize the variance of the conditioned Gaussian.

These discussions raise two questions.
\begin{itemize}
\item Conceptually:
{\em What is the relationship between these two perspectives on finding optimal basis functions?}
\item Practically: {\em Can we find the optimal basis functions numerically in an effective way?}
\end{itemize}
We address both questions in this article.
Conceptually, we will argue that the two viewpoints --- viewing the solution as a Gaussian process, and viewing the solution an element in the solution space --- coincide.
The basis functions that achieve the Kolmogorov $n$-width exactly correspond to the test functions that ``reveal" the variance information of the Gaussian field.
In other words, tXShe most revealing test functions correspond to the basis functions that best approximate the solution space.
Moreover, the question of optimality can be formulated as a singular value decomposition (SVD), with the optimal basis and the optimal test functions corresponding to left and right singular vectors, respectively.
For the second question on practical computations, the relationship between the Kolmogorov $n$-width and the singular value decomposition enables us to argue that the optimal basis functions can be found quite easily, especially now that SVD can be computed in linear time complexity~\cite{doi:10.1137/090771806}.

In the following sections, we justify the claims of the previous paragraph.
For linear systems, we argue in \Cref{sec:two_concepts} that the  basis functions that attain optimal Kolmogrov $n$-width coincide with optimal Bayesian homogenization.
In \Cref{sec:kol_dis} we explain their relation with SVD and translate the fast SVD solver to our PDE setting.
\Cref{sec:nonlinear} explores the  extension to nonlinear systems whose primary components can be spelled out as a linear operator so one can iterate using the basis functions built a priori.
Numerical tests are presented in~\Cref{sec:numerical} to demonstrate the performance of the optimal basis functions.

\section{Two concepts of optimality}\label{sec:two_concepts}

In this section and the next, we focus on PDEs defined by linear systems in steady state, formulated as
\begin{equation}\label{eqn:MPc}
\mathcal{L} (u(x))=f(x),\quad x\in\Omega,
\end{equation}
where $\mathcal{L}: \mathcal{Y}\rightarrow \mathcal{X}$ is a linear operator that maps the solution $u(x)$ to the source term $f(x)$.
$\mathcal{Y}$ and $\mathcal{X}$ are two Hilbert spaces defined on $\mathbb{R}^d$ and equipped with inner products $\left\langle\cdot,\cdot\right\rangle_{\mathcal{Y}}$ and $\left\langle\cdot,\cdot \right\rangle_{\mathcal{X}}$, respectively.
We often choose $\mathcal{Y}$ and $\mathcal{X}$ to be Sobolev spaces such as $H^s(\mathbb{R}^d)$ for $s\geq 0$.
We assume throughout that the equation has a compact solution map:
\begin{assumption}
\label{assum:SR}
The equation \eqref{eqn:MPc} admits a Green's function $G(x;z)\in\mathcal{S}'(\mathbb{R}^d)$ (the space of tempered distributions on $\mathbb{R}^d$) defined for each $z\in\Omega$  as the weak solution to
\begin{equation}\label{eqn:G1}
\mathcal{L} (G(x;z))=\delta_z(x),\quad x\in\Omega.
\end{equation}
Defining the linear solution operator
\begin{equation}\label{eqn:defphi}
\mathcal{S}(f)(x)=\int_\Omega G(x;z)f(z)\rmd z\,,\quad \forall f\in\mathcal{X}\,,
\end{equation}
we assume that $\mathcal{S}$ is
a bounded injective operator that maps $(\mathcal{X},\|\cdot\|_{\mathcal{X}})$ to $ (\mathcal{Y},\|\cdot\|_{\mathcal{Y}})$, and thus the solution to~\eqref{eqn:MPc} is $u=\mathcal{S}(f)$. Further,  defining  $\mathcal{S}^\ast$ to be the adjoint operator of $\mathcal{S}$, we have that $\mathcal{S}^*\mathcal{S}:(\mathcal{X},\|\cdot\|_{\mathcal{X}})\rightarrow (\mathcal{X},\|\cdot\|_{\mathcal{X}})$ is a compact operator.
\end{assumption}

Under this assumption, $\mathcal{S}$ admits a singular value decomposition (SVD).
Denoting by $\{\lambda_i, \hat{u}_i, \hat{v}_i\}_{i=1}^\infty$ the singular value-function triples for the operator $\mathcal{S}$, we have
\begin{equation}\label{eqn:basisw}
\mathcal{S}(\hat{v}_i) = \lambda_i \hat{u}_i\,,\quad \mathcal{S}^\ast(\hat{u}_i) = \lambda_i \hat{v}_i\,, \quad i=1,2,\dotsc\,,
\end{equation}
where $\lambda_1 \ge \lambda_2 \ge \dotsc > 0$.
The functions $\{\hat{u}_i\}_{i=1}^\infty$ and $\{\hat{v}_i\}_{i=1}^\infty$ are called left- and right-singular functions, and form orthonormal bases of $\mathcal{Y}$ and $\mathcal{X}$, respectively.
We assume henceforth without loss of generality that all the eigenvalues are distinct, so that $\lambda_1>\lambda_2>\lambda_3>\dotsc>0$.
Assumption \ref{assum:SR} is satisfied for most linear PDEs under standard choices of $\mathcal{X},\mathcal{Y}$.
For later notational convenience, we also define
\[
\mathcal{U}=\left\{u\in\mathcal{Y}\, \middle| \, \text{$u$ solves \eqref{eqn:MPc} with $f\in\mathcal{X}$}\right\}=\left\{\mathcal{S}(f) \, \middle| \, f\in\mathcal{X}\right\}
=\mathrm{span}\{ \hat{u}_i\}_{i=1}^\infty\,.
\]

We need to use the discrete version of the problem in the following sections.
Denoting by $N$ the number of discrete representation of solution as well as the sources, we write the discrete version of  \eqref{eqn:MPc} as
\begin{equation}\label{eqn:eqn_dis}
    \mathsf{L} \mathsf{u} = \mathsf{f} \,,
\end{equation}
where $\mathsf{u}$ and $\mathsf{f} \in \mathbb{R}^N$ are the numerical solution and the discretized source, respectively, and $\mathsf{L} \in \mathbb{R}^{N\times N}$ is the discrete version of the operator $\mathcal{L}$.
To represent the solution operator $\mathcal{S}$ in the discrete setting, we introduce the  Green's matrix $\mathsf{G}\in \Rb^{N\times N}$, the discrete version of Green's function, defined to satisfy
\begin{equation}\label{eqn:Gproperty}
    \mathsf{L}
\cdot \mathsf{G}
=\mathsf{I}_{N}\,,
\end{equation}
where $\mathsf{I}_{N}\in \Rb^{N\times N}$ is the identity matrix.
For a well-posed system, $\mathsf{L}$ is invertible, so that  $\mathsf{G} = \mathsf{L}^{-1}$ is well defined.
To approximate the norms of $\|\cdot\|_{\mathcal{X}}$ and $\|\cdot\|_{\mathcal{Y}}$, the domain and the range for $\mathsf{G}$ both need to be equipped with a specially designed metric.
To do so, we denote for all $ \mathsf{a},\mathsf{b}\in\Rb^{N}$ the (weighted) discrete inner products
\begin{equation}\label{eqn:norm_dis}
\langle \mathsf{a},\mathsf{b}\rangle_{\mathsf{X}}=\mathsf{a}^\top \mathsf{\Pi}_{\mathsf{X}} \mathsf{b},\quad \langle \mathsf{a},\mathsf{b}\rangle_{\mathsf{Y}}=\mathsf{a}^\top \mathsf{\Pi}_{\mathsf{Y}} \mathsf{b}\,,
\end{equation}
and $\|\cdot\|_{\mathsf{X}}$ and $\|\cdot\|_{\mathsf{Y}}$ are the discrete norms associated with these inner products in the finite dimensional spaces $\mathsf{X}$ and $\mathsf{Y}$, respectively, where $\mathsf{\Pi}_{\mathsf{X}}$ and $\mathsf{\Pi}_{\mathsf{Y}}$ are symmetric positive definite matrices which encode the weights that make $\|\cdot\|_{\mathsf{X}}$ and $\|\cdot\|_{\mathsf{Y}}$ resemble $\|\cdot\|_{\mathcal{X}}$ and $\|\cdot\|_{\mathcal{Y}}$, respectively.
As such, the discrete Green's matrix $\mathsf{G}$ is a map from $(\mathbb{R}^N,\|\cdot\|_{\mathsf{X}})$ to $(\mathbb{R}^N,\|\cdot\|_{\mathsf{Y}})$.
Similarly, we represent the adjoint operator $\mathcal{S}^\ast$ by the discrete representor $\mathsf{G}^\ast$, which maps $(\mathbb{R}^N,\|\cdot\|_{\mathsf{Y}})$ to $(\mathbb{R}^N,\|\cdot\|_{\mathsf{X}})$.
An explicit formula for $\mathsf{G}^\ast$ can be derived from the weight matrices $\mathsf{\Pi}_{\mathsf{X}}$ and $\mathsf{\Pi}_{\mathsf{Y}}$.
For any $\mathsf{v}, \mathsf{w}$ in $\mathbb{R}^N$, we have
\[
\mathsf{v}^\top (\mathsf{G}^\ast)^\top \mathsf{\Pi}_{\mathsf{X}} \mathsf{w}
= \underbrace{\langle \mathsf{G}^\ast \mathsf{v}\,,\mathsf{w} \rangle_{\mathsf{X}}
= \langle \mathsf{v}\,,\mathsf{G}\mathsf{w}\rangle_{\mathsf{Y}} }_{\text{definition of adjoint}}
= \mathsf{v}^\top \mathsf{\Pi}_{\mathsf{Y}}\mathsf{G}\mathsf{w}\,,
\]
where we denote $\mathsf{G}^\top$ the transpose of $\mathsf{G}$, so that
\begin{equation}\label{eqn:G_ast}
\mathsf{G}^\ast = \mathsf{\Pi}_{\mathsf{X}}^{-1} \mathsf{G}^\top \mathsf{\Pi}_{\mathsf{Y}}\,.
\end{equation}

We can find the SVD for $\mathsf{G}$ in this setting.
Specifically, we seek a list $\{\hat{\lambda}_i, \mathsf{\hat{u}}_i, \mathsf{\hat{v}}_i\}_{i=1}^N$ so that the basis vectors are orthonormal in the weighted space:
\begin{equation} \label{eqn:SVD_1}
    \mathsf{\hat{u}}_i^\top \mathsf{\Pi}_\mathsf{Y} \mathsf{\hat{u}}_j = \mathsf{\hat{v}}_i^\top \mathsf{\Pi}_\mathsf{X} \mathsf{\hat{v}}_j = \mathsf{\delta}_{ij} \,, \quad i,j = 1,\dots,N \,.
\end{equation}
and that
\begin{equation}\label{eqn:SVD_sfG}
    \mathsf{G}\mathsf{\hat{V}} = \mathsf{\hat{U}} \mathsf{\hat{\Lambda}} \,, \quad \mathsf{G}^\ast \mathsf{\hat{U}} = \mathsf{\hat{V}} \mathsf{\hat{\Lambda}}\,,
\end{equation}
where
\begin{equation}\label{eqn:svd_basis}
\mathsf{\hat{U}} = [\mathsf{\hat{u}}_1,\dots,\mathsf{\hat{u}}_N]\,,\quad \mathsf{\hat{V}} = [\mathsf{\hat{v}}_1,\dots,\mathsf{\hat{v}}_N]\,,\quad\text{and}\quad \mathsf{\hat{\Lambda}} = \diag(\hat{\lambda}_1,\dots,\hat{\lambda}_N)
\end{equation}
consist of the left and right singular vectors and all singular values.
Analogous to the continuous setting, we assume that $\hat{\lambda}_1 > \hat{\lambda}_2 > \dotsc > \hat{\lambda}_N >0$.
Given the weighted norm, $\mathsf{G}$ can be decomposed as
\begin{equation}\label{eqn:sfG_expansion}
    \mathsf{G} = \mathsf{\hat{U}} \mathsf{\hat{\Lambda}} \mathsf{\hat{V}}^\top \mathsf{\Pi}_\mathsf{X} \,,\quad\text{and}\quad \mathsf{G}^\ast = \mathsf{\hat{V}} \mathsf{\hat{\Lambda}} \mathsf{\hat{U}}^\top \mathsf{\Pi}_\mathsf{Y}\,.
\end{equation}

In the following two subsections, we review two ways to obtain optimality of bases: the Kolmogorov $n$-width and the Bayesian viewpoint.

\subsection{Kolmogorov $n$-width}\label{sec:Kolmogorov}
The Kolmogorov $n$-width quantifies the ``effective dimension'' of an operator.
It defines the number and the profile of modes needed to approximate a given operator to within a given tolerance.
We seek to apply this concept to the solution operator $\mathcal{S}$.
As originally defined in~\cite{Kolmogoroff1936UberDB}, the \emph{Kolmogorov $n$-width} (for given $n>0$) is
\begin{equation}\label{eqn:basis}
\mathcal{K}_n (\mathcal{S}) = \min_{v_1,\dots,v_n\in\mathcal{X}} \max_{f\in \mathcal{X}} \min_{c_1,\dots,c_n \in \Rb}\frac{\|\mathcal{S}(f)-\sum^n_{i=1}c_{i}\mathcal{S}(v_i)\|_{\mathcal{Y}}}{\|f\|_{\mathcal{X}}}\,,
\end{equation}
That is, we find among all possible sets of basis functions $\{v_i\}^n_{i=1}$ with $v_i\in\mathcal{X}$ the basis that minimizes the largest relative error between the output $\mathcal{S}(f)$ and its $n$-dimensional approximation $\sum_i c_i\mathcal{S}(v_i)$, where the $c_i$ are optimal projection coefficients.
The relationship to SVD is evident.
Indeed, $\mathcal{K}_n(\mathcal{S})$ is the $(n+1)$-st singular value of $\mathcal{S}$.
\begin{theorem}\label{thm:basis}
Under Assumption~\ref{assum:SR}, for a fixed $n>0$, we have $\mathcal{K}_n(\mathcal{S})=\lambda_{n+1}$, and the minimum in \eqref{eqn:basis} is achieved by the singular vectors $\{\hat{v}_i\}^n_{i=1}$ defined in \eqref{eqn:basisw}.
Equivalently, for any $u\in\mathcal{U}$, we have
\begin{equation}\label{eqn:approximationerrorbound}
\left\|u-\sum^n_{i=1}\lambda_i\hat{c}_{i}\hat{u}_i\right\|_{\mathcal{Y}}\leq \lambda_{n+1}\|\mathcal{L}(u)\|_{\mathcal{X}}\,,
\end{equation}
with $\hat{c}_{i}=\left\langle \mathcal{L}(u),\hat{v}_i\right\rangle_{\mathcal{X}}$.
\end{theorem}
\begin{proof}
Since $u\in\mathcal{U}$, we can find $f$ so that $u = \mathcal{S}(f)$ or equivalently $\mathcal{L}(u) = f$.
We will show that $\mathcal{K}_n(\mathcal{S})$ defined in~\eqref{eqn:basis} is equal to $\lambda_{n+1}$.
Since $\{\hat{v}_i\}^n_{i=1}$ are right-singular vectors of $\mathcal{S}$, and hence eigenvectors of $\mathcal{S}^*\mathcal{S}$ with eigenvalues $\{\lambda_i^2\}^n_{i=1}$, we have
\[
\left\|\mathcal{S}(f)-\sum^n_{i=1} \hat{c}_i \mathcal{S}(\hat{v}_i)\right\|^2_{\mathcal{Y}}
=\left\langle f-\sum^n_{i=1}\hat{c}_{i}\hat{v}_i,\mathcal{S}^*\mathcal{S}\left(f-\sum^n_{i=1}\hat{c}_i\hat{v}_i\right)\right\rangle_{\mathcal{X}}
= \sum_{i=n+1}^\infty \lambda_i^2 \hat{c}_i^2 \leq \lambda_{n+1}^2 \|f\|^2_{\mathcal{X}}\,.
\]
Equality in the final step is achieved by setting $f=\hat{v}_{n+1}$.
Therefore, by the definition of the Kolmogorov $n$-width, given $\hat{u}_i = \mathcal{S}(\hat{v}_i)$, we have
\[
\mathcal{K}_n(\mathcal{S})
\leq \max_{f\in \mathcal{X}}\min_{c_1,\dots,c_n \in \Rb}\frac{\|\mathcal{S}(f)-\sum^n_{i=1}\lambda_i c_{i}\hat{u}_i\|_{\mathcal{Y}}}{\|f\|_{\mathcal{X}}}
=\lambda_{n+1}  \,,
\]
so we have $\mathcal{K}_n(\mathcal{S})\leq\lambda_{n+1}$.
To show equality, it suffices to prove the $\{\hat{v}_i\}^n_{i=1}$ is the optimal basis, meaning, for any collection of functions $\{v_i\}^n_{i=1}$ with $v_i\in\mathcal{X}$:
\[
\max_{f\in \mathcal{X}}\min_{c_1,\dots,c_n \in \Rb}\frac{\|\mathcal{S}(f)-\sum^n_{i=1}c_{i}\mathcal{S}(v_i)\|_{\mathcal{Y}}}{\|f\|_{\mathcal{X}}}\geq \lambda_{n+1}\,.
\]
Indeed, it is always possible to choose a nontrivial function
$f \in \mathrm{span}\{\hat{v}_i\,,i=1\,,\cdots,n+1\} \, \cap \, \mathrm{span}\{v_i\,,i=1\,,\cdots,n\}^\perp$  such that the function $f_n = f-\sum_{i=1}^n c_i v_i$ has non-trivial components in $\mathrm{span}\{\hat{v}_i\,,i=1\,,\cdots,n+1\}$, making the quotient no less than $\lambda_{n+1}$.
~\eqref{eqn:approximationerrorbound} is obtained by substituting $u=\mathcal{S}(f)$ and $f=\mathcal{L}(u)$.
\end{proof}

In the discrete setting, the Kolmogorov $n$-width of the operator $\mathsf{G}$ is defined by
\begin{equation}\label{eqn:nwidth_dis}
\mathcal{K}_n(\mathsf{G}) = \min_{\mathsf{V}_n\in\mathbb{R}^{N\times n}}\max_{\mathsf{f}\in \mathbb{R}^\mathrm{N}}\min_{\mathsf{c}\in\Rb^n}\frac{\|\mathsf{G}\mathsf{f} - \mathsf{G} \mathsf{V}_n \mathsf{c} \|_\mathsf{Y}}{\|\mathsf{f}\|_\mathsf{X}}\,,
\end{equation}
where $\mathsf{c} = [c_1,\dots,c_n]^\top$. The analog of Theorem~\ref{thm:basis} is as follows.
\begin{theorem}
For $n \le N$, the Kolmogorov $n$-width~\eqref{eqn:nwidth_dis}
of the operator $\mathsf{G}$ defined in \eqref{eqn:Gproperty} is obtained from the SVD \eqref{eqn:SVD_1}, \eqref{eqn:SVD_sfG}, and  we have $\mathcal{K}_n(\mathsf{G}) = \hat{\lambda}_{n+1}$.
Moreover,
\[
\hat{\sfV}_n = \argmin_{\mathsf{V}_n\in\mathbb{R}^{N\times n}} \, \max_{\mathsf{f}\in \mathbb{R}^\mathrm{N}} \, \min_{\mathsf{c}\in\Rb^n} \, \frac{\|\mathsf{G}\mathsf{f} - \mathsf{G} \mathsf{V}_n \mathsf{c} \|_\mathsf{Y}}{\|\mathsf{f}\|_\mathsf{X}}\,,
\]
where $\hat{\sfV}_n$ collects the first $n$ columns of the matrix $\sfV$ from \eqref{eqn:svd_basis}. Finally, for any $\mathsf{f}\in\mathbb{R}^N$, defining $\mathsf{\hat{c}} = [\left\langle \mathsf{f},\mathsf{\hat{v}}_1\right\rangle_{\mathsf{X}},\dots, \left\langle \mathsf{f},\mathsf{\hat{v}}_n\right\rangle_{\mathsf{X}}]^\top$ and $\mathsf{\hat{\Lambda}}_n = \diag(\hat{\lambda}_1,\dots,\hat{\lambda}_n)$, we have
\[
\left\| \mathsf{u} - \mathsf{\hat{U}}_n \mathsf{\hat{\Lambda}}_n \mathsf{\hat{c}} \right\|_{\mathsf{Y}}
\leq \hat{\lambda}_{n+1} \|\mathsf{f}\|_{\mathsf{X}}\,.
\]
where $\hat{\sfU}_n$ collects the first $n$ columns in $\hat{\sfU}$ defined in \eqref{eqn:svd_basis}.
\end{theorem}
The proof tracks that of the previous result, so we omit it.
Special attention must be paid in the proof to the weight matrices $\mathsf{\Pi}_\mathsf{X}$ and $\mathsf{\Pi}_\mathsf{Y}$.

\subsection{A Bayesian perspective on numerical PDEs}\label{sec:Bayesian}
The perspective the Bayesian homogenization takes is the following. Suppose one does not have the full knowledge of $f$ but can nevertheless ``observe" the solution through a couple of experiments, can we reconstruct the full solution? To execute this idea, we consider only the discrete problem and assume $\mathsf{\Pi}_\mathsf{X}, \mathsf{\Pi}_\mathsf{Y}$ are both identity matrices for simplicity. We first place a prior on the source term $\mathsf{f}$ in~\eqref{eqn:eqn_dis} and consider the stochastic equation
\begin{equation}\label{eqn:MP1B}
\mathsf{L} \mathsf{u} = \mathsf{\xi}\,.
\end{equation}
Different prior knowledge of $\mathsf{f}$ would give different conditions on $\mathsf{\xi}$. Suppose one chooses to set the prior as a random field on the grid points with i.i.d. standard Gaussian components, then $\mathsf{\xi}$ would approximate white noise. In other words, we place a Gaussian prior on the source term. Due to the linearity of $\mathsf{L}$, the solution $\mathsf{u}$ is also a Gaussian field, and the covariance for $\mathsf{f}$ uniquely determines the covariance of $\mathsf{u}$. In particular, we have
\[
p(\mathsf{f}) \propto \exp\left(- \frac{1}{2} \mathsf{f}^\top \mathsf{f} \right)\;\Rightarrow\; p(\mathsf{u}) \propto \exp\left(- \frac{1}{2} \mathsf{u}^\top \mathsf{C}^{-1} \mathsf{u} \right)\,, \quad \mbox{where} \; \mathsf{C} = (\sfC_{ij})_{N\times N} = \mathsf{G} \mathsf{G}^\top
\]
is the covariance.

Suppose we ``observe'' the solution $n$ times, with a linearly independent set of test vector $\mathsf{m}_i \in \mathbb{R}^N$. Then we have knowledge of
\begin{equation}\label{def:Psi}
\mathsf{\Psi} = \mathsf{M}^\top \mathsf{u} \,,
\end{equation}
where $\mathsf{M} = [\mathsf{m}_1, \dots, \mathsf{m}_n] \in \mathbb{R}^{N\times n}$.
Conditioned on these observations, we can expect to have a better understanding of the solution. To derive the distribution conditioned on the observed data, we use the Bayes' formula. We assume the data likelihood to have the following form
\[
p(\mathsf{\Psi} | \mathsf{u}) \propto \exp\left( -\frac{1}{2\delta} (\mathsf{\Psi} - \mathsf{M}^\top \mathsf{u})^\top (\mathsf{\Psi} - \mathsf{M}^\top \mathsf{u}) \right)
\]
where $\delta>0$ is the variance. Bayes' formula tells us
\begin{equation}
p(\mathsf{u} | \mathsf{\Psi}) \propto p(\mathsf{\Psi} | \mathsf{u}) p(\mathsf{u}) \\
    \propto \exp\left(  -\frac{1}{2} ( \mathsf{u} - \mathsf{\mu} )^\top \mathsf{\Sigma}^{-1} ( \mathsf{u} - \mathsf{\mu} ) \right)
\end{equation}
where the mean and covariance are respectively:
\begin{align}
    \mathsf{\mu} &= \mathsf{\delta}^{-1} \mathsf{\Sigma} \mathsf{M} \mathsf{\Psi} = \mathsf{K}^\top ( \mathsf{\Theta} + \delta \mathsf{I}_n )^{-1} \mathsf{\Psi}, \\
    \mathsf{\Sigma} &= ( \mathsf{C}^{-1} + \delta^{-1} \mathsf{M} \mathsf{M}^\top )^{-1}
    = \mathsf{C} - \mathsf{K}^\top (\mathsf{\Theta} + \delta \mathsf{I}_n )^{-1} \mathsf{K}\,, \label{eqn:sigma_formula}
\end{align}
where
\begin{equation}\label{def:sfK_sfTheta}
\mathsf{K} = \mathsf{M}^\top \mathsf{C} \in \mathbb{R}^{n\times N}\,,\quad\mathsf{\Theta} = \mathsf{M}^\top \mathsf{C} \mathsf{M} \in \mathbb{R}^{n\times n}\,,
\end{equation}
and we applied the Sherman-Morrison-Woodbury identity in the second equality in~\eqref{eqn:sigma_formula}.

The value of $\delta$ essentially represents the ``trust'' we place on the reading of the observation. Suppose the reading is fully accurate (no observation error), we send $\delta\to0$ and obtain
\begin{equation}\label{eqn:mean_variance_delta0}
\begin{aligned}
    \mathsf{\Sigma}
    \xrightarrow{\delta\to0}
    \Sigma_{\delta=0} & = \mathsf{C} - \mathsf{K}^\top \mathsf{\Theta}^{-1} \mathsf{K}\,, \\
    \mathsf{\mu} \xrightarrow{\delta\to0}
    \mu_{\delta=0} & = \mathsf{K}^\top \mathsf{\Theta}^{-1} \mathsf{\Psi}
    \triangleq \mathsf{W} \cdot \mathsf{\Psi},
\end{aligned}
\end{equation}
with $\mathsf{W} = \mathsf{K}^\top \mathsf{\Theta}^{-1}$.
This means that if the given source $\mathsf{f}$ is a i.i.d. Gaussian random field and we make observation $\Psi$ and  fully trust the data, we already know that the solution should be $\mu$ on average, with confidence interval described by $\mathsf{\Sigma}$.

To make this argument more concrete, we cite the following theorem.
\begin{theorem}[\cite{doi:10.1137/140974596}, Theorems 3.5, 5.1] \label{thm:OW_dis} Let $\mathsf{u}$ be the solution of \eqref{eqn:MP1B} for a Gaussian random field $\mathsf{f}$ with i.i.d. Gaussian components, then $\mathsf{u}$ conditioned on the value of $\mathsf{\Psi}$ is a Gaussian random variable with explicitly computable mean and variance defined in~\eqref{eqn:mean_variance_delta0}.
Moreover, the reconstructed solution $\hat{\mathsf{u}}=\mathsf{W} \cdot \mathsf{\Psi}$ satisfies
\begin{equation}\label{eqn:errorboundB_dis}
\left\| \mathsf{u} - \mathsf{W} \cdot \mathsf{\Psi} \right\|_2 \leq \sqrt{\mathrm{Tr}(\mathsf{\Sigma})} \|\mathsf{f}\|_2 \,.
\end{equation}
\end{theorem}
\begin{proof}
The inequality follows directly from the definition of $\mathsf{\Sigma}$:
\begin{equation}
    \left\| \mathsf{u} - \mathsf{W} \cdot \mathsf{\Psi} \right\|_2 = \left\| (\mathsf{G} - \mathsf{K}^\top \mathsf{\Theta}^{-1} \mathsf{M}^\top  \mathsf{G}) \mathsf{f} \right\|_2 \leq \left\| \mathsf{G} - \mathsf{K}^\top \mathsf{\Theta}^{-1} \mathsf{M}^\top \mathsf{G} \right\|_\mathrm{F}  \left\|\mathsf{f} \right\|_2 = \sqrt{\mathrm{Tr}(\mathsf{\Sigma})} \|\mathsf{f}\|_2 \,,
\end{equation}
In the last equality, we use the following computations
\begin{equation}
    \begin{aligned}
    \left\| \mathsf{G} - \mathsf{K}^\top \mathsf{\Theta}^{-1} \mathsf{M}^\top \mathsf{G} \right\|_\mathrm{F}^2
    =& \mathrm{Tr}\left[ (\mathsf{G} - \mathsf{K}^\top \mathsf{\Theta}^{-1} \mathsf{M}^\top \mathsf{G})^\top (\mathsf{G} - \mathsf{K}^\top \mathsf{\Theta}^{-1} \mathsf{M}^\top \mathsf{G}) \right] \\
    =& \mathrm{Tr}\left[ \mathsf{G} \mathsf{G}^\top \right] - \mathrm{Tr}\left[ \mathsf{K}^\top \mathsf{\Theta}^{-1} \mathsf{K} \right]
    = {\mathrm{Tr}(\Sigma)},
    \end{aligned}
\end{equation}
where in the second equality, we used the fact that
\[
\mathrm{Tr}\left[(\mathsf{K}^\top \mathsf{\Theta}^{-1} \mathsf{M}^\top \mathsf{G})^\top \mathsf{G}\right]= \mathrm{Tr}\left[\mathsf{G}^\top (\mathsf{K}^\top \mathsf{\Theta}^{-1} \mathsf{M}^\top \mathsf{G})\right]=\mathrm{Tr}\left[\sfK\sfTheta^{-1}\sfK\right]
\]
and that
\[
\mathrm{Tr}\left[(\mathsf{K}^\top \mathsf{\Theta}^{-1} \mathsf{M}^\top \mathsf{G})^\top (\mathsf{K}^\top \mathsf{\Theta}^{-1} \mathsf{M}^\top \mathsf{G})\right] = \mathrm{Tr}\left[\mathsf{K}^\top \mathsf{\Theta}^{-1} \mathsf{M}^\top \mathsf{G} \mathsf{G}^\top \mathsf{M}\mathsf{\Theta}^{-1}\mathsf{K}\right] =\mathrm{Tr}\left[\sfK\sfTheta^{-1}\sfK\right]\,.
\]

\end{proof}
This beautiful viewpoint translates the solving of a PDE to the mapping between two Gaussian fields. Furthermore, the equation~\eqref{eqn:errorboundB_dis} suggests that different sets of ``observations" can extract different amounts of information. The ``best" observations should be the ones that minimize the trace of the covariance. In the most extreme case, when  $\mathrm{Tr}(\mathsf{\Sigma}) \approx 0$, the system is deterministic, and there are no extra uncertainties, meaning that the observation tell the full story of $\mathsf{u}$. We would like to find a set of observables $\mathsf{\Psi}$ that reveal as much information as possible about $\mathsf{u}$
That is, given a budget of $n$  observations, we seek the $n$ test functions that minimize $\mathrm{Tr}(\mathsf{\Sigma})$. These optimal test functions are the columns of $\mathsf{\hat{M}}$ defined by
\begin{equation}\label{eqn:optimal_dis}
\mathsf{\hat{M}}
=\argmin\limits_{\substack{\mathsf{M}\in\mathbb{R}^{N\times n} \\ \mathrm{rank}(\mathsf{M})=n}} \mathrm{Tr}\left[\mathsf{C} - \mathsf{K}^\top \mathsf{\Theta}^{-1} \mathsf{K}\right]
=\argmax\limits_{\substack{\mathsf{M}\in\mathbb{R}^{N\times n} \\ \mathrm{rank}(\mathsf{M})=n}} \mathrm{Tr}\left[\mathsf{K}^\top \mathsf{\Theta}^{-1} \mathsf{K}\right]\,.
\end{equation}

\subsection{Connection}

While the two definitions of optimality are motivated differently, we argue in this section that they coincide. Namely, the Kolmogorov $n$-width optimization strategy and the variance minimization strategy essentially yield the same sets of basis functions.

\begin{theorem}\label{thm:equivalence}
The optimization problems \eqref{eqn:nwidth_dis} and~\eqref{eqn:optimal_dis} are equivalent in the following sense:
\begin{itemize}
\setlength{\itemindent}{2em}
    \item[]{$\bullet$} The left singular vectors defined in~\Cref{thm:OW_dis} provide one set of optimal ``observations" for~\eqref{eqn:optimal_dis}, namely:
\begin{equation}
\mathsf{\hat{M}} = \mathsf{\hat{U}}_n = \left[ \hat{\mathsf{u}}_1,\dots, \hat{\mathsf{u}}_n \right]\,,
\end{equation}
solves \eqref{eqn:optimal_dis}. In this setting, recall \eqref{def:Psi} and \eqref{eqn:mean_variance_delta0}, we also have
\begin{equation}\label{eqn:opt_measure_func}
\mathsf{w}_i = \mathsf{\hat{u}}_i, \quad \Psi_i = \hat{\lambda}_i \langle \mathsf{f}, \mathsf{\hat{v}}_i \rangle \,.
\end{equation}
    \item[]{$\bullet$} The optimal observation provides the left singular vectors for $\mathsf{G}$. Namely, for the solution $\mathsf{\hat{M}}$ of~\eqref{eqn:optimal_dis}, we have
\begin{equation}\label{eqn:optimal_M_cond}
    \mathrm{Span}(\hat{\sfU}_n) = \mathrm{Span}(\mathsf{\hat{M}})\,.
\end{equation}
\end{itemize}
\end{theorem}

\begin{proof}
Recalling the definition of $\mathsf{K}$ and $\mathsf{\Theta}$ in~\eqref{def:sfK_sfTheta}, the objective function in~\eqref{eqn:optimal_dis} can be written as
\[
\mathrm{Tr}\left[ \mathsf{K}^\top \mathsf{\Theta}^{-1} \mathsf{K} \right]
= \mathrm{Tr}\left[ \mathsf{\hat{\Lambda}}^2 \mathsf{\hat{U}}^\top \mathsf{M} (\mathsf{M}^\top \mathsf{\hat{U}} \mathsf{\hat{\Lambda}}^2 \mathsf{\hat{U}}^\top \mathsf{M})^{-1} \mathsf{M}^\top \mathsf{\hat{U}} \mathsf{\hat{\Lambda}}^2 \right]
= \mathrm{Tr} \left[ \mathsf{\hat{\Lambda}} \mathsf{P}_{\mathsf{\hat{\Lambda}}\mathsf{\hat{U}}^\top\mathsf{M}} \mathsf{\hat{\Lambda}} \right] \,.
\]
where $\mathsf{P}_{\mathsf{\hat{\Lambda}}\mathsf{\hat{U}}^\top\mathsf{M}}$ are the orthogonal projection onto $\mathsf{\hat{\Lambda}}\mathsf{\hat{U}}^\top\mathsf{M}$. Hence, the optimal test function problem~\eqref{eqn:optimal_dis} is equivalent to
\begin{equation}\label{eqn:bayes_trace}
\mathsf{\hat{M}}
= \argmax\limits_{\substack{\mathsf{M}\in\mathbb{R}^{N\times n} \\ \mathrm{rank}(\mathsf{M})=n}} \mathrm{Tr} \left[ \mathsf{\hat{\Lambda}} \mathsf{P}_{\mathsf{\hat{\Lambda}}\mathsf{\hat{U}}^\top\mathsf{M}} \mathsf{\hat{\Lambda}} \right] \,.
\end{equation}

We write the projection matrix as:
\begin{equation}
    \mathsf{P}_{\mathsf{\hat{\Lambda}}\mathsf{\hat{U}}^\top\mathsf{M}}
    = \mathsf{R} \mathsf{R}^\top \,.
\end{equation}
where $\mathsf{R} = [\mathsf{r}_1,\dots,\mathsf{r}_N] \in \mathbb{R}^{N\times n}$ is a unitary matrix satisfying $\mathsf{R}^\top \mathsf{R} = \mathsf{I}_n$. By substituting  into~\eqref{eqn:bayes_trace}, we obtain
\begin{equation}\label{eqn:bayes_trace_simp}
\mathrm{Tr} \left[ \mathsf{\hat{\Lambda}} \mathsf{P}_{\mathsf{\hat{\Lambda}}\mathsf{\hat{U}}^\top\mathsf{M}} \mathsf{\hat{\Lambda}} \right]
= \sum_{i=1}^n \mathsf{r}_i^{\top} \mathsf{\hat{\Lambda}}^2 \mathsf{r}_i
= \sum_{i=1}^n \sum_{j=1}^N \hat{\lambda}_j^2 \mathsf{r}^2_{ij}
= \sum_{j=1}^N \hat{\lambda}_j^2 a_j  \,,
\end{equation}
where $\mathsf{r}_i := [r_{i1},\dots,r_{iN}]^\top$ and we define $a_j = \sum_{i=1}^n r_{ij}^2$.
The optimal measurement is achieved if we can choose a set of orthonormal basis $\{\mathsf{r}_i\}_{i=1}^N$ such that the weighted sum of $a_j$ is maximized. Since $\mathsf{R}$ is unitary, we obtain the following restrictions on $a_j$
\[
0 \leq a_j \leq 1 \,, \text{ for } i = 1,\dots, N \quad \text{ and }
\sum_{i=1}^N a_j = n \,.
\]
Noting that $\hat{\lambda}_1>\dots>\hat{\lambda}_N>0$,~\eqref{eqn:bayes_trace_simp} is maximized if and only if the unitary matrix $\mathsf{R}$ has $a_j = 1$ for $1 \leq j \leq n$ and $a_j = 0$ for $n+1 \leq j \leq N$.
That is, $\mathsf{R}$ has the following form
\begin{equation}
    \mathsf{R} =
    \begin{bmatrix}
    \mathsf{R_0} \\
    \mathsf{0}_{(N-n)\times n}
    \end{bmatrix} \,,
\end{equation}
where $\mathsf{R_0} \in \mathbb{R}^{n\times n}$ and $\mathsf{R_0}^\top \mathsf{R_0} = \mathsf{I}_n$.
Hence, for any an optimal observation matrix $\mathsf{\hat{M}}$, the projection matrix $\mathsf{P}_{\mathsf{\hat{\Lambda}}\mathsf{\hat{U}}^\top\mathsf{\hat{M}}}$ satisfies
\[\mathsf{P}_{\mathsf{\hat{\Lambda}}\mathsf{\hat{U}}^\top\mathsf{\hat{M}}} = \mathsf{D} := \diag(\underbrace{1,\dots, 1}_{n}, \underbrace{0, \dots,0}_{N-n})\,.
\]
We thus have $\mathrm{Span}(\mathsf{\hat{\Lambda}}\mathsf{\hat{U}}^\top\mathsf{\hat{M}}) = \mathrm{Span}(\mathsf{D}_{:,1:n})$ and it follows the optimality condition~\eqref{eqn:optimal_M_cond} \[
\mathrm{Span}(\mathsf{\hat{M}}) = \mathrm{Span}(\mathsf{\hat{U}}\mathsf{\hat{\Lambda}}^{-1}\mathsf{D}_{:,1:n}) = \mathrm{Span}(\mathsf{\hat{U}}_n)\,.
\]

On the other hand, selecting $\mathsf{\hat{\Lambda}}\mathsf{\hat{U}}^\top\mathsf{\hat{M}} = \mathsf{D}_{:,1:n}$, we get $\mathsf{\hat{M}} = \mathsf{\hat{U}}_n\mathsf{\hat{\Lambda}}^{-1}_{:,1:n}$. Renormalizing $\mathsf{\hat{M}}$, we obtain $\mathsf{\hat{M}} = \mathsf{\hat{U}}_n$.
\end{proof}

\section{Kolmogorov $n$-width and the randomized SVD}\label{sec:kol_dis}

The beautiful relationship among the Kolmogorov $n$-width, the SVD of the operator $\mathcal{S}$, and the optimality condition in the Bayesian sense allows us to use the fast SVD solvers developed in recent years to find optimal basis functions. We study  a numerical setup in this subsection that  integrates the randomized SVD algorithm.

We recall that the $\mathcal{S}$ operator can be represented as the discrete Green's matrix $\mathsf{G}$, which can be regarded as an operator mapping from $(\mathbb{R}^N, \|\cdot\|_\mathsf{X})$ to $(\mathbb{R}^N, \|\cdot\|_\mathsf{Y})$. Its adjoint operator $\mathsf{G}^\ast$ maps $(\mathbb{R}^N, \|\cdot\|_\mathsf{Y})$ back to $(\mathbb{R}^N, \|\cdot\|_\mathsf{X})$. The SVD of $\mathsf{G}$ satisfies
\[
\mathsf{G} \mathsf{\hat{V}} = \mathsf{\hat{U}} \mathsf{\hat{\Lambda}} \,, \quad \mathsf{G}^\ast \mathsf{\hat{U}} = \mathsf{\hat{V}} \mathsf{\hat{\Lambda}} \,.
\]
Recalling~\eqref{eqn:sfG_expansion}, the basis $\mathsf{\hat{V}}$ can be found via standard eigendecomposition algorithms:
\begin{equation}\label{eqn:G_ast2}
 \mathsf{G}^\ast \mathsf{G}=\mathsf{\hat{V}}\Sigma\mathsf{\hat{V}}^{-1}\,,\quad\text{and}\quad
\mathsf{\Pi}_\mathsf{X}^{-1/2}\mathsf{G}^\top\mathsf{\Pi}_{\mathsf{Y}}\mathsf{G}\mathsf{\Pi}_\mathsf{X}^{-1/2}
=(\mathsf{\Pi}_\mathsf{X}^{1/2}\mathsf{\hat{V}})\Sigma(\mathsf{\Pi}_\mathsf{X}^{1/2}\mathsf{\hat{V}})^\top\,.
\end{equation}
We note that $\mathsf{\Pi}_\mathsf{X}^{1/2}\mathsf{\hat{V}}$ is orthogonal.

Naively, one can prepare $\mathsf{G}$ ahead of time, computing $\mathsf{G}^\ast \mathsf{G}$ or $\mathsf{\Pi}_\mathsf{X}^{-1/2}\mathsf{G}^\top\mathsf{\Pi}_{\mathsf{Y}}\mathsf{G}\mathsf{\Pi}_\mathsf{X}^{-1/2}$ as in~\eqref{eqn:G_ast2} before performing eigenvalue decomposition. However, $\mathsf{G}$ is typically a very large matrix, and preparing it ahead of time may require a large overhead. The procedure is roughly as follows.
\begin{itemize}
    \item Preparation of $\mathsf{G}$: This amounts to finding the LU decomposition for $\sfL$ and solving each columns of $\mathsf{G}$. The cost is $O(N^{2.3}+N^3)$ (We use Coppersmith–Winograd algorithm for measuring the complexity of LU decomposition. Faster solvers are available when $\sfL$ is sparse, but we omit further exploration of this point);
    \item Computation of $\mathsf{G}^\ast \mathsf{G}$ or $\mathsf{\Pi}_\mathsf{X}^{-1/2}\mathsf{G}^\top\mathsf{\Pi}_{\mathsf{Y}}\mathsf{G}\mathsf{\Pi}_\mathsf{X}^{-1/2}$: $O(N^3)$ by~\eqref{eqn:G_ast2} assuming the weight matrices are sparse;
    \item Eigenvalue decomposition of $\mathsf{G}^\ast \mathsf{G}$ or $\mathsf{\Pi}_\mathsf{X}^{-1/2}\mathsf{G}^\top\mathsf{\Pi}_{\mathsf{Y}}\mathsf{G}\mathsf{\Pi}_\mathsf{X}^{-1/2}$ costs $O(N^3)$.
\end{itemize}

We note that the largest overhead computational cost comes from the preparation of $\mathsf{G}$. It would be ideal if the calculation can be done on-the-fly, without the need to prepare for the entire matrix ahead of time. According to \eqref{eqn:G_ast2}, performing eigenvalue decomposition for $\mathsf{G}^\ast\mathsf{G}$ can be translated to performing SVD for $\mathsf{\Pi}_{\mathsf{Y}}^{1/2}\mathsf{G}\mathsf{\Pi}_{\mathsf{X}}^{-1/2}$. Thus, we propose to run randomized SVD for $\mathsf{\Pi}_{\mathsf{Y}}^{1/2}\mathsf{G}\mathsf{\Pi}_{\mathsf{X}}^{-1/2}$.

\subsection{Computation of optimal basis using Randomized SVD}\label{sec:rsvd}
Randomized SVD, introduced in~\cite{doi:10.1073/pnas.0709640104,WOOLFE2008335,doi:10.1137/090771806} shows that a low rank matrix $\mathsf{A}$ with rank $\sim r$, if multiplied by a set of $r+p$ i.i.d. randomly generated Gaussian vectors, with $p\sim 2$, has a high chance of recovering the singular vectors/values of $\mathsf{A}$. Algorithm~\ref{alg:rsvd} summarizes the approach.

\begin{algorithm}[h]
\caption{\textbf{Randomized SVD for $\mathsf{A}$}}\label{alg:rsvd}
\begin{algorithmic}
\State Input: Matrix $\mathsf{A}\in\Rb^{N\times N}$; estimated rank: $r$; oversampling number $p$.

\State \textbf{Randomized range finder for $\mathsf{Q}=\text{range}(\mathsf{A})$}
\State \space1. Generate random samples: Draw an $\NN\times (r+p)$ Gaussian random matrix $\mathsf{P}$. Denote $\mathsf{P}_{:,i}\in\Rb^{\MM}$ the $i$-th column of $\mathsf{P}$.
\State \space2. Compute $\mathsf{A}\mathsf{P}$.

\State \space3. Find the range: Perform QR-factorization $[\mathsf{Q},\sim]=\text{QR}(\mathsf{C})$.

\State \textbf{Compute $\mathsf{B} = \mathsf{Q}^\top \mathsf{A}$:}
\State \textbf{Perform SVD for $\mathsf{B}$}
\begin{itemize}
    \item[--] Call $\mathsf{B} = \widetilde{\mathsf{U}}{\mathsf{S}}\widetilde{\mathsf{V}}^\top = \sum_is_i\widetilde{\mathsf{u}}_i\widetilde{\mathsf{v}}^\top_i$
\end{itemize}

\State \textbf{Output:}$\{\widetilde{\mathsf{v}}_{i}\}^r_{i=1}$.
\end{algorithmic}
\end{algorithm}

Noting that the Algorithm~\ref{alg:rsvd} does not require  full access to the matrix $\mathsf{A}$, but rather its left matrix-vector product with $\mathsf{p}$, a random Gaussian i.i.d. vector, and its right product with the matrix  $\mathsf{q}$. In our context, we thus avoid the cost of preparing $\mathsf{G}$. To utilize this algorithm, we note that in our situation,
\[
\mathsf{A}=\mathsf{\Pi}_{\mathsf{Y}}^{1/2}\mathsf{G}\mathsf{\Pi}_{\mathsf{X}}^{-1/2}=\mathsf{\Pi}_{\mathsf{Y}}^{1/2}\mathsf{L}^{-1}\mathsf{\Pi}_{\mathsf{X}}^{-1/2}\,,
\]
so the main tasks are to understand how to compute $\mathsf{A}\mathsf{p}=\mathsf{\Pi}_{\mathsf{Y}}^{1/2}\mathsf{L}^{-1}\mathsf{\Pi}_{\mathsf{X}}^{-1/2}\mathsf{p}$ and $\mathsf{q}^\top\mathsf{A}=\mathsf{q}^\top\mathsf{\Pi}_{\mathsf{Y}}^{1/2}\mathsf{L}^{-1}\mathsf{\Pi}_\mathsf{X}^{-1/2}$. Details are as follows.
\begin{itemize}
    \item Defining $\mathsf{c} = \mathsf{L}^{-1}\mathsf{\Pi}_{\mathsf{X}}^{-1/2}\mathsf{p}$, we have
    \[
    \mathsf{L}\mathsf{c} = \mathsf{\Pi}_{\mathsf{X}}^{-1/2}\mathsf{p}\,.
    \]
    This amounts to solving the equation on the discrete setting using the source term  $\mathsf{\Pi}_{\mathsf{X}}^{-1/2}\mathsf{p}$. Then $\mathsf{A}\mathsf{p}=\mathsf{\Pi}_{\mathsf{Y}}^{1/2}\mathsf{c}$.
    \item Let $\mathsf{d} = \mathsf{q}^\top\mathsf{\Pi}_{\mathsf{Y}}^{1/2}\mathsf{L}^{-1}\mathsf{\Pi}_{\mathsf{X}}^{-1/2}$, we have:
    \[
    \mathsf{L}^{-\top}\mathsf{\Pi}_{\mathsf{Y}}^{1/2}\mathsf{q}=\left(\mathsf{d}\mathsf{\Pi}_{\mathsf{X}}^{1/2}\right)^\top\,.
    \]
    By solving $\mathsf{L}^\top\tilde{\mathsf{d}}^\top = \mathsf{\Pi}_{\mathsf{Y}}^{1/2}\mathsf{q}$, we obtain $\mathsf{d}$ by setting $\mathsf{d} = \tilde{\mathsf{d}}\mathsf{\Pi}_{\mathsf{X}}^{-1/2}$. This process amounts to inverting the transpose of the discrete system on a weighted source.
\end{itemize}

With these calculations incorporated to Algorithm~\ref{alg:rsvd}, we have an algorithm for finding the optimal PDE basis, which we summarize in Algorithm~\ref{ALG1}. The computational costs are as follows.
\begin{itemize}
    \item Computation of $\mathsf{\Pi}^{1/2}_{\mathsf{Y}}\mathsf{G}\mathsf{\Pi}_{\mathsf{X}}^{-1/2}\mathsf{P}$:
    \begin{enumerate}[wide, labelwidth=!, labelindent=0pt]
        \item Computation of $\mathsf{\Pi}_{\mathsf{X}}^{-1/2}\mathsf{P}$ costs $O((r+p)N^2)$;
        \item Solve for $\widetilde{\mathsf{C}}\in\mathbb{R}^{N\times (r+p)}$ according to \eqref{eqn:Csolver}: There are $r+p$ PDE solves, with each of cost $O(N^3)$. The total cost is $O(N^{2.3}+(r+p)N^2)$;
    \end{enumerate}
    \item QR-factorization of $\mathsf{C}$ costs $O((r+p)^2N)$.
    \item Computation of $\mathsf{Q}^\top \mathsf{\Pi}^{1/2}_{\mathsf{Y}}\mathsf{G}\mathsf{\Pi}_{\mathsf{X}}^{-1/2}$:
    \begin{enumerate}[wide, labelwidth=!, labelindent=0pt]
        \item Solve for $\mathsf{E}\in\mathbb{R}^{N\times (r+p)}$ according to \eqref{eqn:Dsolver}: There are $r+p$ PDE solves, costing $O(N^{2.3}+(r+p)N^2)$;
        \item Computation of $\mathsf{E}^\top \mathsf{\Pi}_{\mathsf{X}}^{-1/2}$ costs $O((r+p)N^2)$;
    \end{enumerate}
    \item Perform SVD for $\mathsf{B}$ costs $O((r+p)N^2)$.
\end{itemize}
The total cost is  $O(N^{2.3}+(r+p)N^2)$. In comparison with looking directly for an optimal basis by preparing  the whole Green's matrix, this solver significantly reduce the numerical cost. We should also note that the solver is \emph{online} in the sense that we can gradually increase $r$. Suppose by increase $r$ to $r+1$, the newly found $\mathsf{q}$ lies within $\epsilon$ accuracy of the space spanned by the old $\mathsf{Q}$, the singular value decomposition saturates with an accuracy of $\epsilon$.

\begin{algorithm}[h]
\caption{\textbf{Optimal basis finder}}\label{ALG1}
\begin{algorithmic}
\State Input: The discrete linear operators $\mathsf{L}$; $\mathsf{\Pi}_\mathsf{X};\mathsf{\Pi}_\mathsf{Y}$; estimated rank $r$; oversampling $p$;

\State \textbf{Randomized range finder for $\mathsf{Q}=\text{range}\left(\mathsf{\Pi}^{1/2}_{\mathsf{Y}}\mathsf{G}\mathsf{\Pi}_{\mathsf{X}}^{-1/2}\right)$}
\State \space1. Generate random samples: Draw an $N\times (r+p)$ Gaussian random matrix $\mathsf{P}$. Denote $\mathsf{P}_{:,i}\in\Rb^N$ the $i$-th column of $\mathsf{P}$.
\State \space2. Compute $\mathsf{\Pi}^{1/2}_{\mathsf{Y}}\mathsf{G}\mathsf{\Pi}_{\mathsf{X}}^{-1/2}\mathsf{P}$:
\begin{itemize}
\item[--] Solve for $\widetilde{\mathsf{C}}\in\mathbb{R}^{N\times (r+p)}$ so that
\begin{equation}\label{eqn:Csolver}
\mathsf{L}
\widetilde{\mathsf{C}} = \mathsf{\Pi}_{\mathsf{X}}^{-1/2}\mathsf{P}\,;
\end{equation}
\item[--] Set $\mathsf{C}=\mathsf{\Pi}_{\mathsf{Y}}^{1/2}\widetilde{\mathsf{C}}$.
\end{itemize}
\State \space3. Find the range: Perform QR-factorization $[\mathsf{Q},\sim]=\text{QR}(\mathsf{C})$.

\State \textbf{Compute $\mathsf{B} = \mathsf{Q}^\top \Pi^{1/2}_{\mathsf{Y}}\mathsf{G}\mathsf{\Pi}_{\mathsf{X}}^{-1/2}$:}
\begin{itemize}
\item[--] Define $\widetilde{\mathsf{Q}} = \mathsf{\Pi}_{\mathsf{Y}}^{1/2}\mathsf{Q}$;
    \item[--] Solve
    \begin{equation}\label{eqn:Dsolver}
    \mathsf{L}^\top\mathsf{E}=\widetilde{\mathsf{Q}}\,,
    \end{equation}
    \item[--] Define $\mathsf{B} =\mathsf{E}^\top \mathsf{\Pi}_{\mathsf{X}}^{-1/2}$.
\end{itemize}

\State \textbf{Perform SVD for $\mathsf{B}$}
\begin{itemize}
    \item[--] Call $\mathsf{B} = \widetilde{\mathsf{U}}{\mathsf{\Lambda}}\widetilde{\mathsf{V}}^\top= \sum_i\hat{\lambda}_i\widetilde{\mathsf{u}}_i\widetilde{\mathsf{v}}^\top_i$.
\end{itemize}
\State \textbf{Calculate optimal basis}
\[
\mathsf{\hat{v}}_{i}=\mathsf{\Pi}_\mathsf{X}^{-1/2}\widetilde{\mathsf{v}}_i,\quad \mathsf{\hat{u}}_{i}=\frac{1}{\hat{\lambda}_i}\mathsf{G}\mathsf{\hat{v}}_{i}\,,\quad\text{for}\quad i=1\,,\dots, r\,.
\]

\State \textbf{Output:} $\{\hat{\lambda}_i,\mathsf{\hat{v}}_{i},\mathsf{\hat{u}}_i\}^r_{i=1}$.
\end{algorithmic}
\end{algorithm}

\section{Optimal basis for nonlinear PDEs}\label{sec:nonlinear}
One critique of Kolmogorov $n$-width or the concept of Bayesian PDE is that the concepts are valid only for linear equations. Indeed only in the linear case can one argue about the solution space, and only in the linear case can the Gaussian field of the source translates to a Gaussian field of the solution.

What would happen for nonlinear PDEs? We argue in this section that if a nonlinear PDE has a substantial linear component, the optimal basis functions still form a good venue to conduct iterations. To start, we write out equation as
\begin{equation}\label{eqn:MP1r}
\mathcal{L} (u(x))+\mathcal{N}(u(x))=f(x),\quad x\in\Omega
\end{equation}
where $\mathcal{L}:\mathcal{Y}\to\mathcal{X}$ and $\mathcal{N}:\mathcal{Y}\to\mathcal{X}$ represent the linear and nonlinear components respectively.

\begin{remark}
Many PDE examples can be written in the form of~\eqref{eqn:MP1r}. This includes:
\begin{itemize}
\setlength{\itemindent}{2em}
    \item[]{$\bullet$} The semilinear elliptic equation:
\[
-\nabla \cdot (\kappa(x) \nabla u(x)) + u^3(x)= f(x),\quad x\in\Omega\,.
\]
The equation describes nonlinear diffusion generated by nonlinear sources~\cite{joseph1973quasilinear}, and the gravitational equilibrium of stars~\cite{chandrasekhar1957introduction,lions1982existence}. In this case the linear and nonlinear term are respectively
\begin{equation}
\mathcal{L} u = -\nabla \cdot (\kappa(x) \nabla u)\,, \quad \mathcal{N}(u) = u^3 \,.
\end{equation}
    \item[]{$\bullet$} The semilinear radiative transport equation that governs the dynamics of the distribution of photons $u(x,v)$ on the phase space:
\[
v\cdot \nabla_x u + (\sigma_\mathrm{t}(x) + \sigma_\mathrm{b}(x) \langle u \rangle (x)) u =  \sigma_\mathrm{s}(x) \int_{\mathbb{V}} u(x,v') \rmd \mu(v') + f(x,v)\,, \quad (x,v)\in\Omega\times\mathbb{V},
\]
where $x\in\Omega$ is the physical domain variable and $v\in\mathbb{V}$ is the velocity domain variable. The equation characterizes the distribution of photon particles upon scattering with the underlying media through $\sigma_{\mathrm{s}}$ and absorbed through $\sigma_{\mathrm{t}}$. In this case, the linear and the nonlinear components are
\begin{equation}
    \begin{aligned}
    &\mathcal{L} u = v\cdot \nabla_x u + \sigma_\mathrm{t}(x) u
-  \sigma_\mathrm{s}(x) \int_{\mathbb{V}} u(x,v') \rmd \mu(v')\,, \\
    &\mathcal{N}(u) = \sigma_\mathrm{b}(x) \langle u \rangle (x) u \,,
    \end{aligned}
\end{equation}
where the nonlinear component represents the two-photon scattering. See~\cite{BaDeEn2009:handbook,ReZh2021:unique,LaReZh2022:inverse}.
\end{itemize}
We study both examples numerically  in~\Cref{sec:numerical}.
\end{remark}

One way to view this nonlinear equation is to set $f(x)-\mathcal{N}(u(x))$ as the source term, and find the corresponding solution by inverting the linear operator $\mathcal{L}$. In doing so, we can still project the solutions in the space spanned by the leading spectrum components. Recall the definition \eqref{eqn:basisw} of $\mathcal{S}$, the solution operator associated with $\mathcal{L}$, and denote $\{\lambda_i,\hat{u}_i,\hat{v}_i\}$ as the singular value-function triplets.
Then, if the solution is projected on this space, we have the following a-priori error estimate:
\begin{theorem}\label{thm:nonlinearcase} Under Assumption \ref{assum:SR}, for any $f\in\mathcal{X}$, suppose that the solution $u$ to \eqref{eqn:MP1r} satisfies $u\in\mathcal{Y}$ and $\mathcal{N}(u(x))\in\mathcal{X}$, then for any $n>0$, we have
\begin{equation}\label{eqn:basis_nonlin}
\left\|u-\sum^n_{i=1}\lambda_i\hat{c}_i \hat{u}_{i}\right\|_\mathcal{Y}\leq \lambda_{n+1}\left(\|f\|_{\mathcal{X}}+\|\mathcal{N}(u)\|_{\mathcal{X}}\right)\,,
\end{equation}
where $\hat{c}_i$ is the projection
\begin{equation}\label{eqn:choiceofxnonlinear}
\hat{c}_i = \left\langle f-\mathcal{N}(u), \hat{v}_{i} \right\rangle  \quad \text{for} \quad i=1\,,\dots\,,n\,.
\end{equation}
\end{theorem}
The proof is straightforward: Recall Theorem~\ref{thm:basis} and apply the triangle inequality to split the bound for $f$ and $\mathcal{N}(u)$.

We note that $\mathcal{N}(u)$ on the right hand side of~\eqref{eqn:basis_nonlin} is the application of $\mathcal{N}$ on the true solution $u$ and thus $\left(\|f\|_{\mathcal{X}}+\|\mathcal{N}(u)\|_{\mathcal{X}}\right)$ is a bounded quantity. This indicates that if the sequence of $\{\lambda_n\}$ is vanishing, the true solution $u$ could nevertheless be represented well by $\{\hat{u}_i\}$, the basis of the linear component $\mathcal{L}$. These estimates let us to approximate the solution as:
\begin{equation}\label{eqn:approx_nonlinear_u}
u(x)\approx u_n = \sum^{n}_{i=1}\lambda_ic_{i} \hat{u}_{i}
\,,\quad
\mathcal{L}u(x)\approx \mathcal{L}u_n = \sum^{n}_{i=1}c_i \hat{v}_{i}
\end{equation}
where the coefficients are
\begin{equation}\label{eqn:firstchoiceofci}
\{\hat{c}_{i}\}^{n}_{i=1}=\argmin\limits_{\left\{c_{i}\right\}^n_{i=1}\subset\mathbb{R}} \left\|\sum^{n}_{i=1}{c_{i}}\hat{v}_{i}-\left(f-\mathcal{N}\left(\sum^{n}_{i=1}\lambda_ic_{i} \hat{u}_{i}\right)\right)\right\|_{\mathcal{X}}\,,
\end{equation}
or
\begin{equation}\label{eqn:secondchoiceofci}
\{\hat{c}_{i}\}^{n}_{i=1}=\argmin\limits_{\left\{c_{i}\right\}^n_{i=1}\subset\mathbb{R}} \left\|\sum^{n}_{i=1}{c_{i}}\hat{v}_{i}-P_n\left(f-\mathcal{N}\left(\sum^{n}_{i=1}\lambda_ic_{i} \hat{u}_{i}\right)\right)\right\|_{\mathcal{X}}\,.
\end{equation}
In the formula, $P_n$ is the projection map onto $\text{Span}\left\{\hat{v}_{i}\right\}^n_{i=1}$, namely:
\begin{equation}\label{eqn:proj_P}
P_n(g) = \sum_{i=1}^n\langle g,\hat{v}_i\rangle_\mathcal{X} \hat{v}_i\,,\quad P_n^\perp(g) = \sum_{i=n+1}^\infty\langle g,\hat{v}_i\rangle_\mathcal{X} \hat{v}_i\,.
\end{equation}

We will show that under mild assumptions, the form of~\eqref{eqn:approx_nonlinear_u} with the coefficients $\{\hat{c}_i\}$ computed from~\eqref{eqn:firstchoiceofci} or \eqref{eqn:secondchoiceofci} provide good approximation. The theorem is built upon the following stability assumption.
\begin{assumption}[Stability and locally Lipschitz property]\label{assum:stability} The equation \eqref{eqn:MP1r} and the nonlinear operator $\mathcal{N}$ satisfy the following properties:
\begin{enumerate}[label=$\bullet$,wide, labelwidth=!, labelindent=10pt]
    \item There exists a non-decreasing function $E:\mathbb{R}^+\rightarrow\mathbb{R}^+$ such that $\lim_{x\rightarrow0}E(x)=0$ and for any $f,g\in\mathcal{X}$, if $u_f,u_g\in\mathcal{Y}$ are solutions to \eqref{eqn:MP1r} with source term $f$ and $g$, respectively, then
    \begin{equation}\label{eqn:stability}
    \left\|u_f-u_g\right\|_\mathcal{Y}\leq E\left(\|f-g\|_\mathcal{X}\right).
    \end{equation}
    \item $\mathcal{N}:\mathcal{Y}\rightarrow\mathcal{X}$ is a locally Lipschitz operator.
\end{enumerate}
\end{assumption}
Under this assumption, we have the following result.
\begin{theorem}\label{thm:nonlinearcase2}
  Suppose that Assumptions \ref{assum:SR} and~\ref{assum:stability} hold and that $n>0$ is sufficiently large. For any $f\in\mathcal{X}$, suppose the solution $u$ to \eqref{eqn:MP1r} satisfies $u\in\mathcal{Y}$ and that $\mathcal{N}(u(x))\in\mathcal{X}$.
 Then  we have
\begin{itemize}
\setlength{\itemindent}{2em}
    \item[]{--} For $u\approx u_n=\sum^{n}_{i=1}\lambda_i\hat{c}_{i} \hat{u}_{i}$ where $\hat{c}_{i}$ is defined in \eqref{eqn:firstchoiceofci}, we have
\begin{equation}\label{eqn:firsterrorbound}
\left\|u-u_n\right\|_\mathcal{Y}\leq E\left(\left(1+C\lambda_{n+1}\right)
\mathcal{E}_{n}^{(1)}
\right)\,,
\end{equation}
where $\mathcal{E}_{n}^{(1)}=\left\|P^{\perp}_n\left(f-\mathcal{N}\left( u \right)\right)\right\|_{\mathcal{X}}$, and
\item[]{--} For $u\approx u_n=\sum^{n}_{i=1}\lambda_i\hat{c}_{i} \hat{u}_{i}$ where $\hat{c}_{i}$ is defined in~\eqref{eqn:secondchoiceofci}, we have
\begin{equation}\label{eqn:seconderrorbound}
\begin{aligned}
\left\|u-u_n\right\|_\mathcal{Y}\leq
\lambda_{n+1}\mathcal{E}_{n}^{(2)} + E\left(C\lambda_{n+1}\left(\mathcal{E}_{n}^{(1)}+\mathcal{E}_{n}^{(2)}\right)\right),
\end{aligned}
\end{equation}
where $\mathcal{E}_{n}^{(2)}=\left\|P^{\perp}_n\left(f-\mathcal{N}\left(\sum^n_{i=1}\hat{c}_{i}\hat{u}_{i}\right)\right)\right\|_{\mathcal{X}}$ with $\{\hat{c}_i\}$ solving~\eqref{eqn:secondchoiceofci}, and $P^\perp_n$ is defined in~\eqref{eqn:proj_P}.
\end{itemize}
In these estimates, $C$ is a constant depending on $\mathcal{N},\|u\|_\mathcal{Y}$ and $E$ is the function defined in~\eqref{eqn:stability}.
\end{theorem}
The proof of this result is deferred to~\Cref{app:thm_nonlinearcase2}.

The theorem essentially gives a justification for using the basis of $\{\hat{u}_i\}$ to represent the solution, as long as the coefficients are found properly. We should note that though the two formulations~\eqref{eqn:firstchoiceofci} and~\eqref{eqn:secondchoiceofci} look rather similar, the error estimates present very different behavior. By comparing~\eqref{eqn:firsterrorbound} and~\eqref{eqn:seconderrorbound}, we see that the first choice for finding $\{\hat{c}_i\}$ has one level of smallness while the second choice is second order. To be more specific: Supposing that $E$ is Lipschitz, the smallness of the error in~\eqref{eqn:firsterrorbound} only comes from the smallness of $\mathcal{E}_n^{(1)}$, the projection error. The decaying $\lambda_{n+1}$ does not introduce any further control of the error. On the other hand, there are two terms in~\eqref{eqn:seconderrorbound}, both of which have the quadratic form. For example, the first term $\lambda_{n+1}\mathcal{E}_n^{(2)}$ is small both because of the decaying $\lambda_{n+1}$ and because of the small projection error. The second term in~\eqref{eqn:seconderrorbound} is also quadratically small. In some sense, the second formulation~\eqref{eqn:secondchoiceofci} benefits from both the decaying $\lambda_{n+1}$ and the decaying truncation error, while the first formulation~\eqref{eqn:firsterrorbound} tapers off with the control reflecting only the truncation error.

Nevertheless, both estimates suggest that the leading modes for the linear part of the operator $\{\hat{u}_i\}$ form a good representation basis for presenting the solution to the nonlinear equation.
We choose to execute the minimization~\eqref{eqn:secondchoiceofci}
through fixed-point iteration, as described
in~\Cref{ALG2}. In the specification of this algorithm, we denote by $\{\mathsf{\hat{v}}_{i}\}^{n}_{i=1}$ and $\left\{\mathsf{\hat{u}}_{i}\right\}^{n}_{i=1}$  the numerical basis functions computed in~\Cref{ALG1} in the offline preparation stage. We note that alternatives to the fixed-point iteration could  be employed, but this is beyond the scope of the current paper.

\begin{algorithm}[h]
\caption{\textbf{Nonlinear multiscale PDE solver in $\Omega$}}\label{ALG2}
\begin{algorithmic}
\State Input: the discrete linear operators $\mathsf{L}$; the discrete source $\mathsf{f}$; nonlinear operator $\mathcal{N}$.

\State \textbf{Offline:} Prepare $\{\hat{\lambda}_i\}_{i=1}^n$; $\left\{\mathsf{\hat{v}}_{i}\right\}^{n}_{i=1}$; $\left\{\mathsf{\hat{u}}_{i}\right\}^{n}_{i=1}$ using~\Cref{ALG1}.
\State \textbf{Online:}
Choose error tolerance $\epsilon>0$, and update $\left\{c^k_{i}\right\}^{n}_{i=1}$ using:
\begin{equation}\label{eqn:updateci}
c^{k+1}_i=\left\langle \mathsf{f} - \mathcal{N}\left(\sum^{n}_{i=1} \hat{\lambda}_i c^k_{i} \mathsf{\hat{u}}_{i} \right),  \mathsf{\hat{v}}_{i} \right\rangle_\mathsf{X}
\end{equation}
until $\sum^n_{i=1}\hat{\lambda}_i^2|c^{k+1}_i-c^{k}_i|^2 < \epsilon$. Let $\hat{c}_i = c^{k+1}_i$ for $i=1\,,\dots, n$.

\State \textbf{Output:}
\[
\mathsf{u}=\sum^{n}_{i=1} \hat{\lambda}_i \hat{c}_i \mathsf{\hat{u}}_{i}\,.
\]
\end{algorithmic}
\end{algorithm}
\begin{remark}
The intuition of the iteration formula \eqref{eqn:updateci} in~\Cref{ALG2} comes from the proof of the wellposedness of~\eqref{eqn:MP1r}. In~\cite{chipot2009elliptic}, one way to show the existence of solution is to prove that $u^n$ from the following updating formula will converge when $n\rightarrow\infty$:
\[
\mathcal{L} (u^{n+1}(x)) =f(x)-\mathcal{N}(u^{n}(x)) \,.
\]
Numerically, projecting both sides of this updating formula in $\text{Span}\left\{\mathsf{\hat{v}}_{i}\right\}^n_{i=1}$ gives the iteration formula~\eqref{eqn:updateci}.
\end{remark}

\section{Numerical results}\label{sec:numerical}

We present numerical evidence to showcase the discussions of the previous sections. Two subsections are dedicated to linear and nonlinear problems, respectively. In all examples, we use spatial domain $\Omega = [0,.5]^2$.

\subsection{Linear systems}\label{sec:linear_num}
This subsection studies~\eqref{eqn:MPc} numerically, inverting a linear operator $\mathcal{L}$ for any given source term $f(x)$. We present two examples, an elliptic equation and the radiative transport equation, both having strong inhomogeneities in the media and  multiscale behavior.

\subsubsection{Elliptic equation}\label{sec:linear_elliptic}
We consider the linear elliptic equation with highly oscillatory medium:
\begin{equation}\label{eqn:linear_elliptic}
\begin{cases}
-\nabla \cdot (\kappa(x_1,x_2) \nabla u(x_1,x_2)) = f(x_1,x_2),\quad& (x_1,x_2)\in\Omega\,, \\
u(x_1,x_2) = 0,\quad& (x_1,x_2)\in\partial\Omega\,,
\end{cases}
\end{equation}
where the medium $\kappa(x_1,x_2)$ has the form of
\begin{equation} \label{eqn:medium}
\kappa(x_1,x_2) = 2+\sin(2\pi x_1)\cos(2\pi x_2)+\frac{2+1.8\sin(2\pi
  x_1/\eps)}{2+1.8\cos(2\pi x_2 / \eps)} + \frac{2+\sin(2\pi x_2
  /\eps)}{2+1.8\cos(2\pi x_1/\eps)}\,.
\end{equation}
We plot the media in \cref{fig:medium} for different values of $\eps$.
\begin{figure}[htbp]
  \centering
  \includegraphics[width=0.3\textwidth]{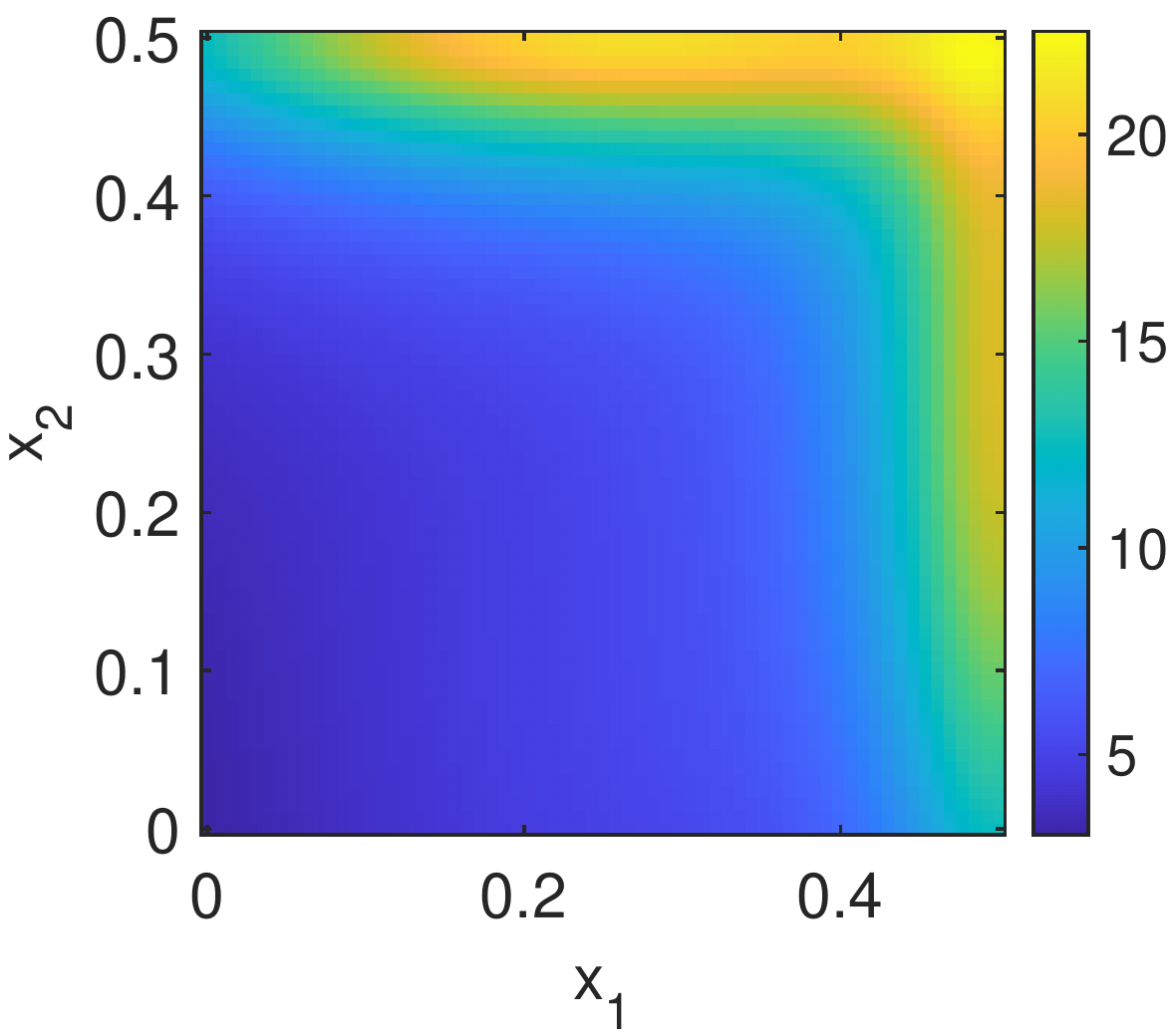}
  \includegraphics[width=0.3\textwidth]{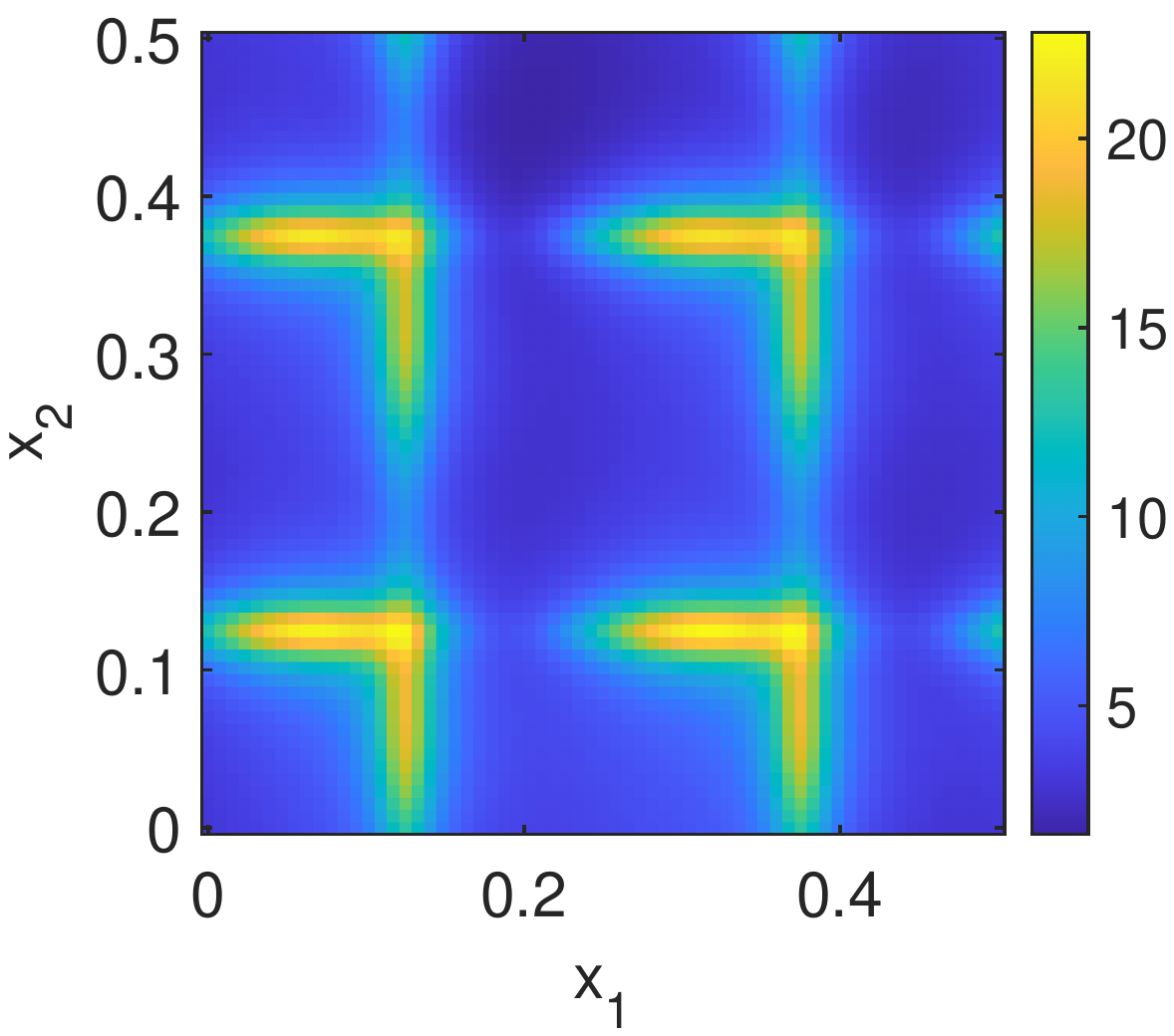}
  \includegraphics[width=0.3\textwidth]{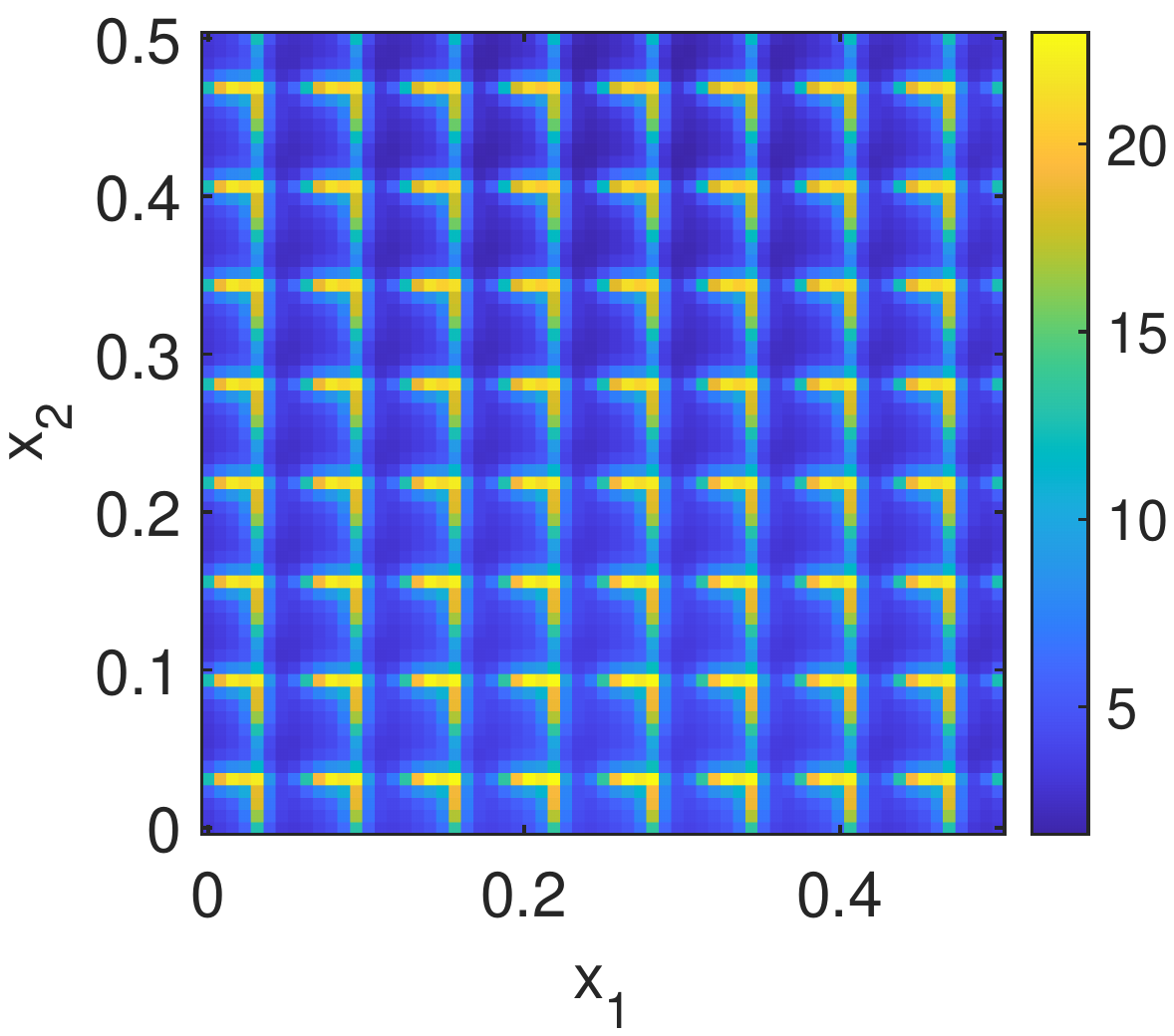}
  \caption{Multiscale medium $\kappa$~\eqref{eqn:medium} for $\eps = 2^0$ (left), $2^{-2}$ (middle), $2^{-4}$ (right).}
  \label{fig:medium}
\end{figure}

The reference solutions and the optimal basis are computed using the standard finite difference scheme on uniform grid with mesh size $h = 2^{-6} = \tfrac{1}{64}$, so that even the smallest $\eps$ is resolved by 8 grid points. The total degrees of freedom $N$ equals $63 \times 63 = 3969$.

In the computation below, we set $\mathsf{\Pi}_\mathsf{Y} = \mathsf{I}_N$ (identity matrix) so that the solution space is $\mathsf{Y} = L^2$. The source space $\mathsf{X}$ is set to be $H^p$ or $L^2$. Accordingly, $\mathsf{\Pi}_\mathsf{X}$ needs to be adjusted, see Appendix~\ref{appendix:Pi}.

We first present the optimal basis functions for different choices of $\eps$. The relative singular values of the discrete Green's matrix $\mathsf{G}$, when measured in $H^p$ and $L^2$ norms, are shown in \cref{fig:svd}.
We make two observations. First, the decay rate of singular values of $\mathsf{G}$ is independent of $\eps$. This point suggests strongly the oscillation level of the medium does not affect the number of important basis for the solution.
Second, the regularity of $\mathsf{X}$, the source space, has a  significant effect: A smoother space produces a faster decay of singular values, meaning that fewer basis functions are deemed important.
In \cref{fig:left_svectors} we showcase the first three basis functions for different choices of $\eps$.
We  note that although the singular values decay with almost identical rate across $\eps$, the profiles of singular vectors (basis functions) are quite different.

\begin{figure}[htbp]
  \centering
  \includegraphics[width=0.4\textwidth]{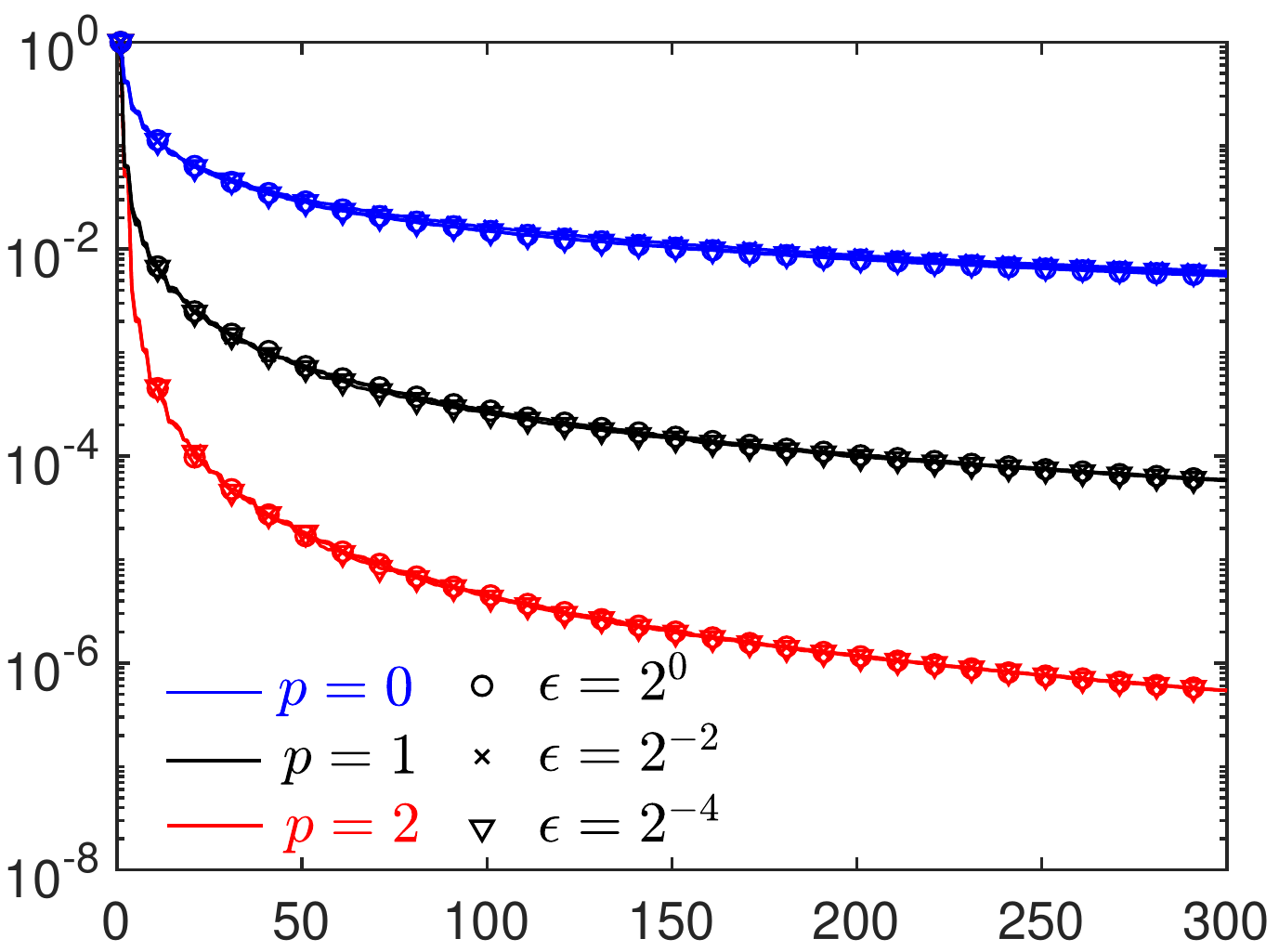}
  \caption{Singular values of the discrete Green's matrix $\mathsf{G}$ as an operator from $H^p$ to $L^2$ for $p=0$ (left), $p=1$ (middle), $p=2$ (right). The Green's matrix is computed from the linear elliptic equation~\eqref{eqn:linear_elliptic} with medium $\kappa$ defined in~\eqref{eqn:medium}.}
  \label{fig:svd}
\end{figure}

\begin{figure}[htbp]
  \centering
  \includegraphics[width=0.2\textwidth, height = 0.1\textheight]{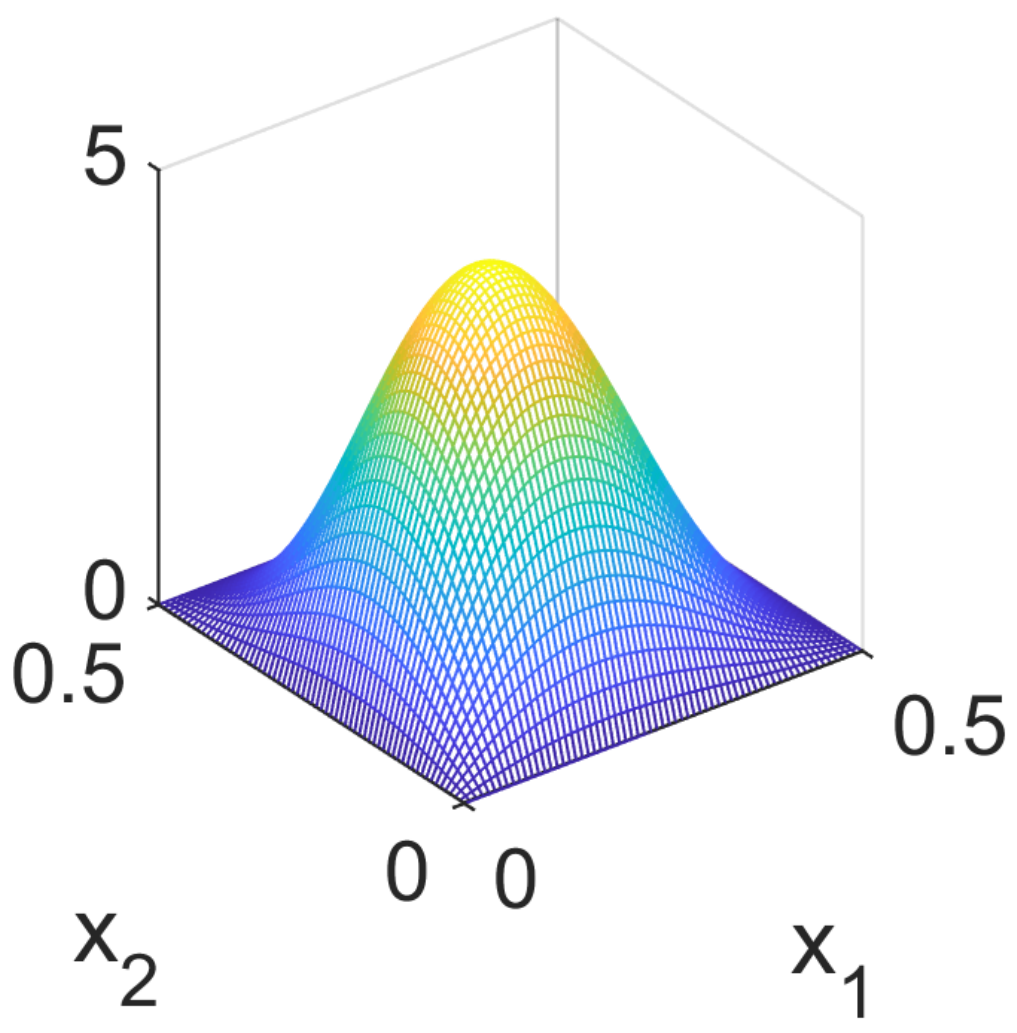}
  \includegraphics[width=0.2\textwidth, height = 0.1\textheight]{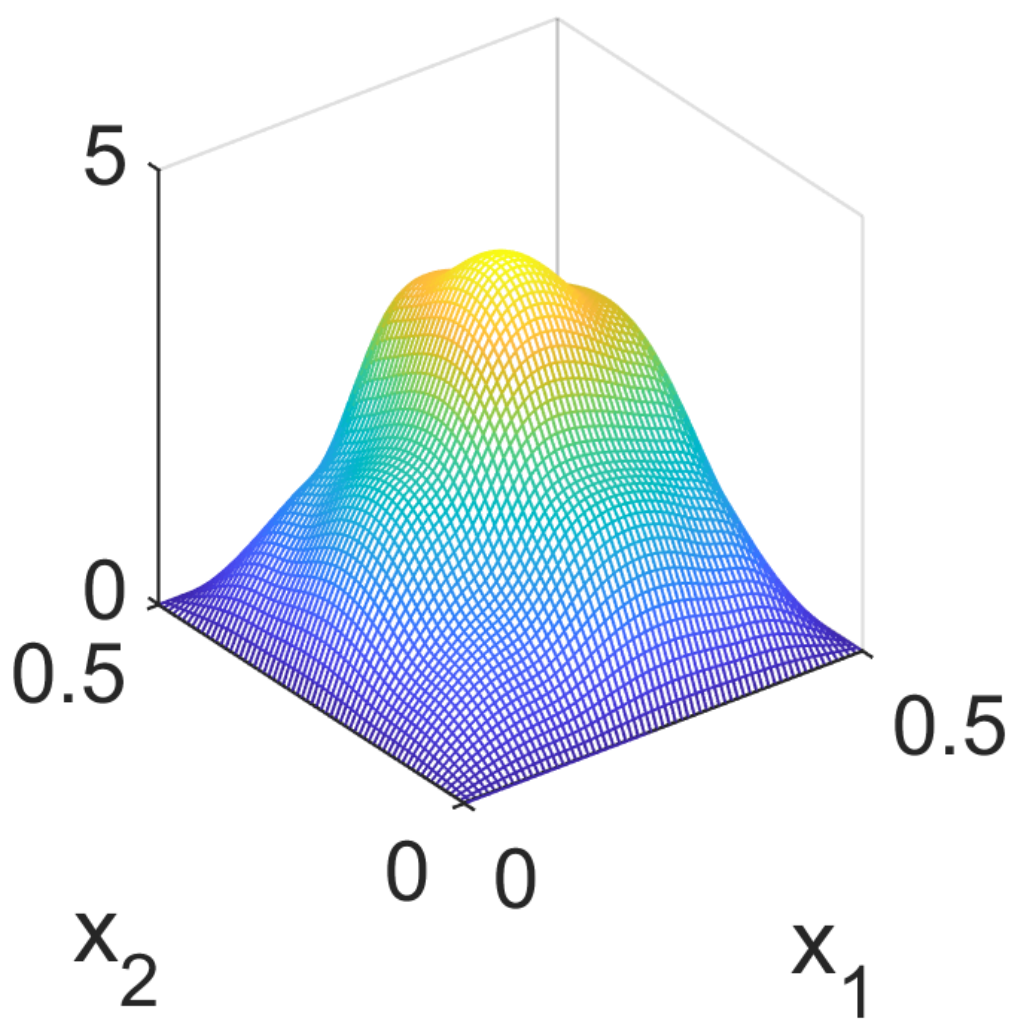}
  \includegraphics[width=0.2\textwidth, height = 0.1\textheight]{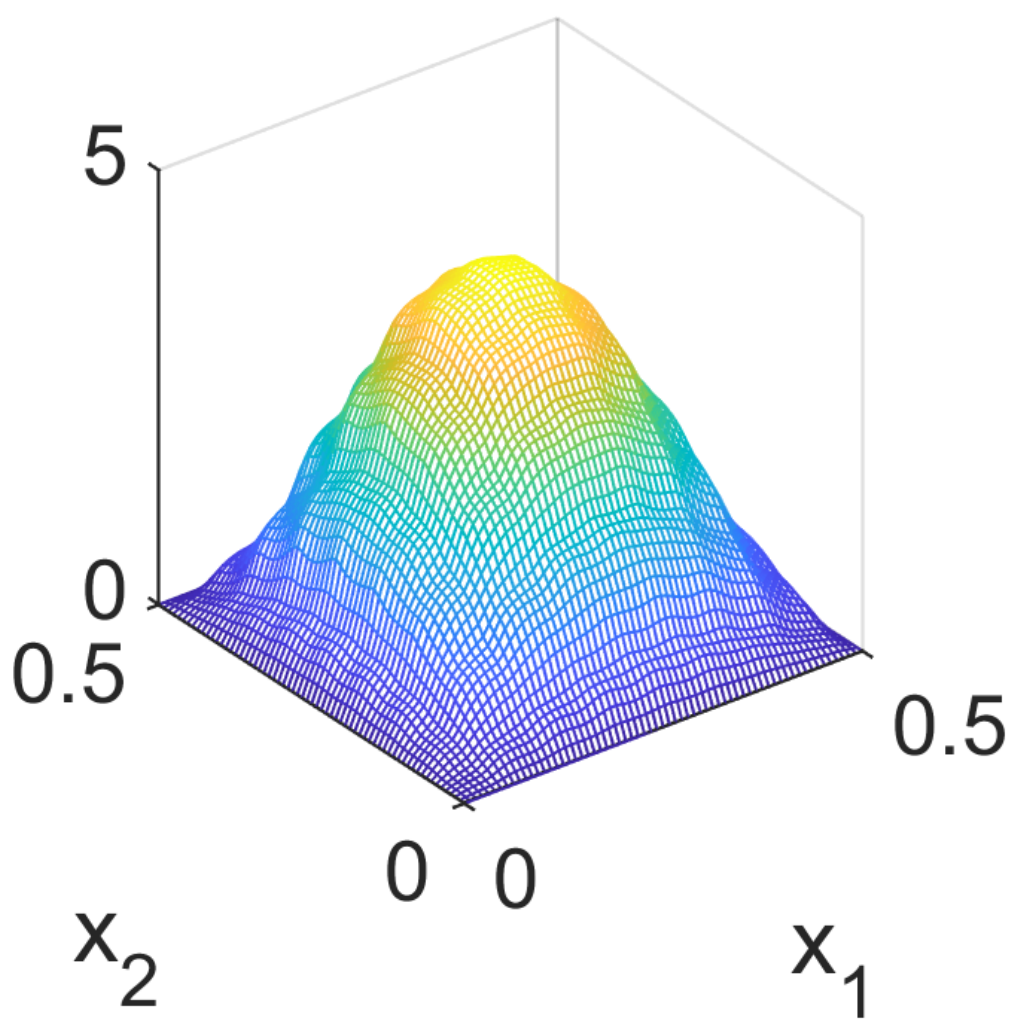}
  \\
  \includegraphics[width=0.2\textwidth, height = 0.1\textheight]{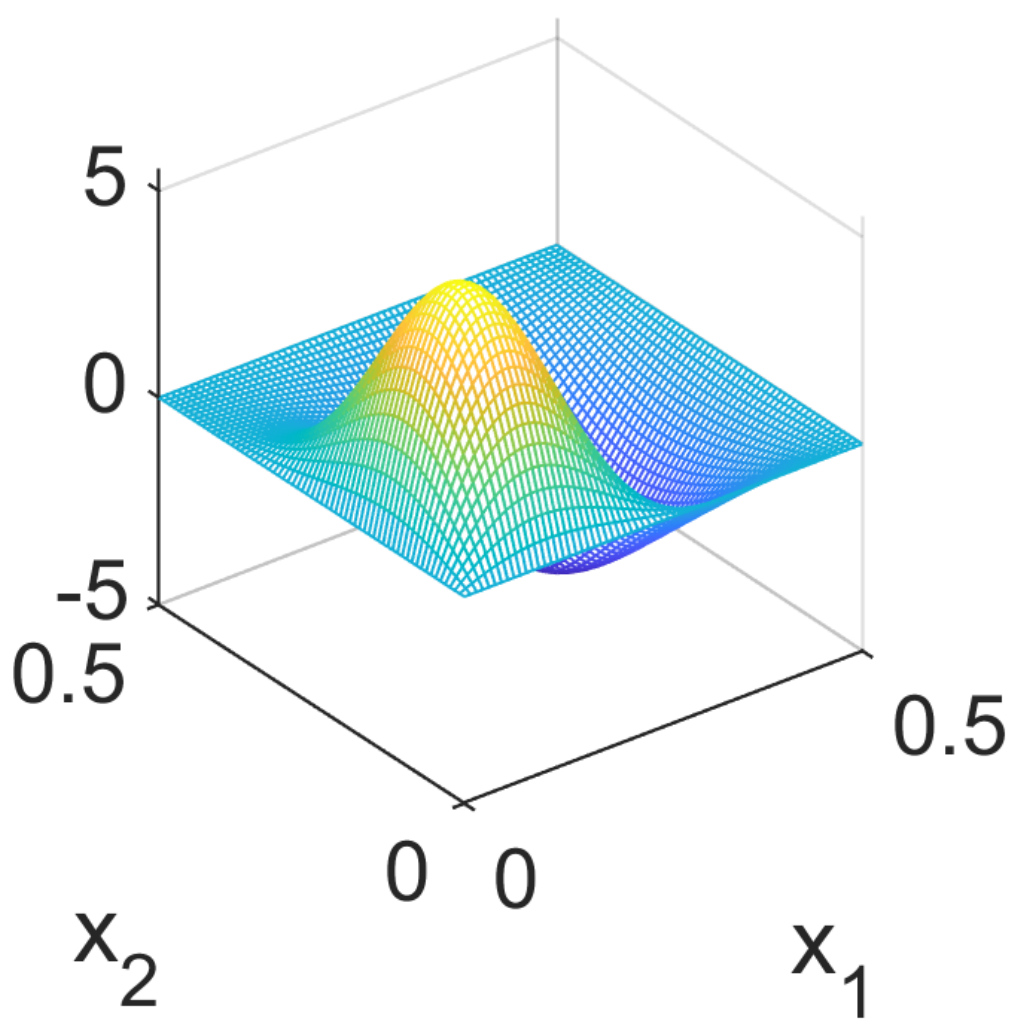}
  \includegraphics[width=0.2\textwidth, height = 0.1\textheight]{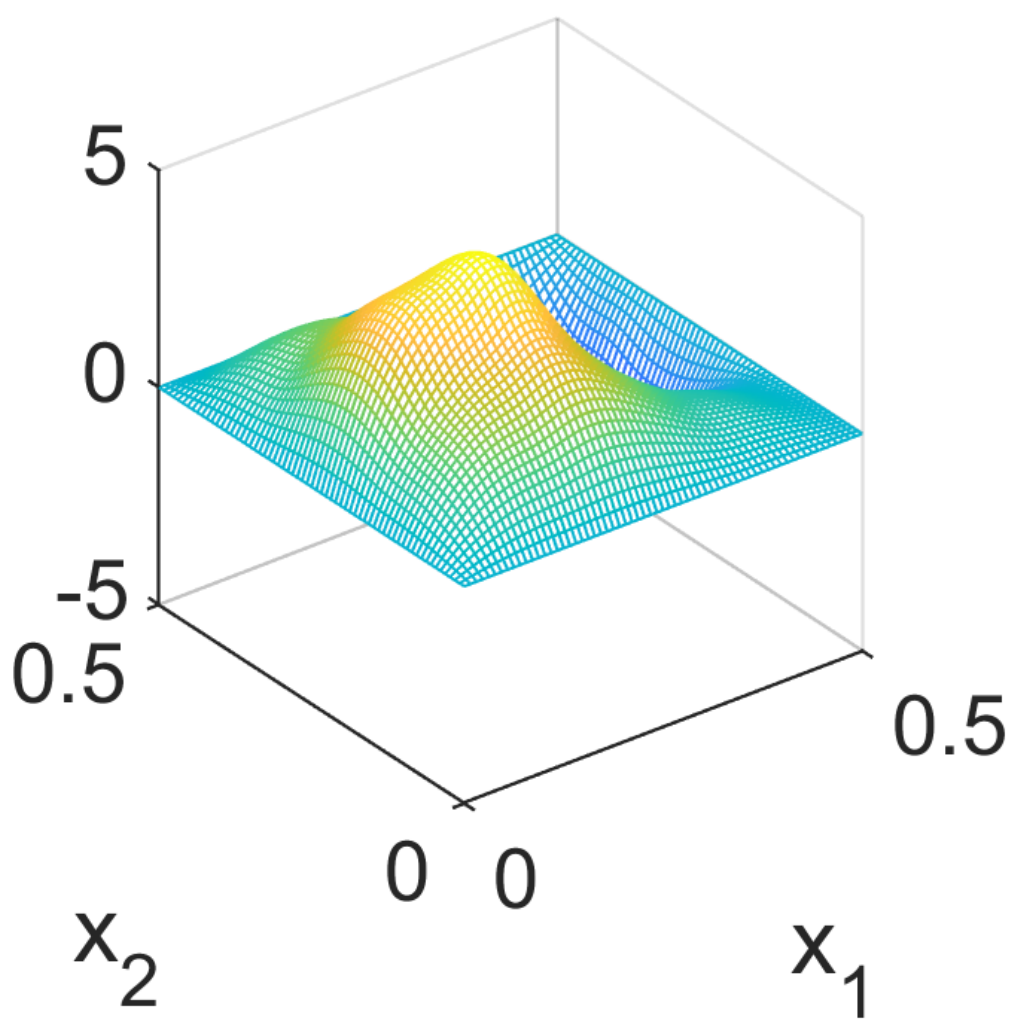}
  \includegraphics[width=0.2\textwidth, height = 0.1\textheight]{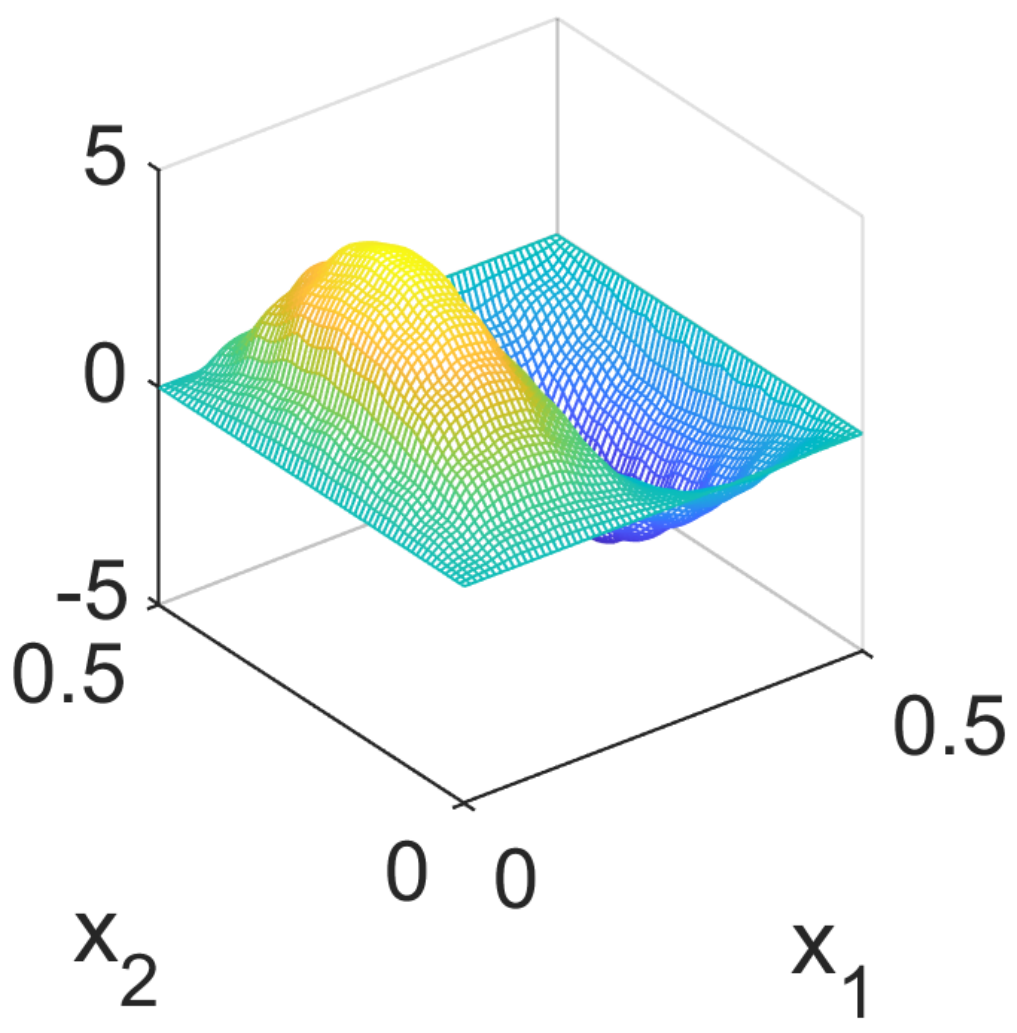}
  \\
  \includegraphics[width=0.2\textwidth, height = 0.1\textheight]{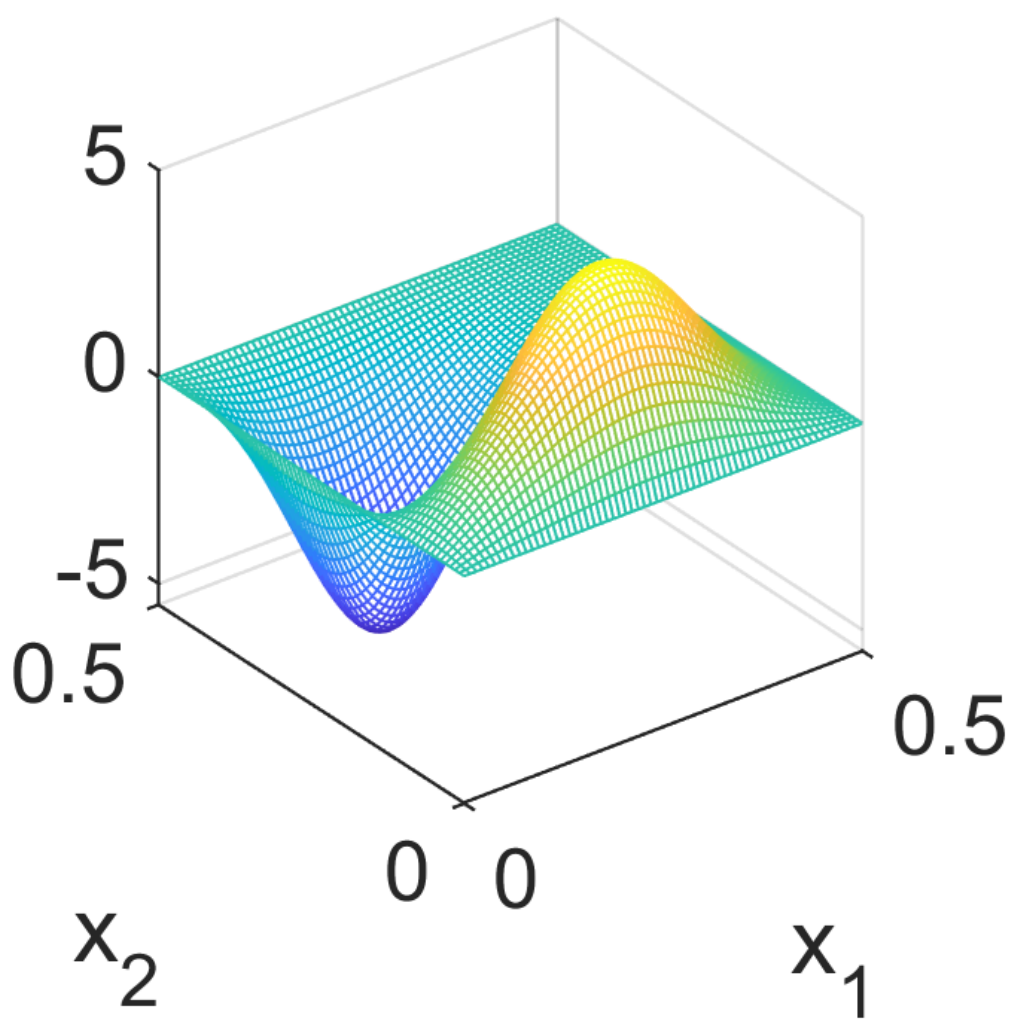}
  \includegraphics[width=0.2\textwidth, height = 0.1\textheight]{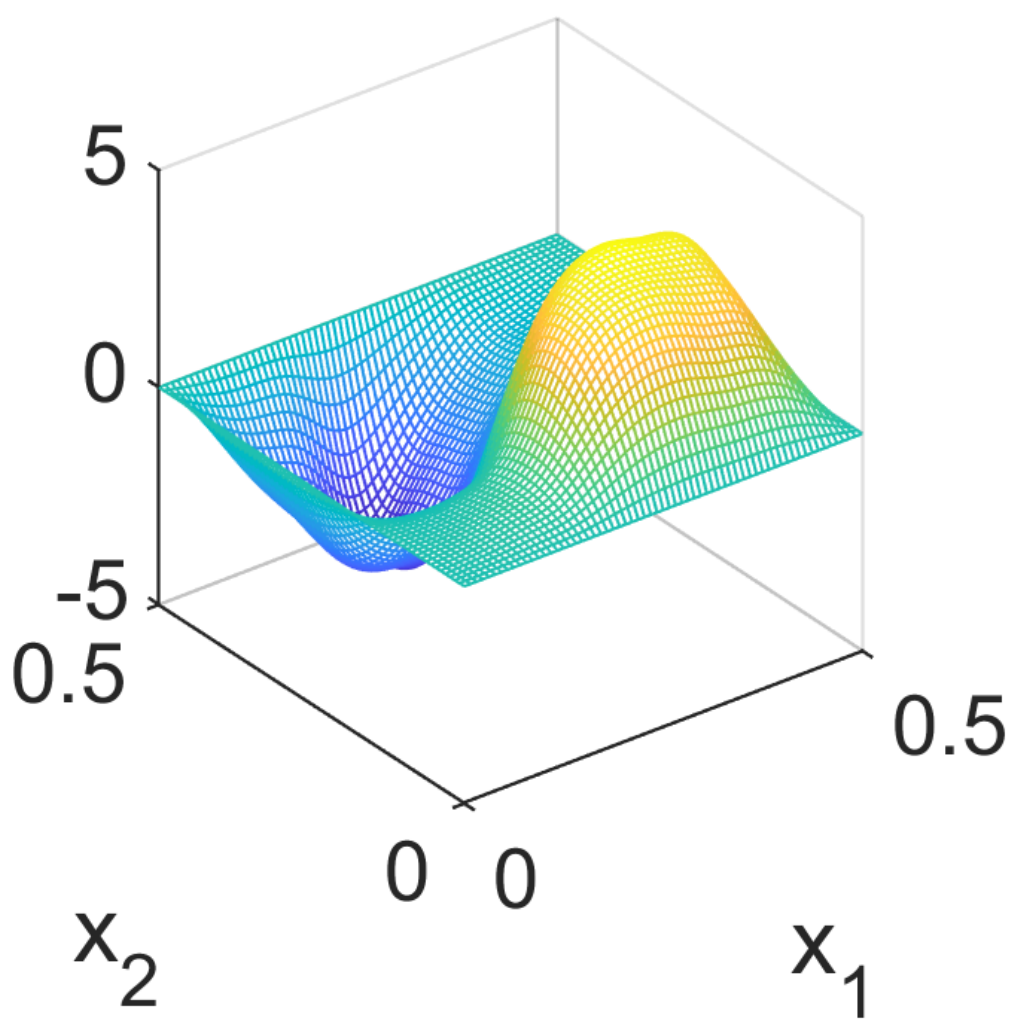}
  \includegraphics[width=0.2\textwidth, height = 0.1\textheight]{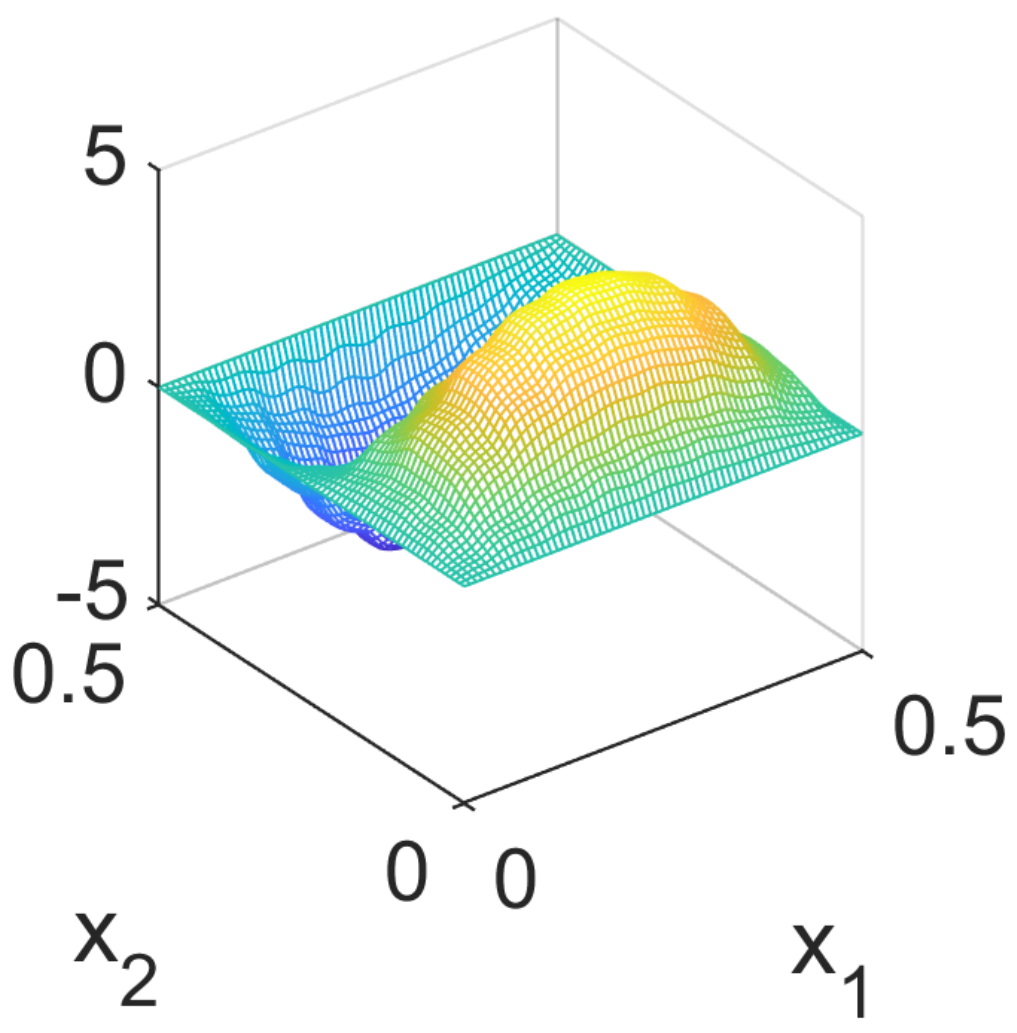}
  \caption{Left singular vectors $\mathsf{\hat{u}}_i$ of Green's matrix $\mathsf{G}$ computed from the rSVD for the linear elliptic equation~\eqref{eqn:linear_elliptic} with medium $\kappa$ defined in~\eqref{eqn:medium} for $\eps=2^{0}$ (left), $2^{-2}$ (middle), $2^{-4}$ (right). The rows from top to bottom show, respectively, the first, second and third singular vectors. The singular vectors are computed using weight matrices $\mathsf{\Pi}_\mathsf{X} = \mathsf{\Pi}_\mathsf{Y} = \mathsf{I}$.
  }
  \label{fig:left_svectors}
\end{figure}

These bases  capture well the features of the reference solutions. We test the performance of these basis function using the source term  $f(x_1,x_2) = \sin(4\pi x_1)\sin(4\pi x_2)$. In \cref{fig:ref_soln} we plot the reference solution for different values of $\eps$ and  show the decay in $L_2$ and energy norm (defined in~\eqref{eqn:energy_norm_def}).

\begin{figure}[htbp]
  \centering
  \includegraphics[width=0.25\textwidth]{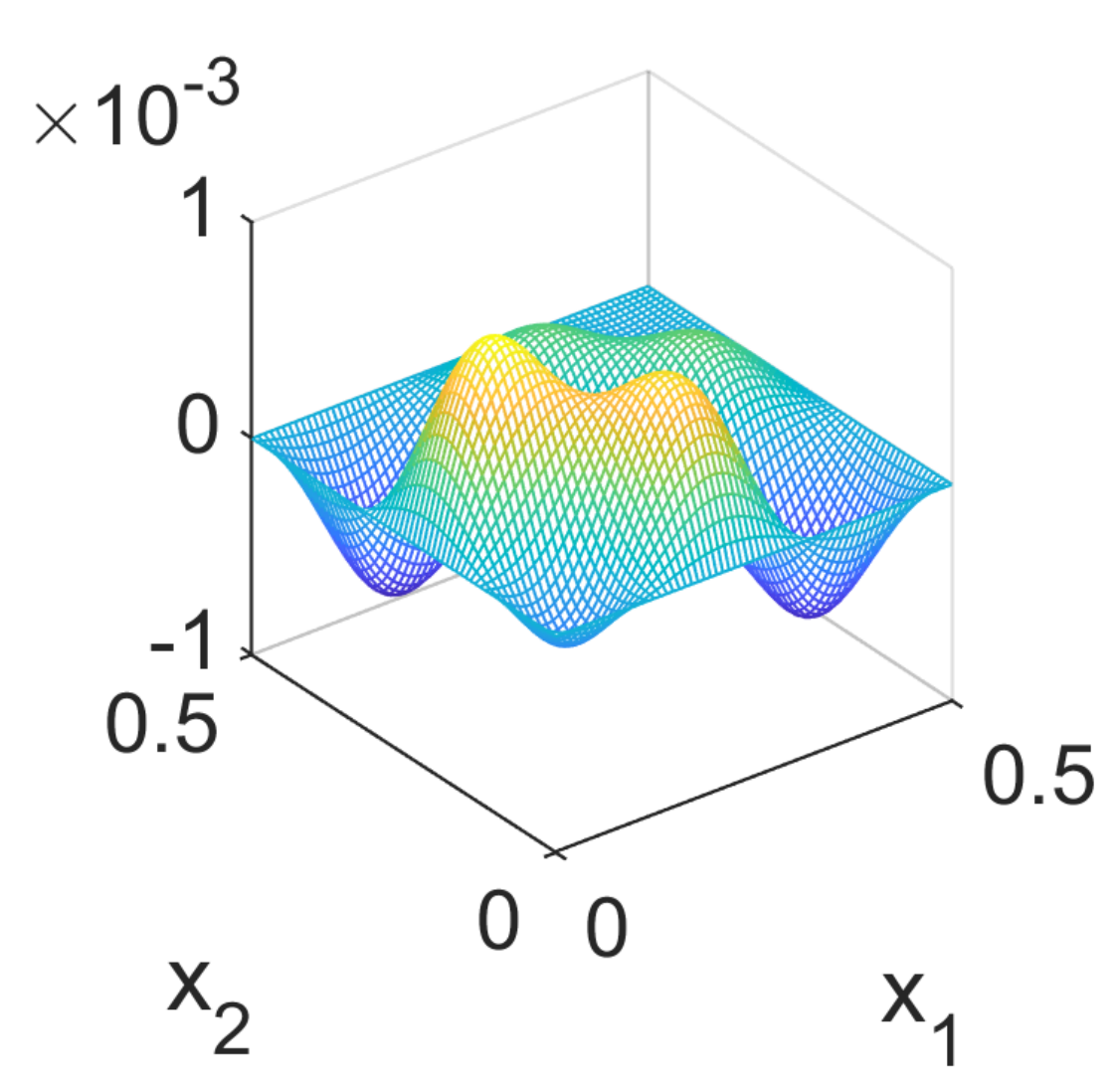}
  \includegraphics[width=0.25\textwidth]{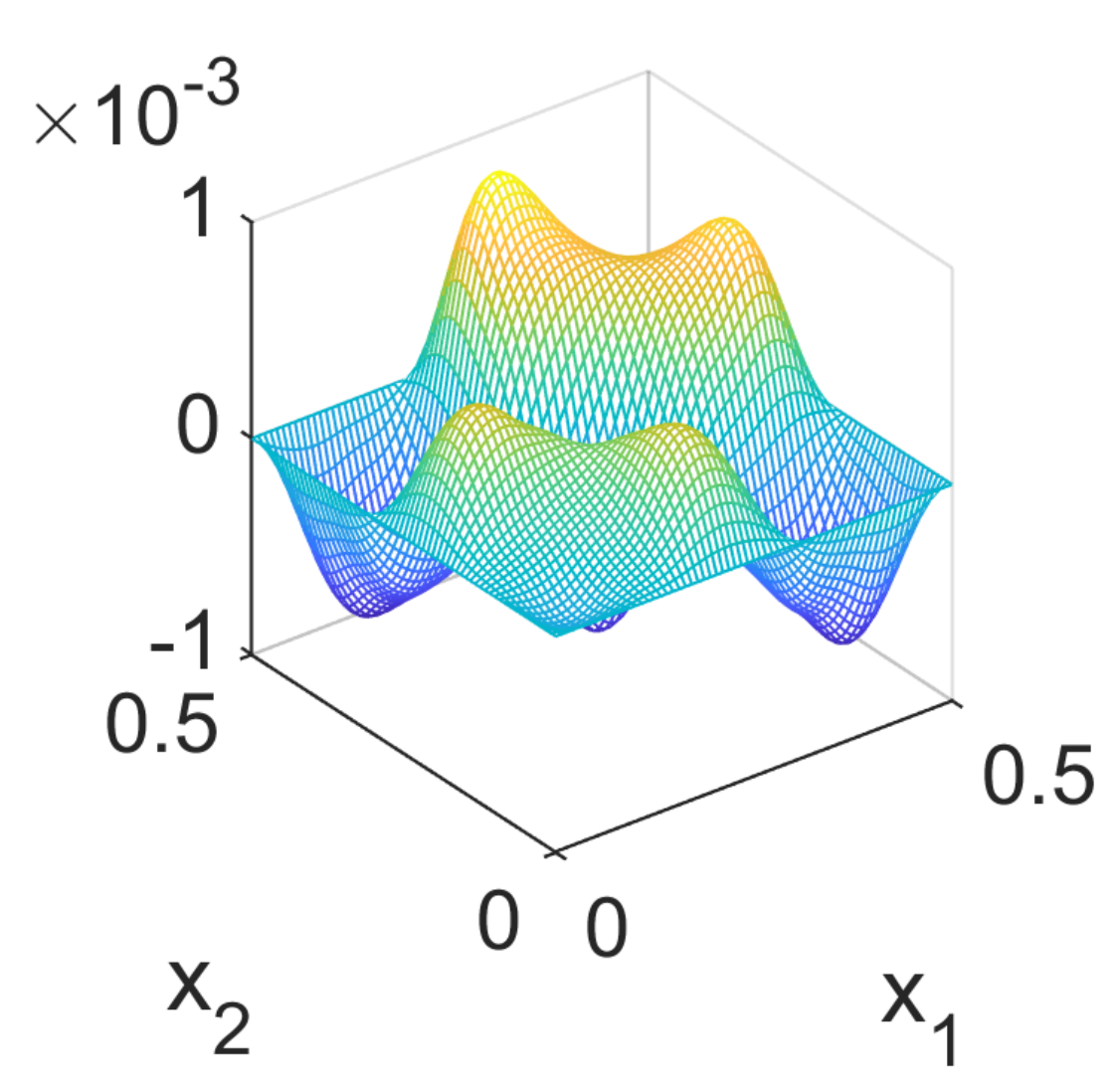}
  \includegraphics[width=0.25\textwidth]{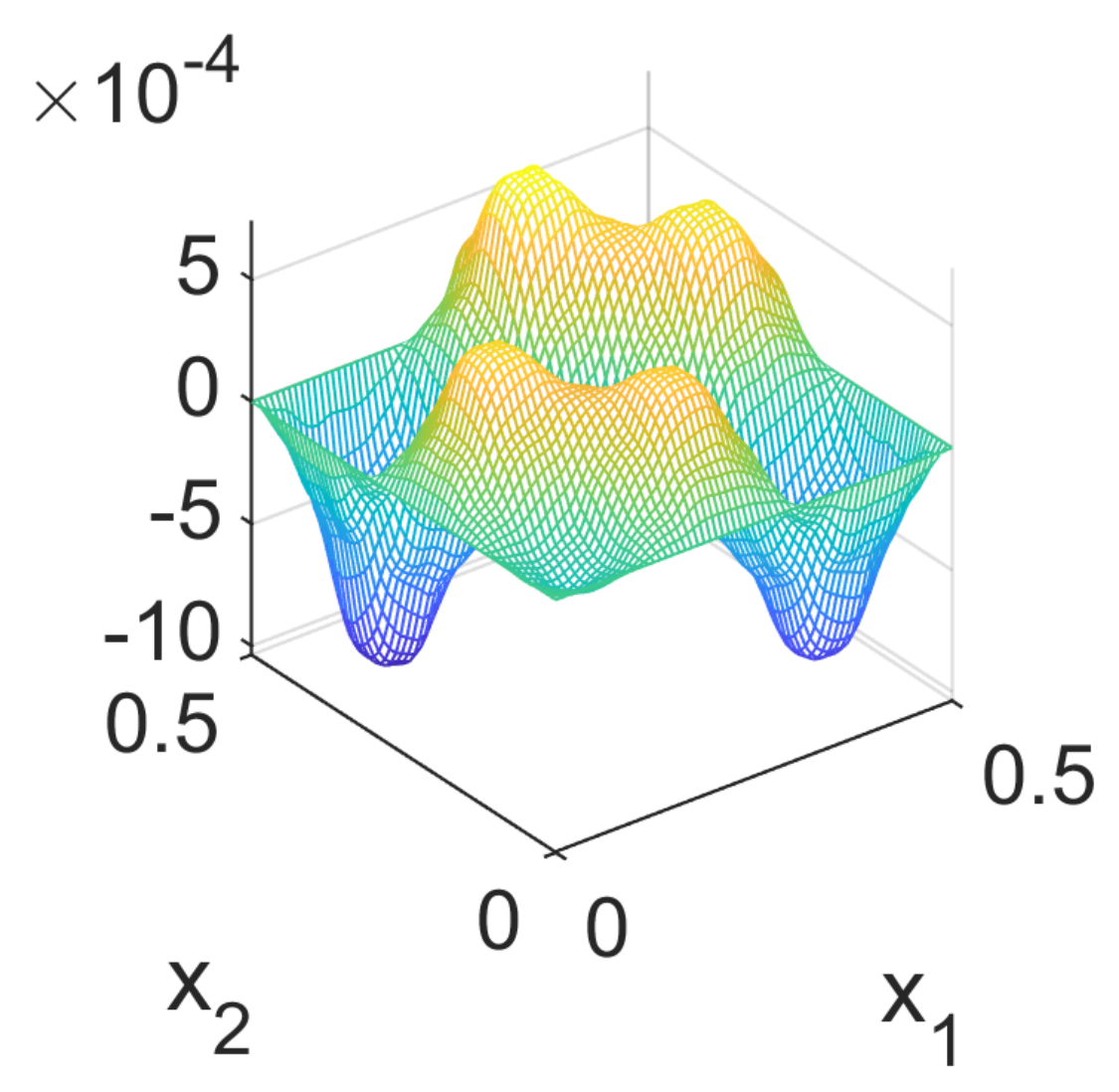} \\
  \includegraphics[width=0.25\textwidth]{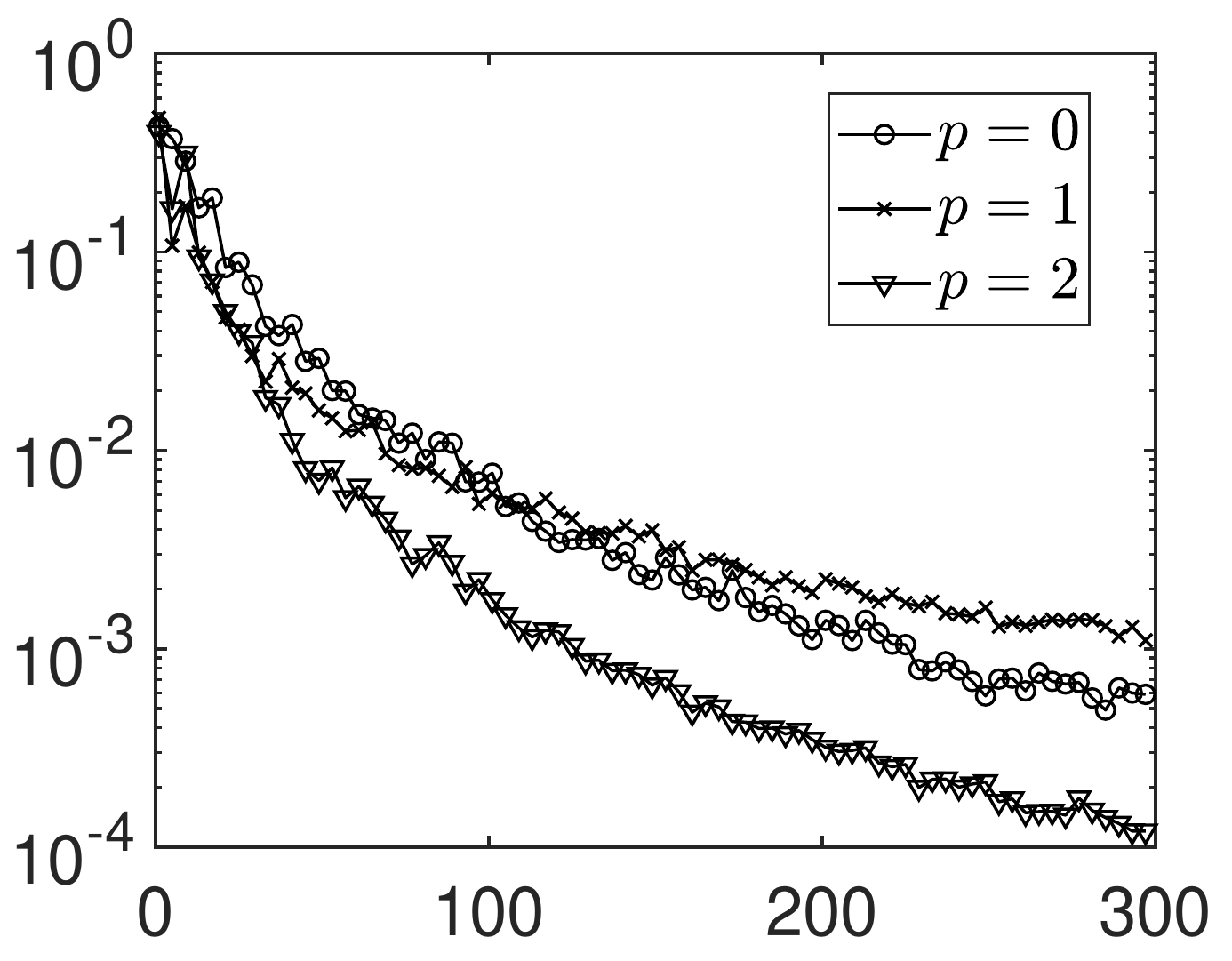}
  \includegraphics[width=0.25\textwidth]{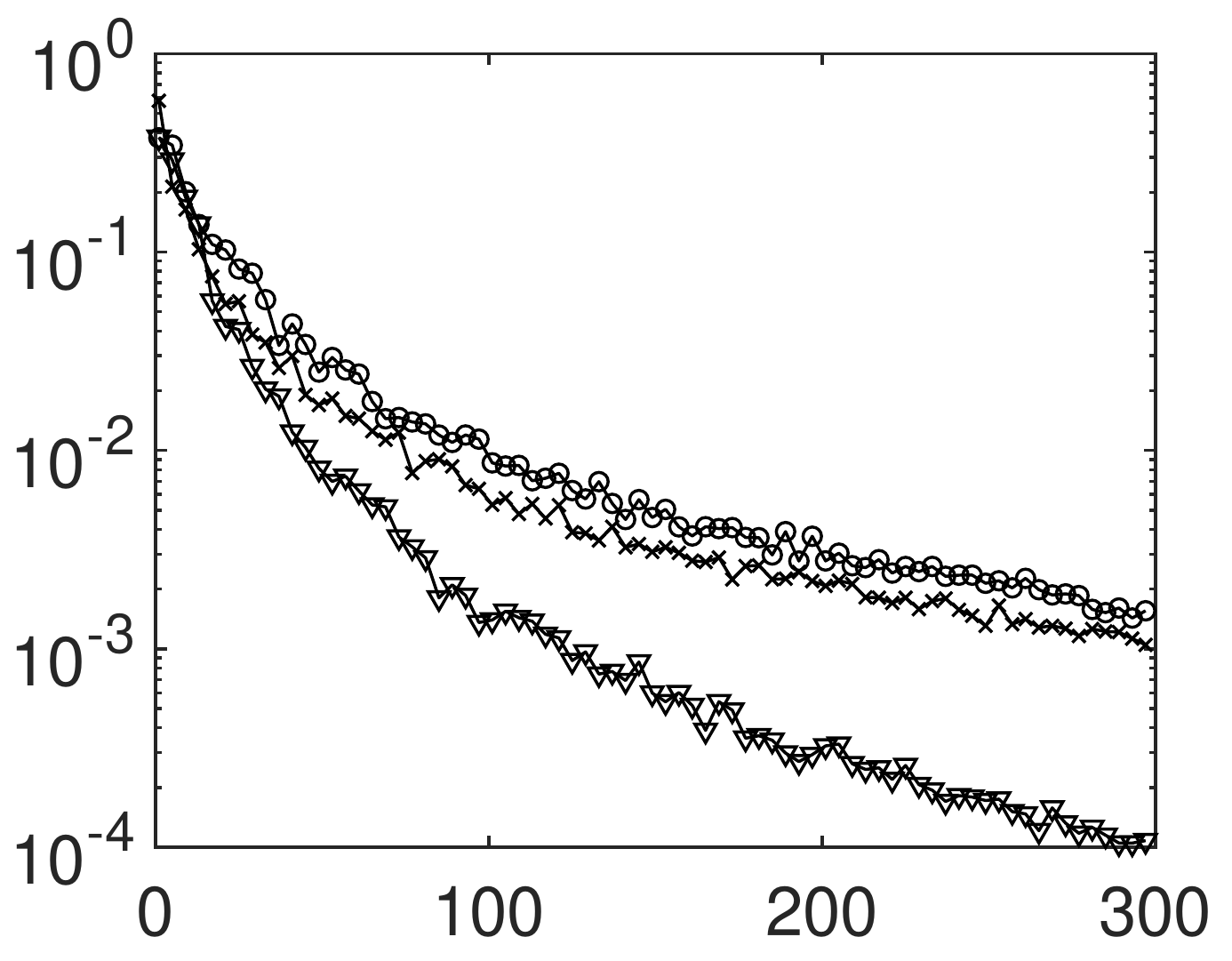}
  \includegraphics[width=0.25\textwidth]{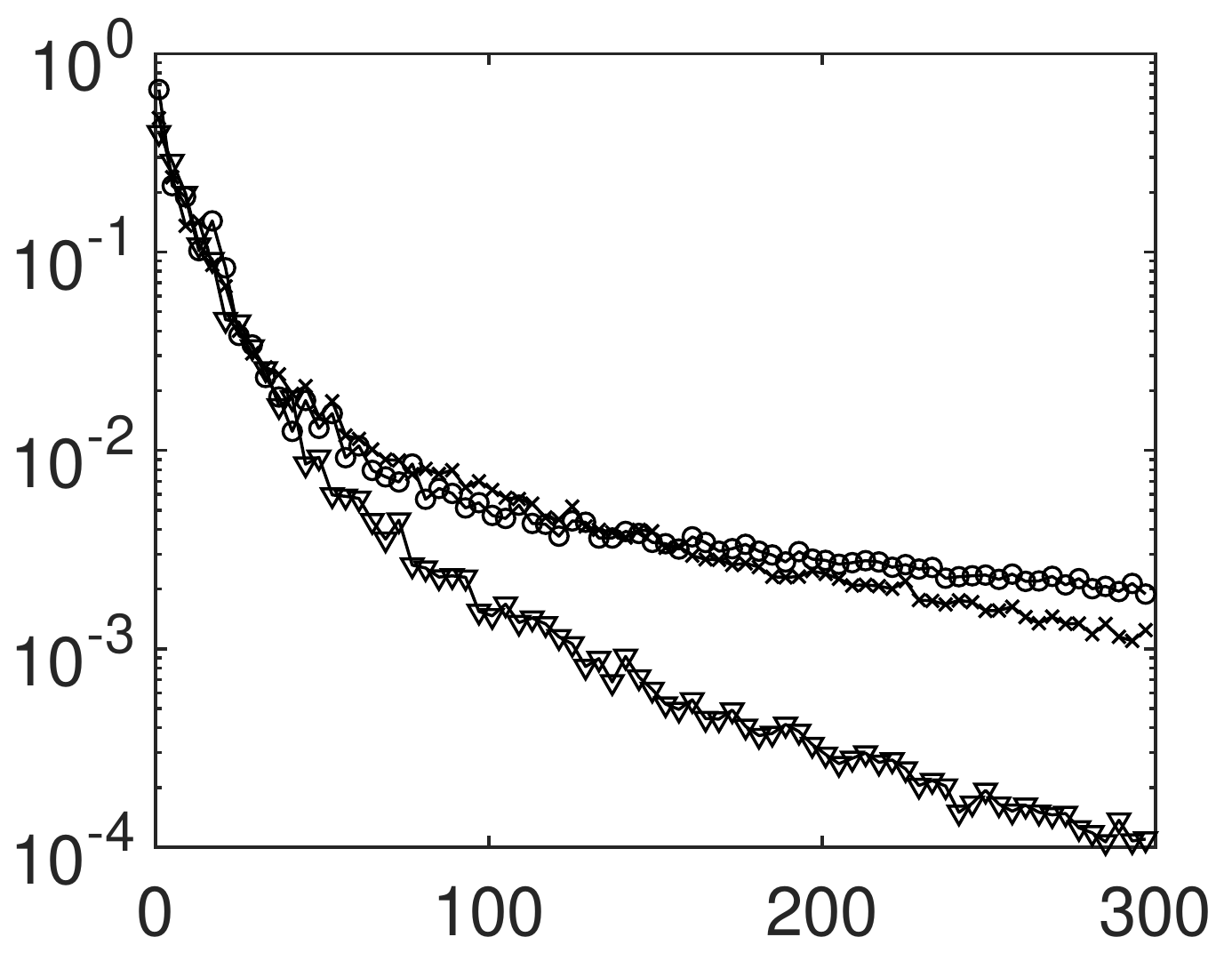}
  \\
  \includegraphics[width=0.25\textwidth]{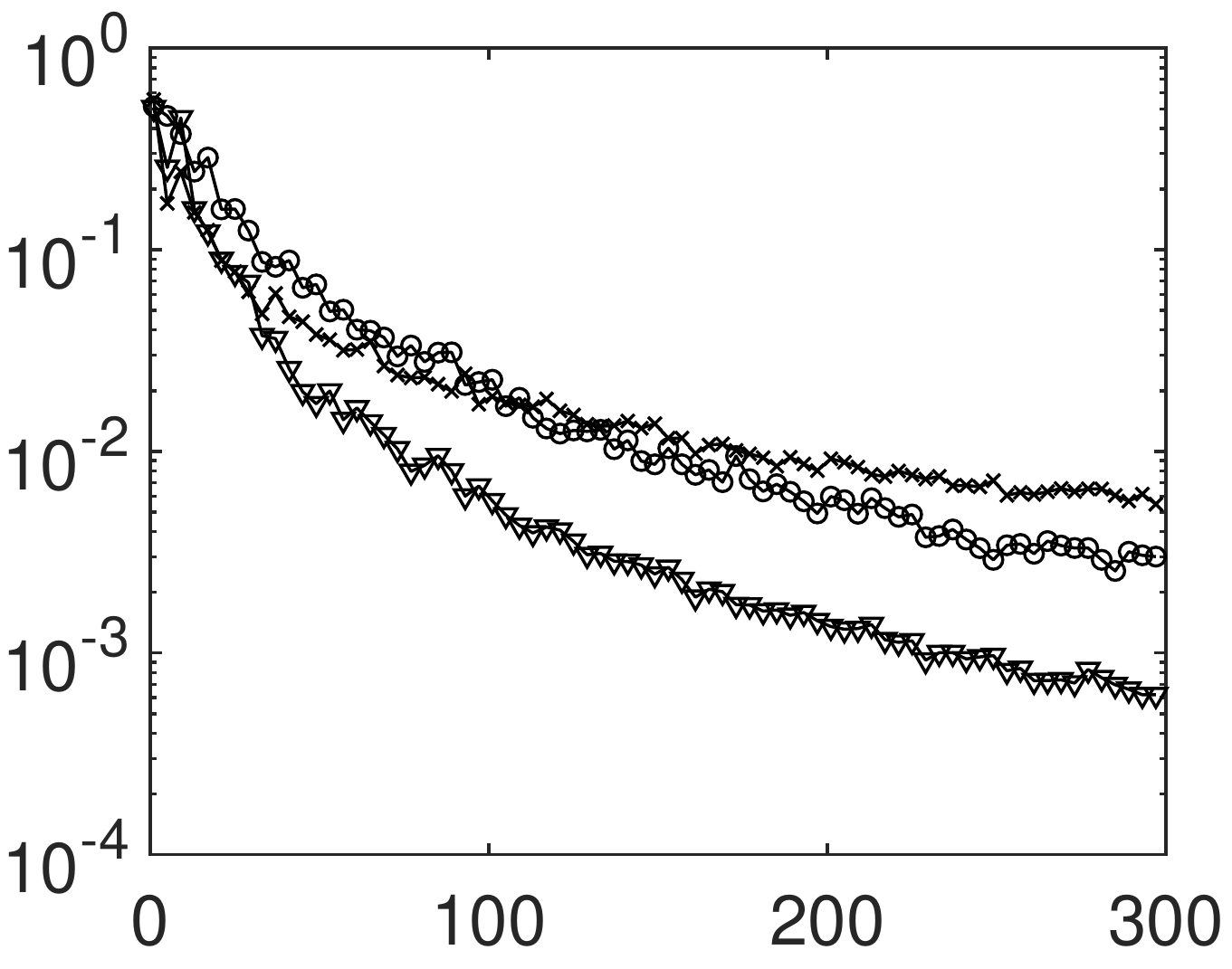}
  \includegraphics[width=0.25\textwidth]{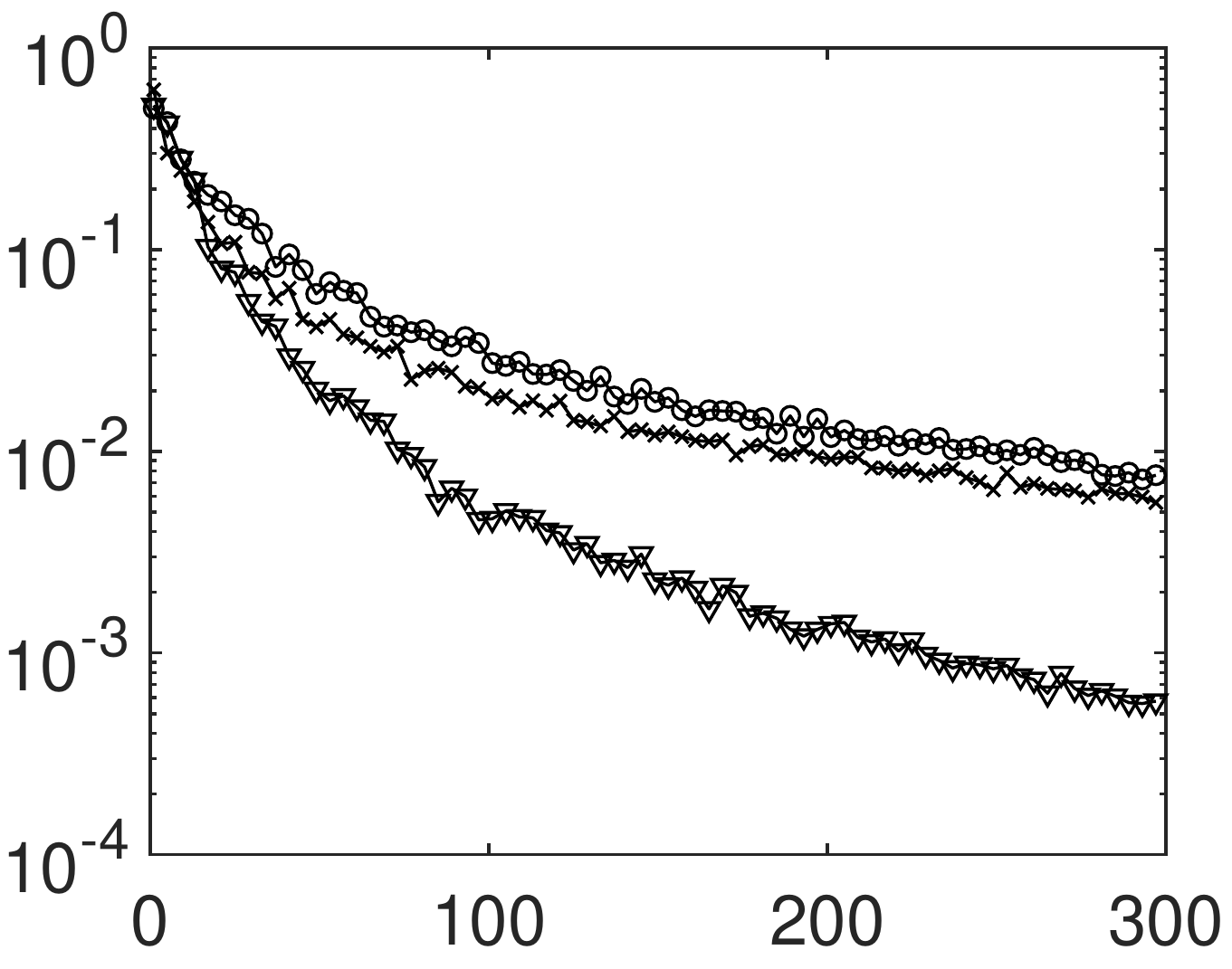}
  \includegraphics[width=0.25\textwidth]{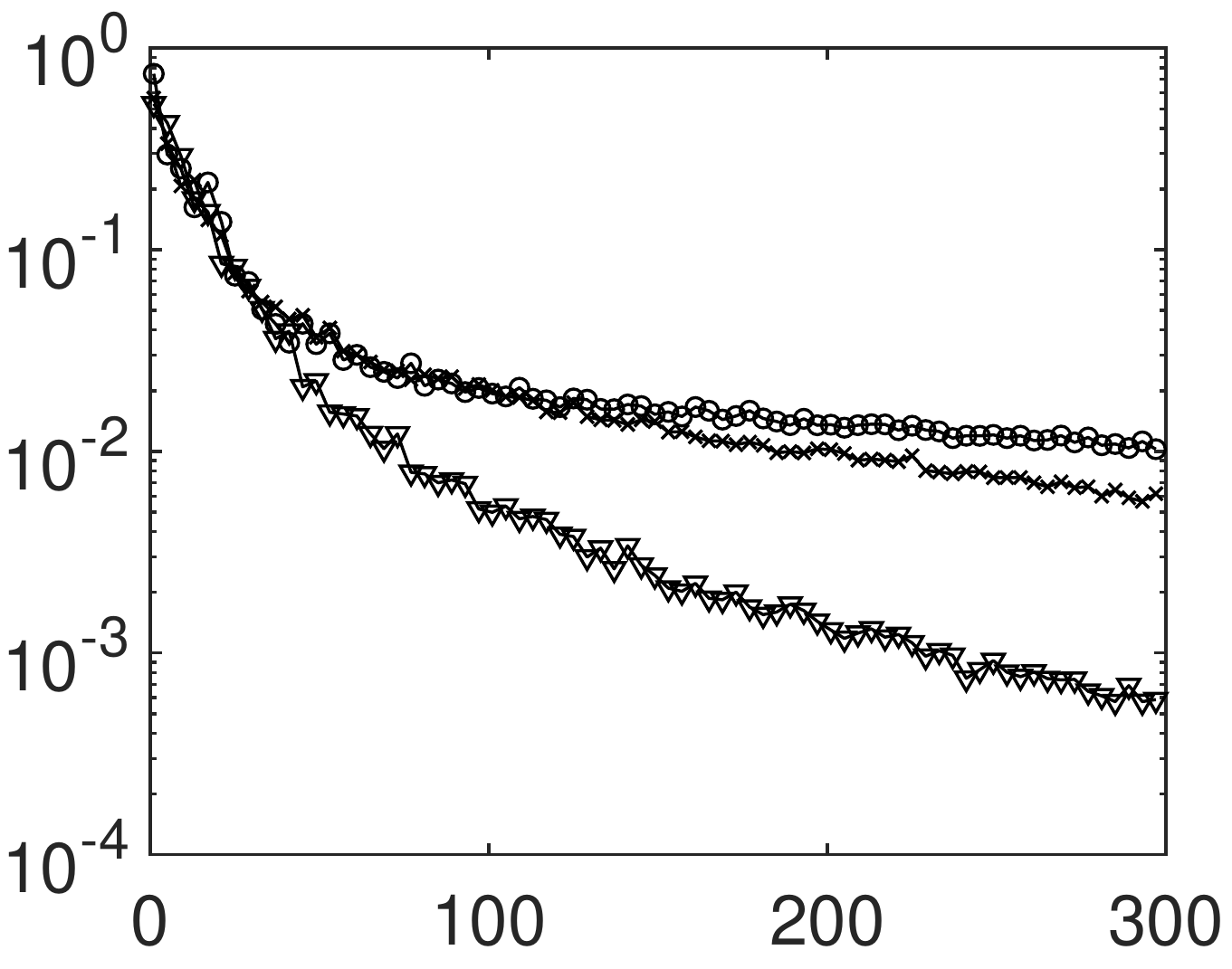}
  \caption{The plots in the first row show the reference solutions for the linear elliptic equation~\eqref{eqn:linear_elliptic} with $\eps=2^{0}$ (left), $2^{-2}$ (middle) and $2^{-4}$ (right) using the medium~\eqref{eqn:medium}. The plots on the second row present the decay of $L^2$-norm as the number of basis increases with $\eps=2^{0}$ (left), $2^{-2}$ (middle) and $2^{-4}$ (right). The plots on the third row present the decay of energy-norm with $\eps=2^{0}$ (left), $2^{-2}$ (middle) and $2^{-4}$ (right). We should note that the $L^2$ error achieves $10^{-4}$ with $300$ basis, in comparison to $3969$ degree of freedom.
  }
  \label{fig:ref_soln}
\end{figure}

\subsubsection{Radiative transport equation}\label{sec:RTE}

The multiscale radiative transport equation is defined as follows:
\begin{equation}\label{eqn:RTE}
\begin{cases}
v\cdot \nabla_x u + (\sigma_\mathrm{a}(x) + \sigma_\mathrm{s}(x)) u =  \sigma_\mathrm{s}(x) \mathcal{K} u(x,v)   + f(x,v),\quad& x\in\Omega\,, v\in\Sb^1\,, \\
u(x,v) = 0,\quad (x,v)\in\Gamma_-\,,
\end{cases}
\end{equation}
where $\Sb^1\subset\Rb^2$ denotes the unit circle on the plane. We denote the spatial variables by $x = (x_1,x_2)$ and parameterize the velocity vector by its angle, that is, $v = (\cos\theta,\sin\theta)\in\Sb^1$, $\theta\in[0,2\pi)$. The homogeneous boundary condition is imposed on the incoming part of the boundary:
\[
\Gamma_- = \{ (x,v)\in\partial\Omega\times\Sb^1: x\in\partial\Omega, \,\, -v\cdot n_x >0\}
\]
where $n_x$ is the exterior normal direction at $x\in\partial\Omega$. We define the scattering operator $\mathcal{K}$ to have the following form
\begin{equation}\label{eqn:scat_op}
    \mathcal{K} u(x,v) = \int_{\Sb^{1}} K(v,v') u(x,v') \rmd \mu(v'),
\end{equation}
where $\mu$ is the normalized spherical measure on $\Sb^{1}$. We use the Henyey-Greenstein phase function~\cite{Ar1999:optical,HeGr1941:diffuse} as the scattering kernel
\begin{equation}\label{eqn:scat_ker}
    K(v,v') = C \frac{1-g^2}{(1+g^2-2g v\cdot v')^{3/2}}\,,
\end{equation}
with $C$ being chosen to satisfy $\int_{\mathbb{S}^{1}} K(v,v')\rmd \mu(v') = 1$. We set $g = 0.5$ in all the following experiments.
We choose the scattering coefficient $\sigma_{\mathrm{s}}^{\eps}$ in the numerical experiments
\begin{equation} \label{eqn:scat_rte}
\sigma_{\mathrm{s}}(x_1,x_2) =
\frac{1}{\eps_1}\kappa(x_1,x_2;\eps_2)\,,
\end{equation}
where the function $\kappa$ is defined in~\eqref{eqn:medium} with $\eps$ replaced by $\eps_2$, and the absorption coefficient $\sigma_{\mathrm{a}}$
\begin{equation} \label{eqn:absp_rte}
\sigma_{\mathrm{a}}(x_1,x_2) =
\eps_1\left( 1 + \sin(4x_1^2x_2^2) + \frac{1.1+\cos(2\pi x_1/\eps_2)}{1.1+\cos(2\pi x_2/\eps_2)}  + \frac{1.1+\sin(\pi x_2/\eps_2)}{1.1+\sin(\pi x_1/\eps_2)} \right)\,.
\end{equation}
The absorption coefficient for different values of $\eps_2$ are plotted in \cref{fig:aborsption}. The small parameter $\eps_1$ is called the Knudsen number and $\eps_2$ is the small medium oscillation period. The parameter $(\eps_1,\eps_2)$ characterizes different regimes of the RTE~\eqref{eqn:RTE}. In general, the limiting regime as $(\eps_1,\eps_2)\to0$ can be taken through different routes, and it could end up at different limits~\cite{abdallah2012diffusion,goudon2003homogenization,goudon2001diffusion}. An advantage of the  method presented in this paper is that it can handle different regimes of RTE in a unified way.

\begin{figure}[htbp]
  \centering
  \includegraphics[width=0.3\textwidth]{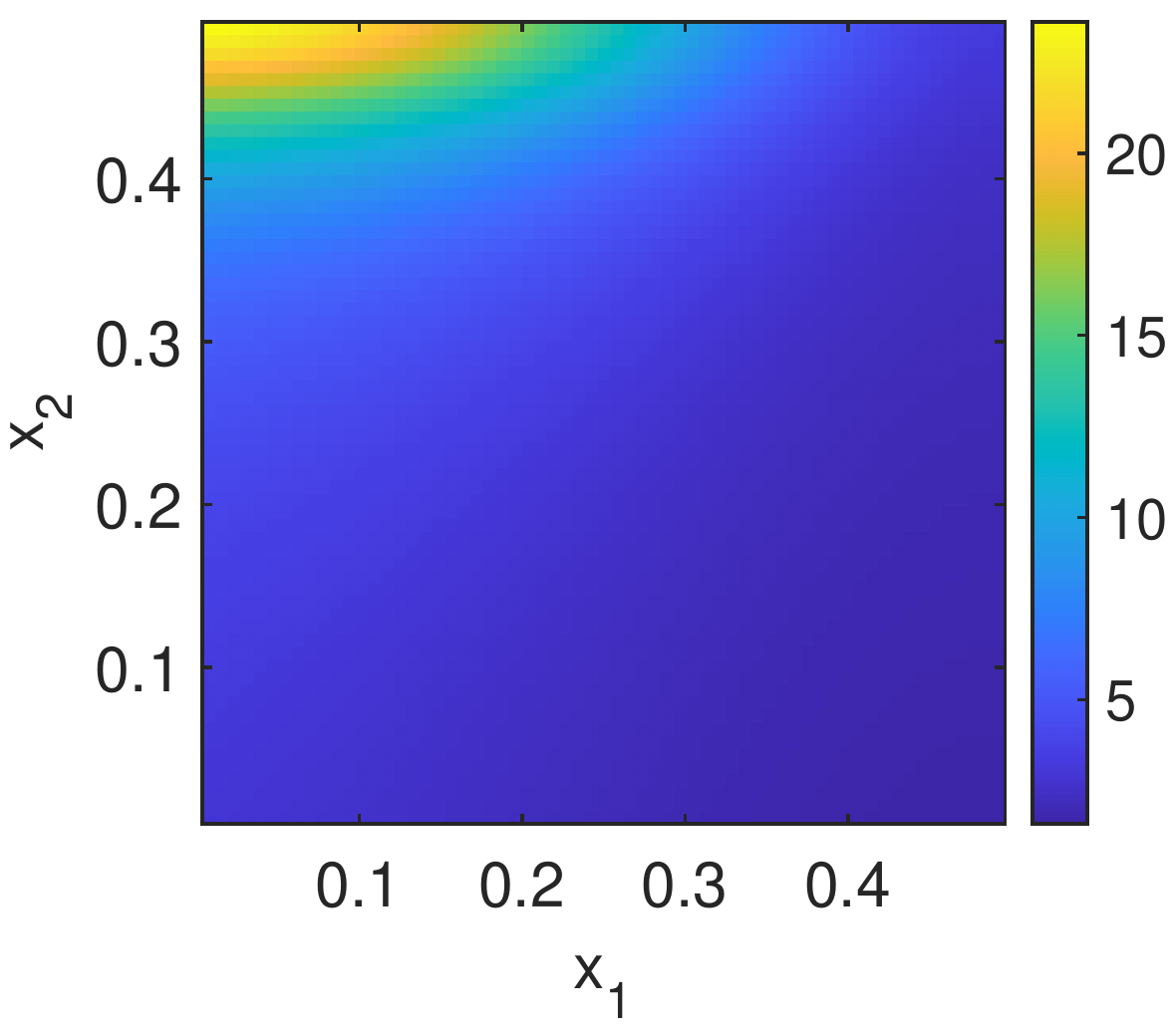}
  \includegraphics[width=0.3\textwidth]{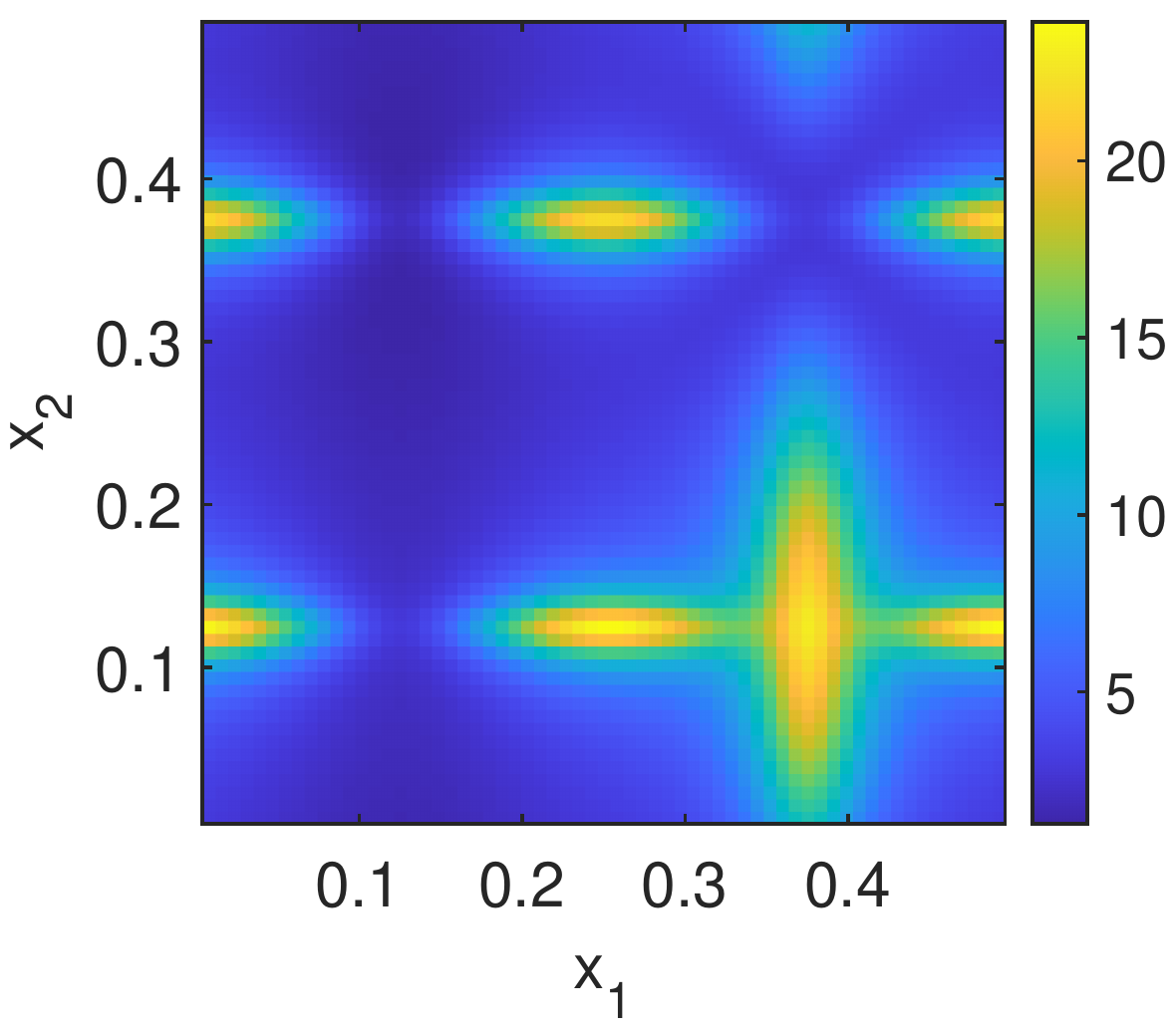}
  \includegraphics[width=0.3\textwidth]{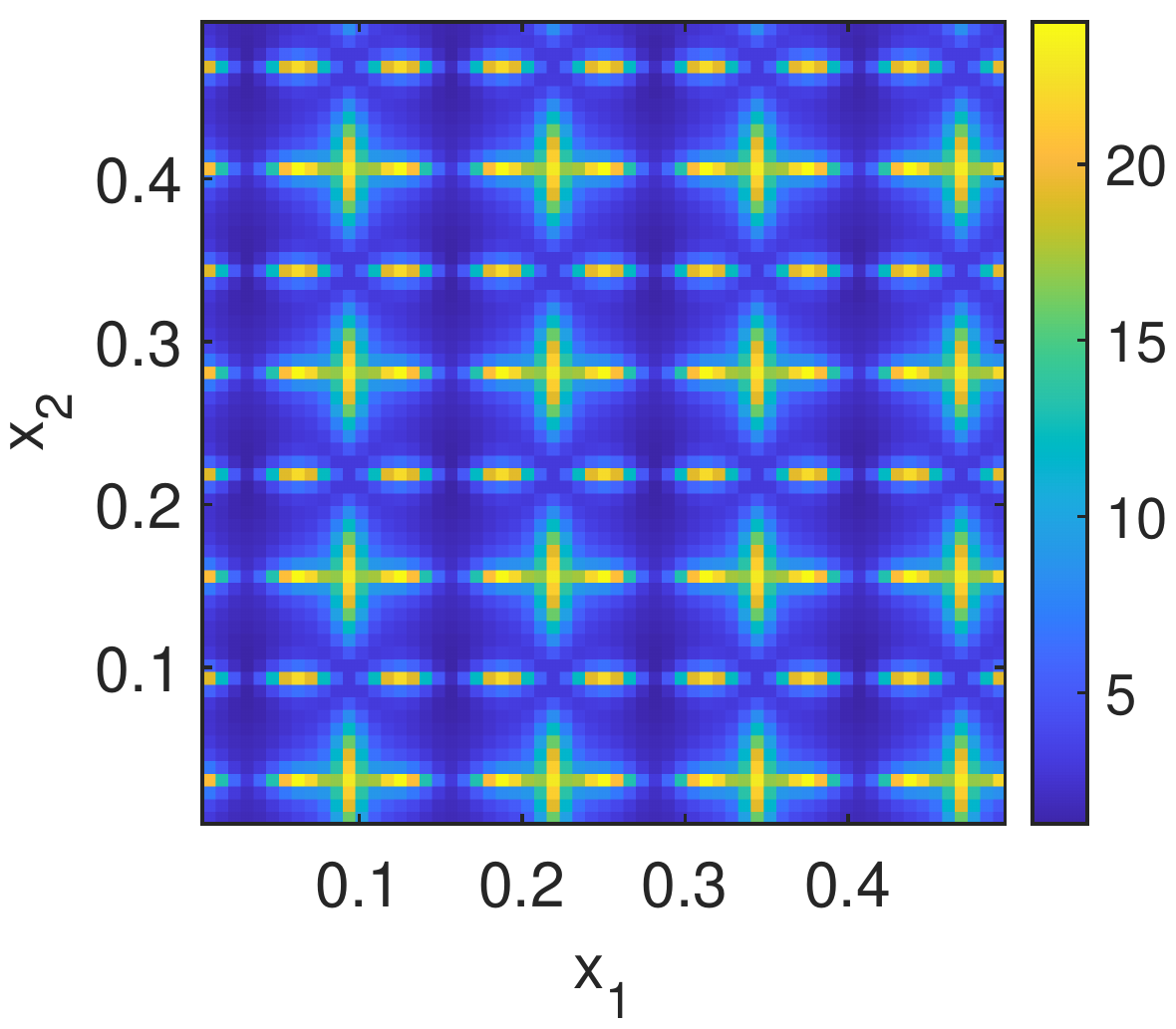}
  \caption{Absorption coefficient $\sigma_{\mathrm{a}}$~\eqref{eqn:absp_rte} for $\eps_2 = 1$(left), $2^{-2}$ (middle), $2^{-4}$ (right). We set $\eps_1 = 1$ in the plots.}
  \label{fig:aborsption}
\end{figure}

The reference solution and the global basis are both computed using the up-wind scheme for the transport term and the trapezoidal rule for the scattering operator. The spatial mesh size is $h = 2^{-6} = \tfrac{1}{64}$ so that $\eps_i$ are resolved. This leads to $N_x=63$ in $x_1$- and $x_2$-direction. The number of quadrature points is $N_v = 40$. The total Degrees Of Freedom equals $63\times 63\times 40 = 158760$ in the discrete equation.

It is shown in~\cite{AgAg:1998boundary} that the smoothness of the source function in the velocity direction does not lead to higher regularity of the RTE solution, so $L_2$ norm is chosen for $v$. Therefore,
the weight matrix $\mathsf{\Pi}_\mathsf{X}$ is chosen to that takes the spatial derivatives into account but places only constant weights on the velocity domain.
Denoting $\mathsf{f} = [\mathsf{f}_{ijl}]_{ijl}$, where $i,j$ denote the indices in the spatial direction and $l$ denotes the index in the velocity direction, we will place discrete Sobolev $(p,0)$-norm on $\mathsf{f}$, with  $p$ denoting the order of spatial derivative and $\mathsf{\Pi}_{\mathsf{X}}$ adjusted accordingly; see Appendix~\ref{appendix:Pi}.

Equipped with different $(p,0)$-norm, the decay rate of singular
values for the Green's matrix changes. In \cref{fig:RTE_svd} we
show the relative singular values for different choices of $p$,
$\eps_1$, and $\eps_2$. Similar to the observation in the previous
example, the decay rate of singular values is insensitive to the
choice of the medium. Various choices of values for $\eps_2$ lead to
almost identical decay rates. At the same time, different $p$-values
reflects the smoothness of the source term. As $p$ increases, the
decay becomes faster. In \cref{fig:RTE_left_svectors_1}, we plot
the most dominant basis (left singular vector) for various of choices
of $(\eps_1,\eps_2)$, where we set $(p,0)=(1,0)$.

\begin{figure}[htbp]
  \centering
  \includegraphics[width=0.3\textwidth]{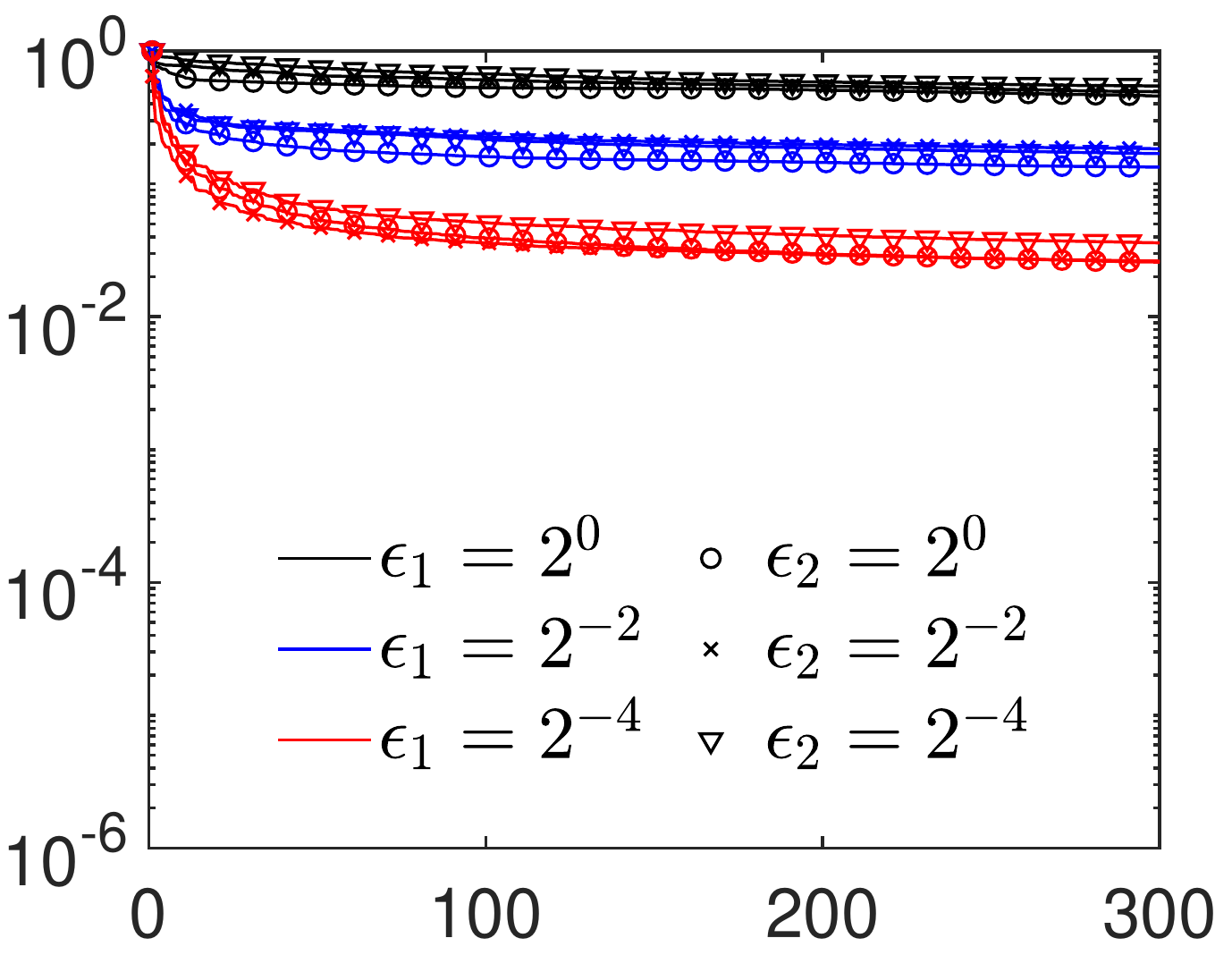}
  \includegraphics[width=0.3\textwidth]{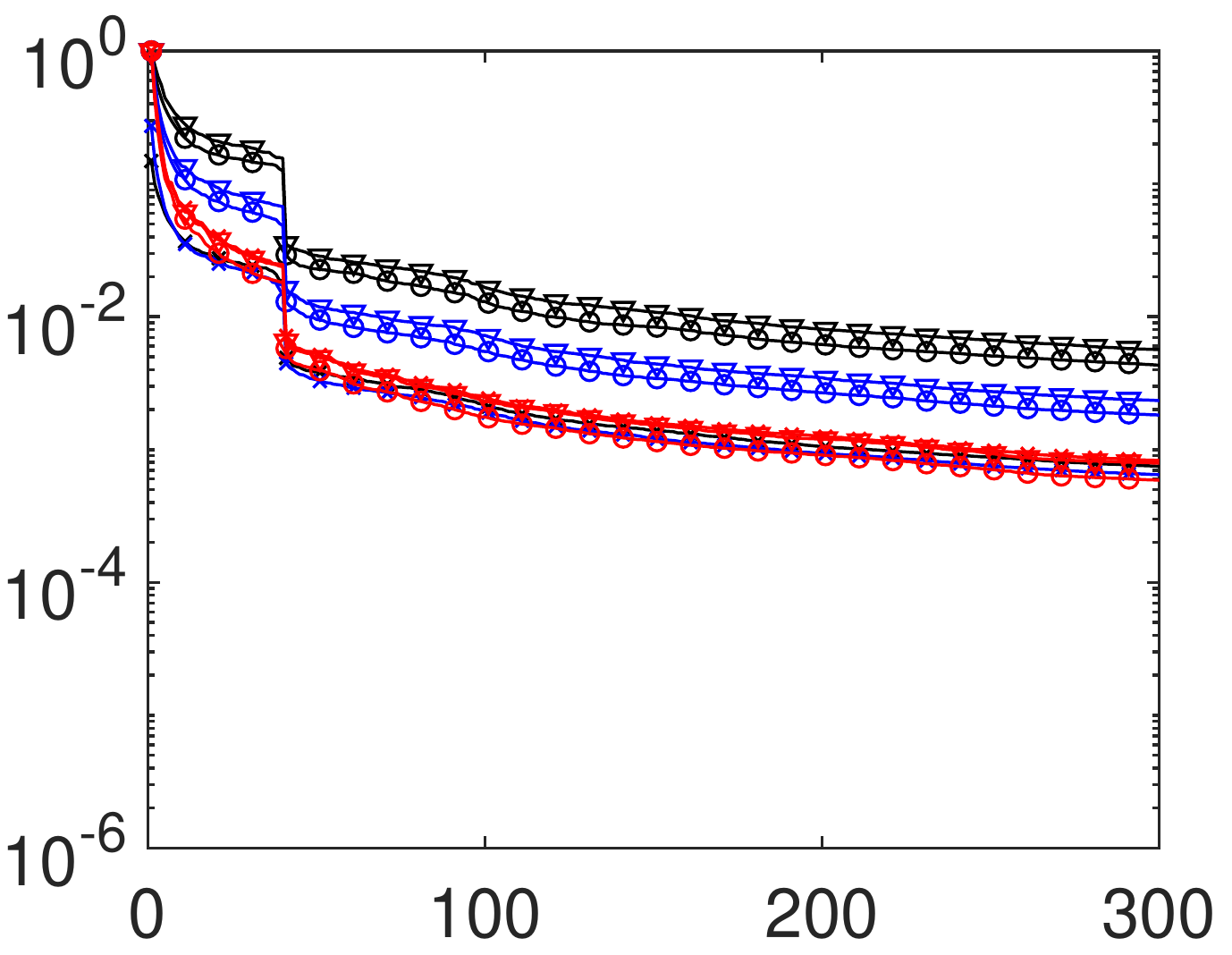}
  \includegraphics[width=0.3\textwidth]{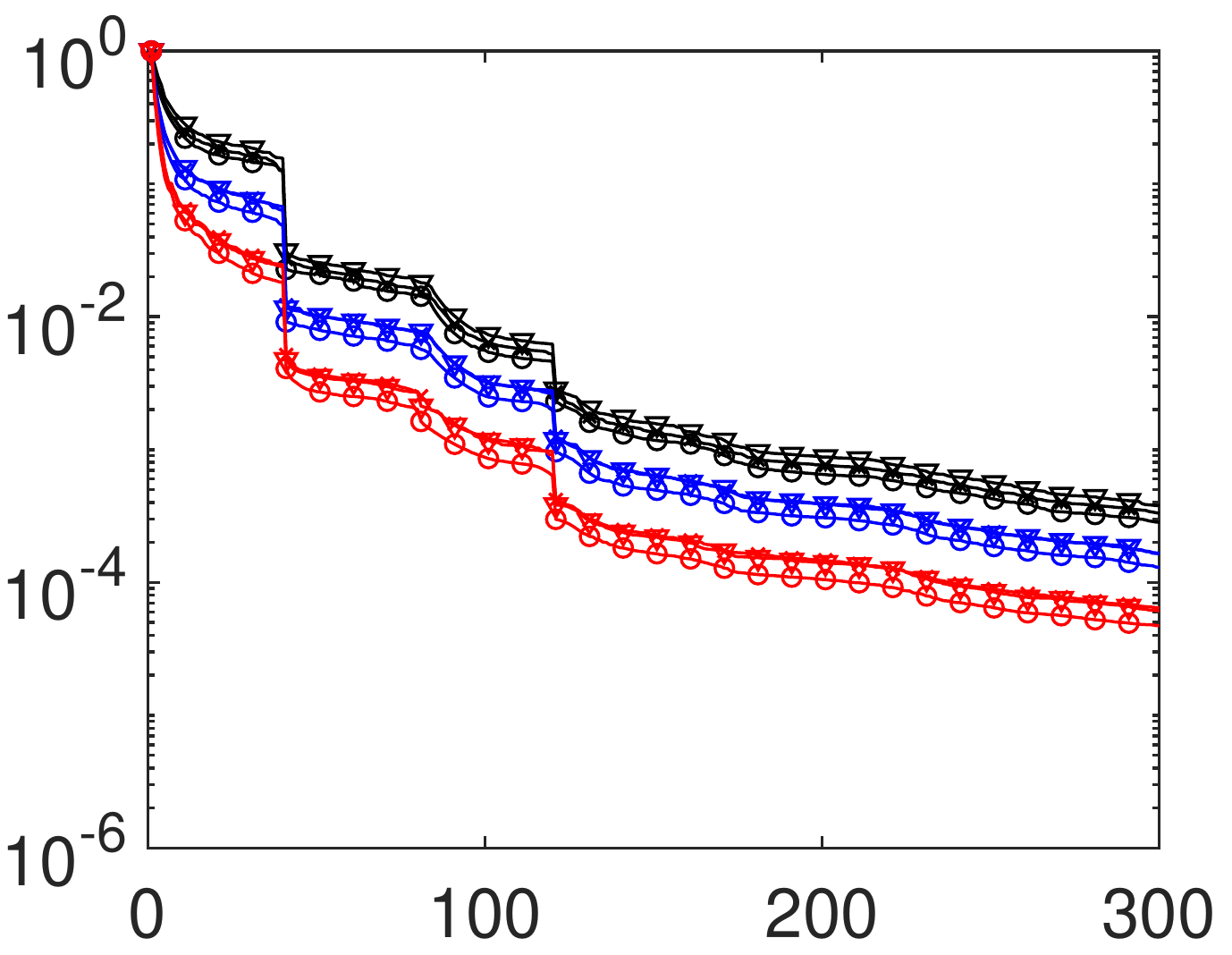}
  \caption{Singular values of the discrete Green's matrix $\mathsf{G}$ as an operator from $H^{p,0}$ to $L^2$ for $p=0$ (left), $p=1$ (middle), $p=2$ (right). The Green's matrix is computed from the RTE~\eqref{eqn:RTE} with medium $\sigma^\eps$ defined in~\eqref{eqn:scat_rte} and~\eqref{eqn:absp_rte}
  }
  \label{fig:RTE_svd}
\end{figure}

\begin{figure}[htbp]
  \centering
  \includegraphics[width=0.2\textwidth, height = 0.1\textheight]{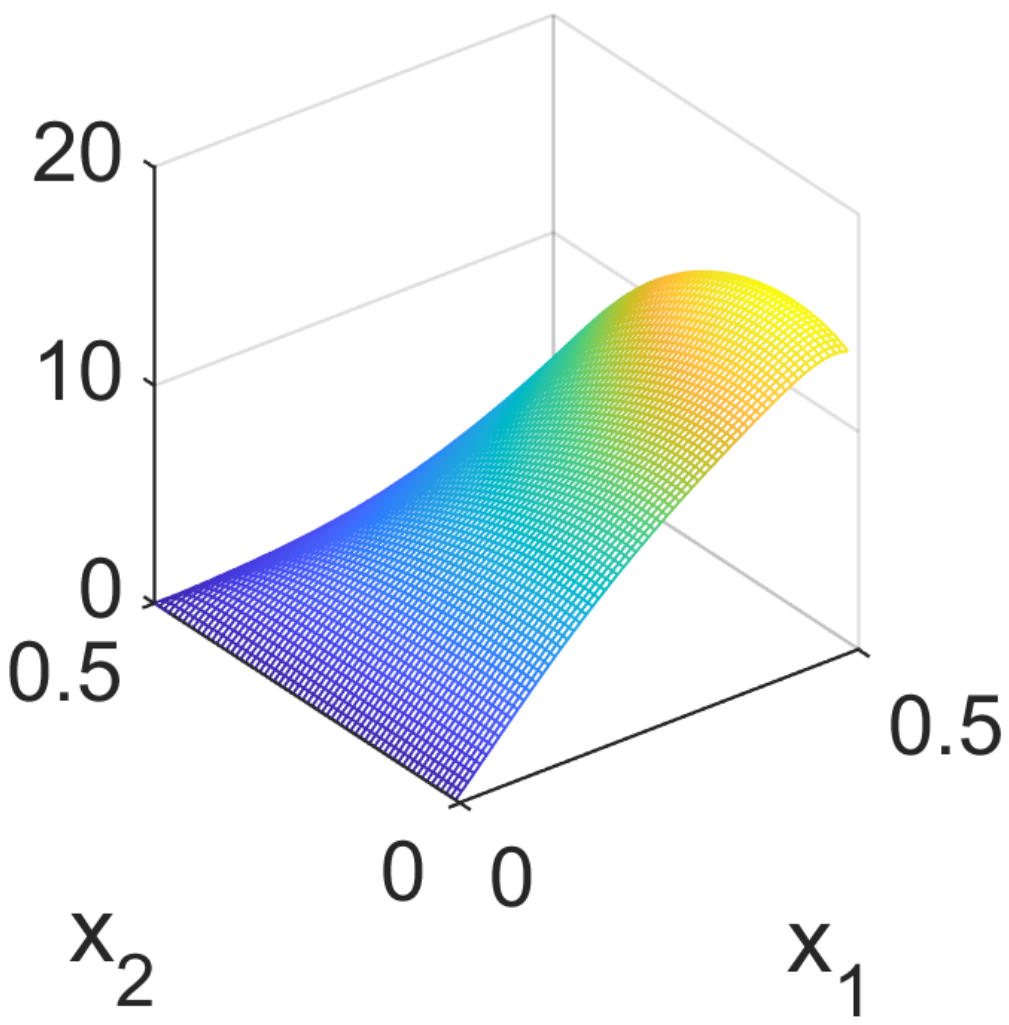}
  \includegraphics[width=0.2\textwidth, height = 0.1\textheight]{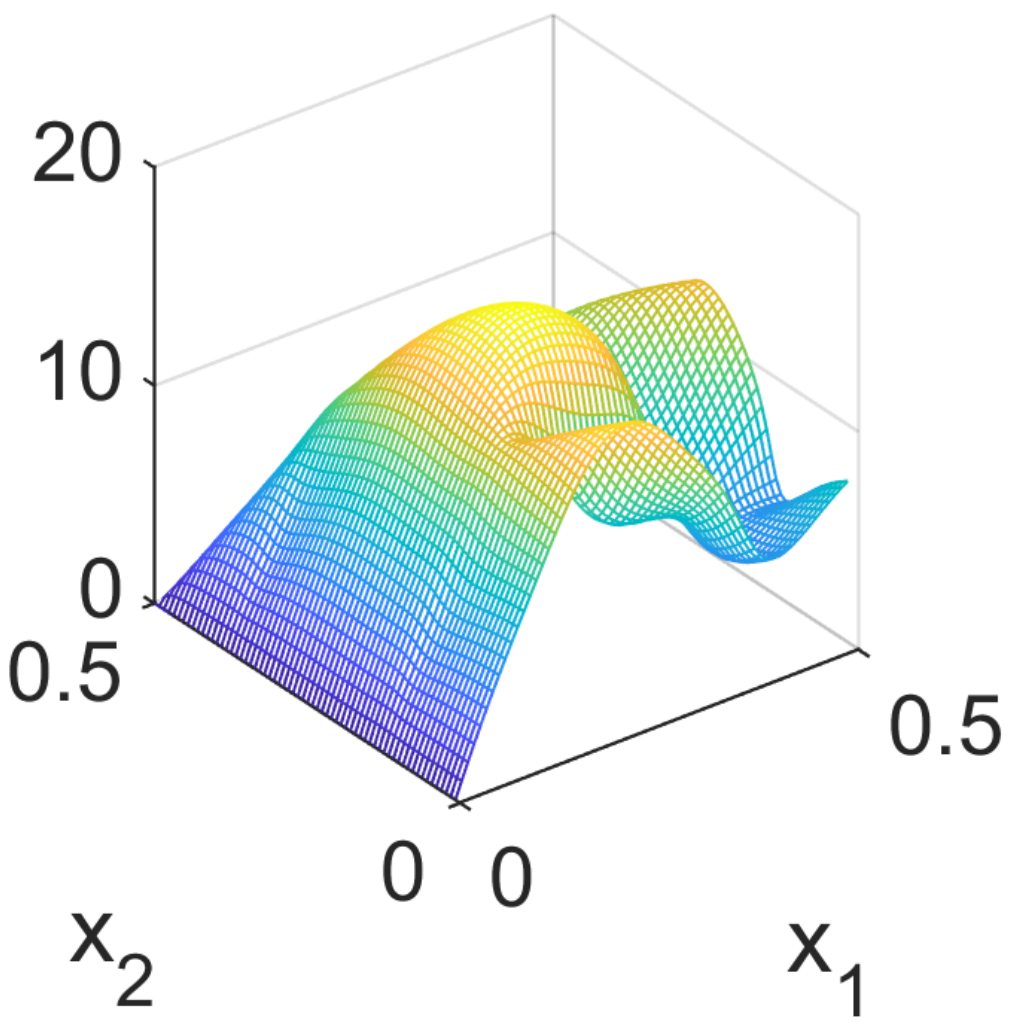}
  \includegraphics[width=0.2\textwidth, height = 0.1\textheight]{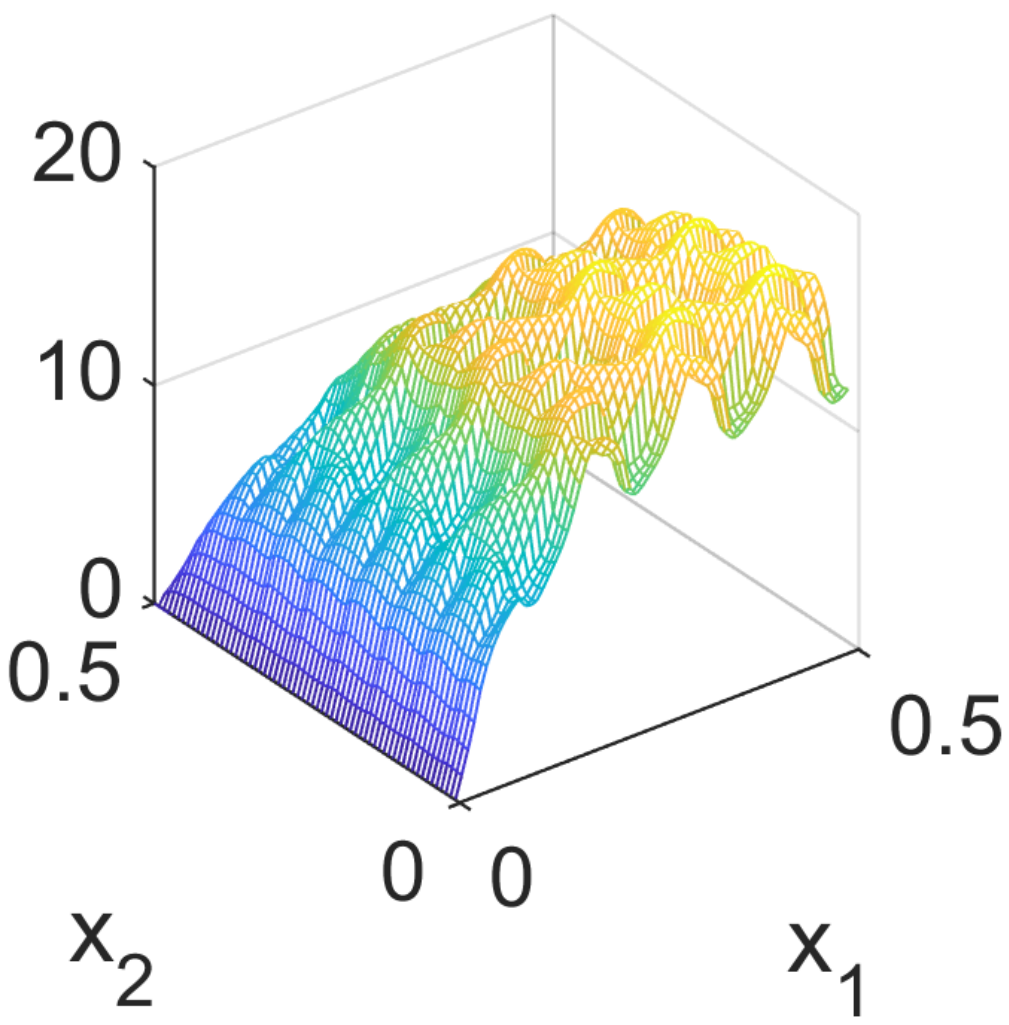}
  \\
  \includegraphics[width=0.2\textwidth, height = 0.1\textheight]{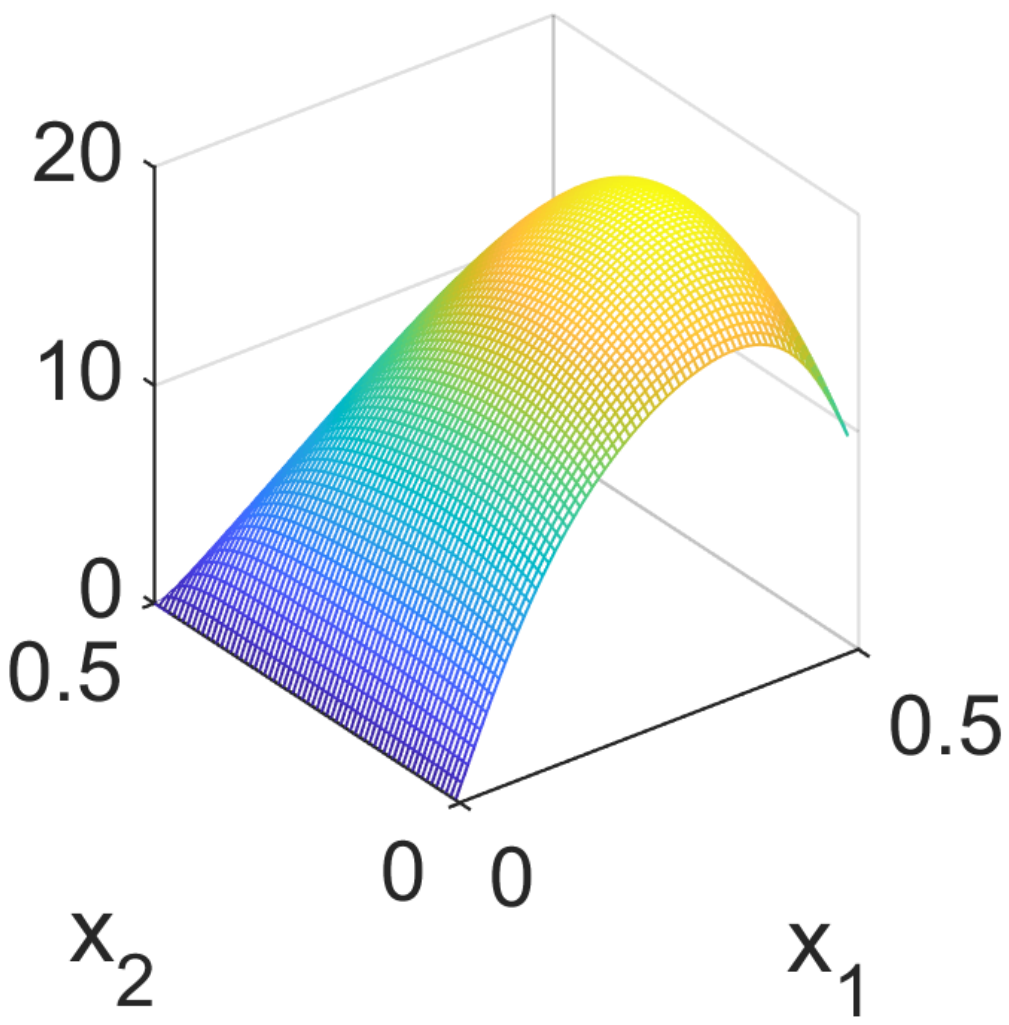}
  \includegraphics[width=0.2\textwidth, height = 0.1\textheight]{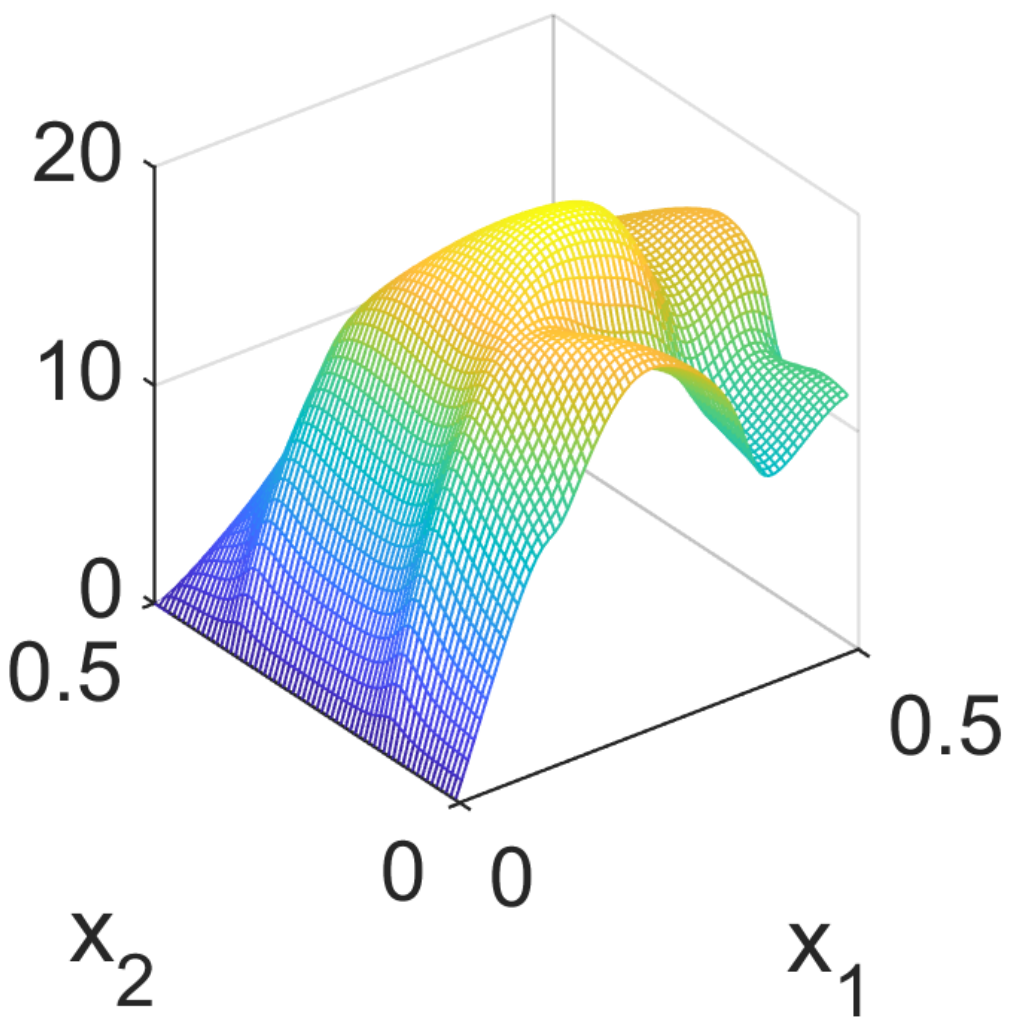}
  \includegraphics[width=0.2\textwidth, height = 0.1\textheight]{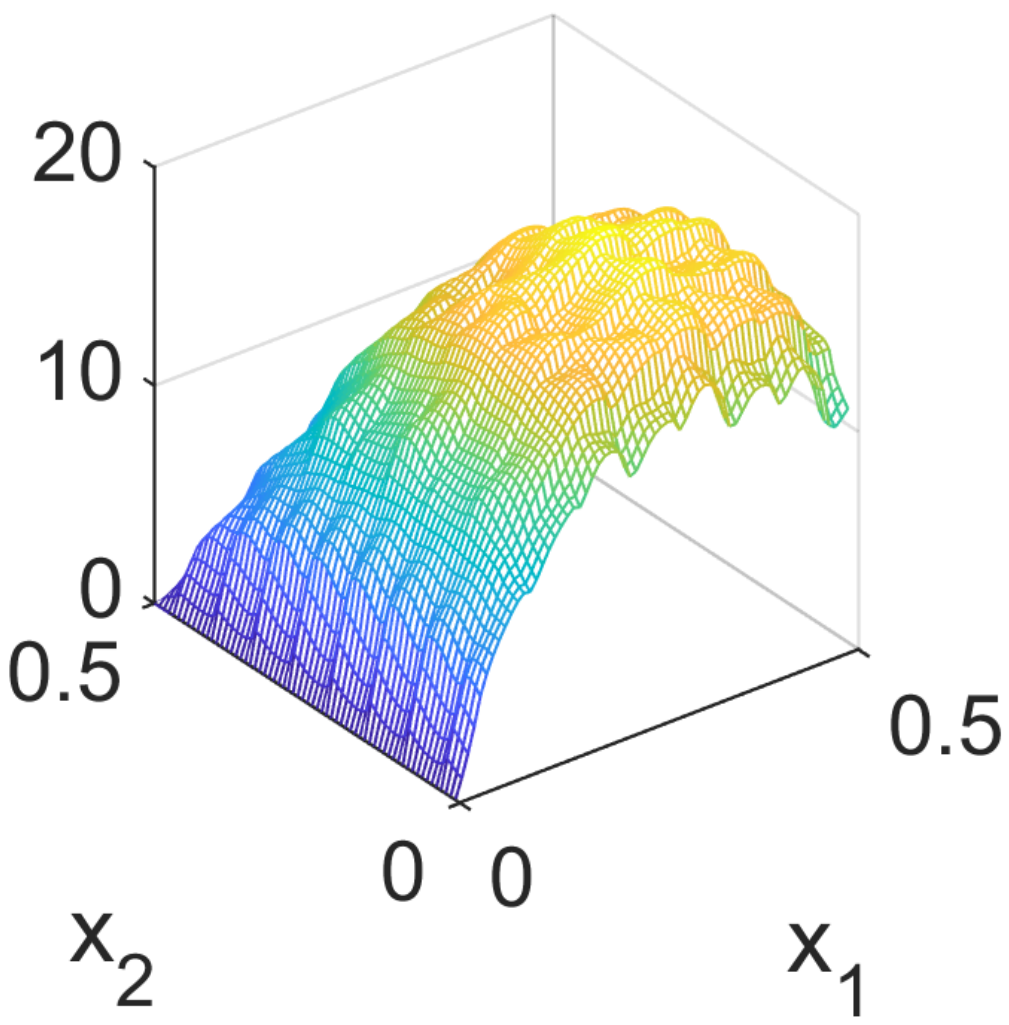}
  \\
  \includegraphics[width=0.2\textwidth, height = 0.1\textheight]{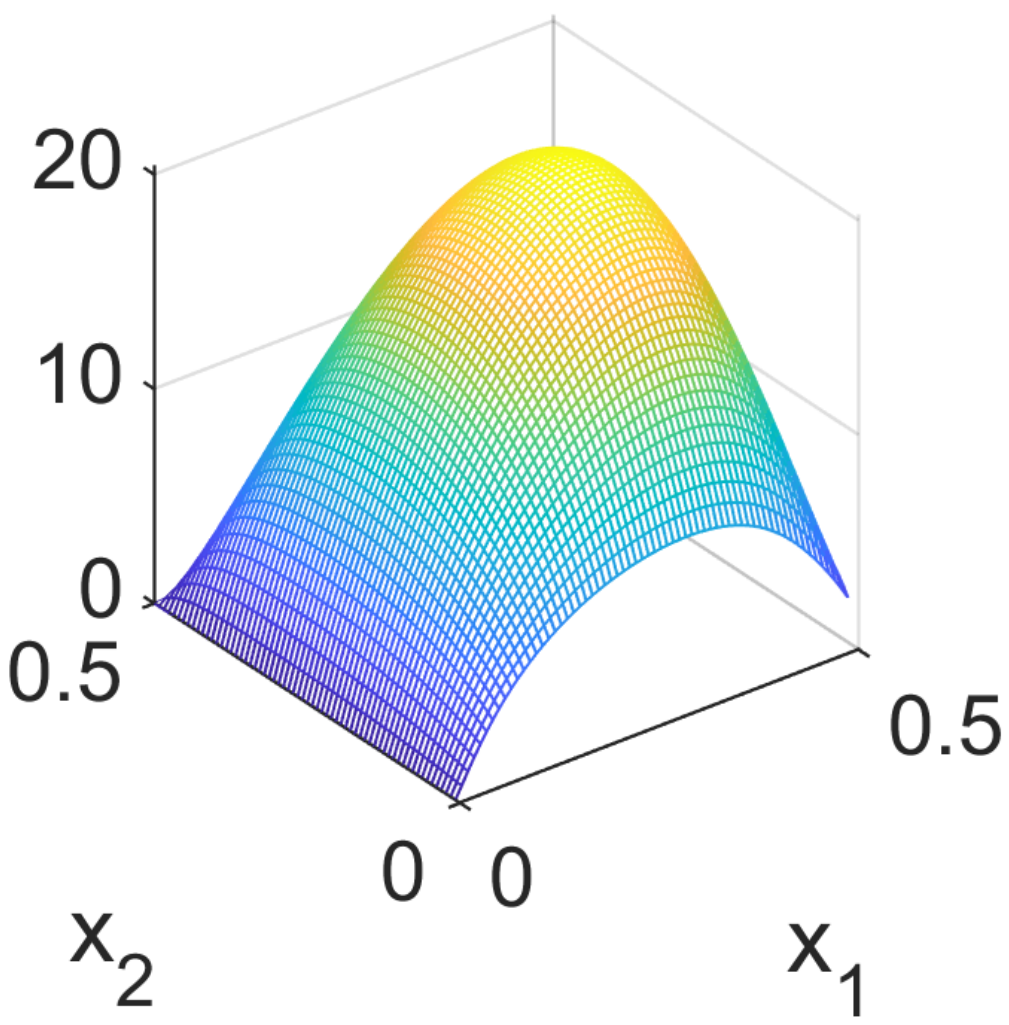}
  \includegraphics[width=0.2\textwidth, height = 0.1\textheight]{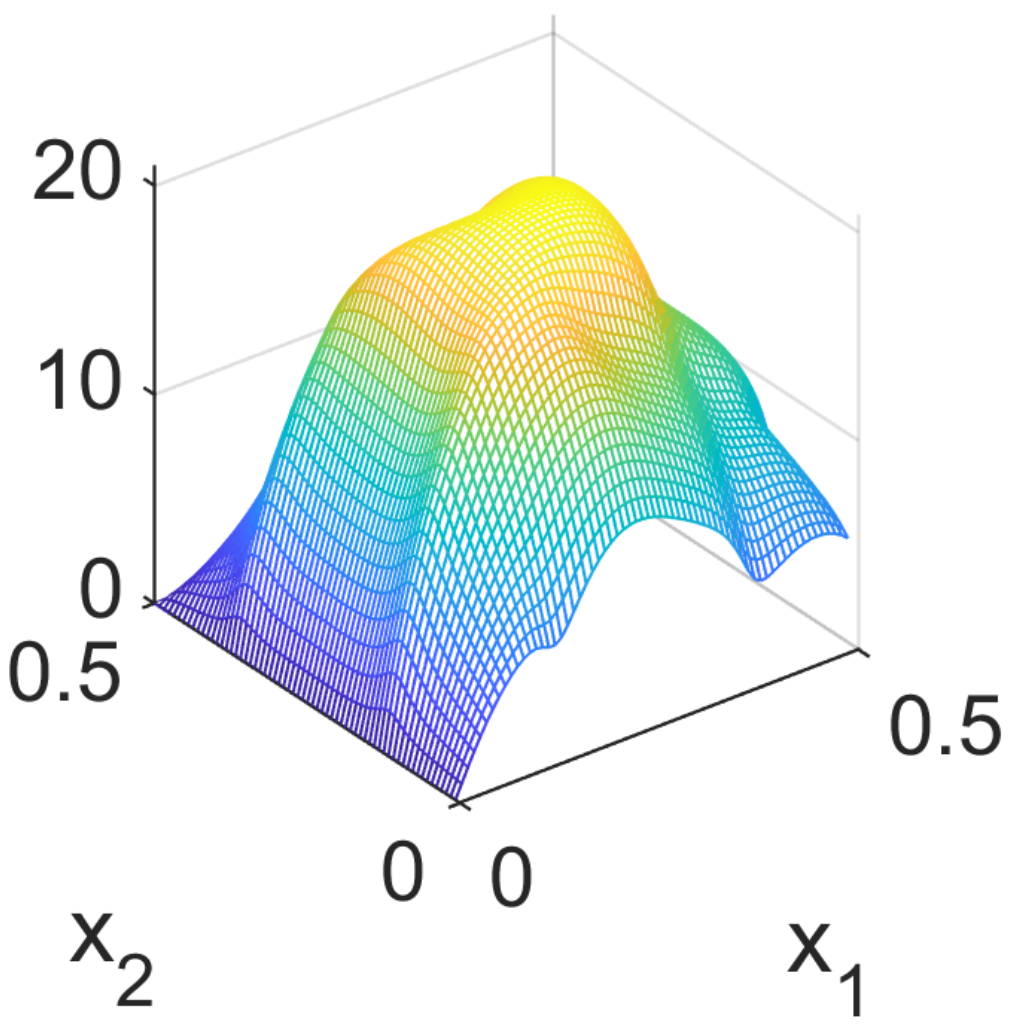}
  \includegraphics[width=0.2\textwidth, height = 0.1\textheight]{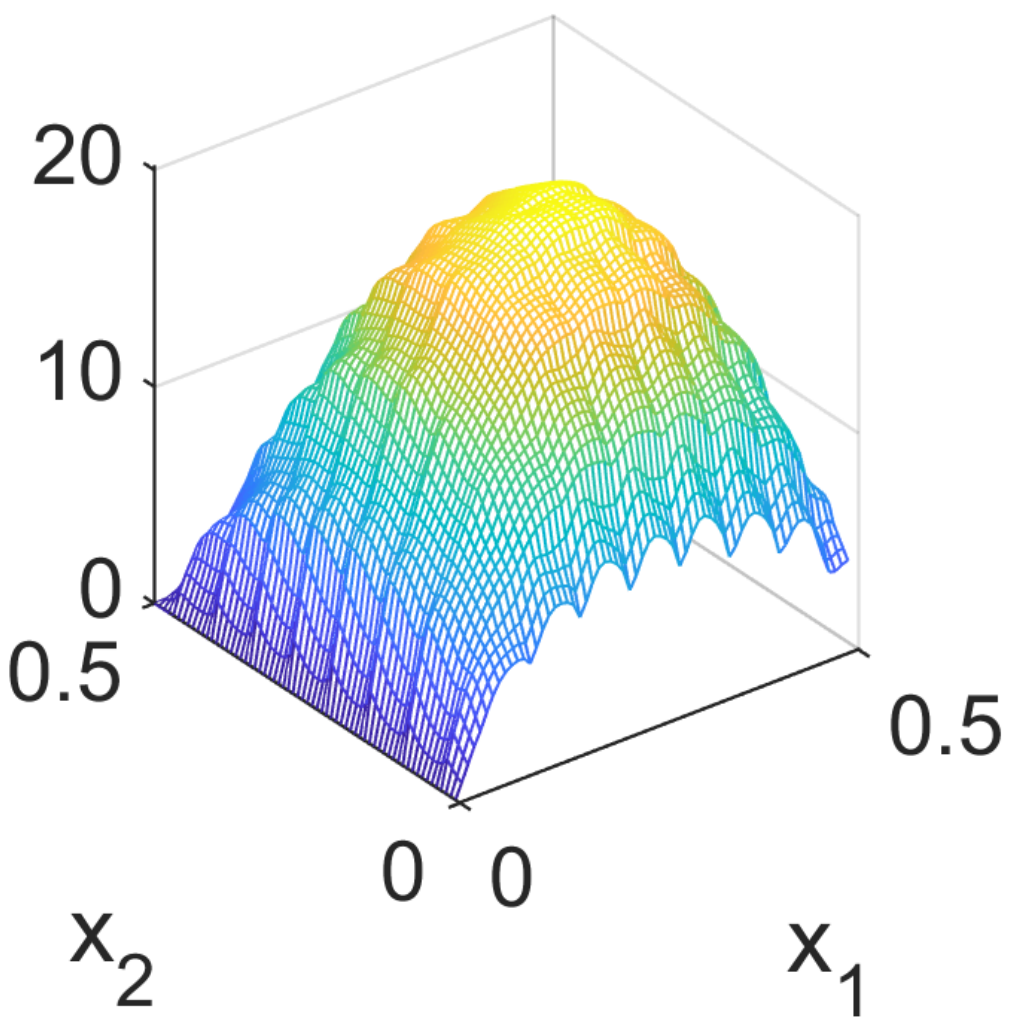}
  \caption{The first left singular vectors $\mathsf{\hat{u}}$ at $\theta=1.75\pi$ of the Green's matrix $\mathsf{G}$ computed from the rSVD for RTE~\eqref{eqn:RTE} with medium $\sigma$ defined in~\eqref{eqn:scat_rte} and~\eqref{eqn:absp_rte} for different $\eps$'s. The columns from left to right show, respectively, $\eps_2=2^{0}$, $\eps_2=2^{-2}$ and $\eps_2 = 2^{-4}$. The rows from top to bottom show, respectively, $\eps_1=2^{0}$, $\eps_1=2^{-2}$ and $\eps_1=2^{-4}$.
  }
  \label{fig:RTE_left_svectors_1}
\end{figure}

Next, we utilize these optimal basis functions to solve the
RTE~\eqref{eqn:linear_elliptic} and show the online test
performance. We choose
\begin{equation} \label{def.f}
f(x_1,x_2,\theta)=\exp\left( -\frac{(x_1-L/2)^2 +
  (x_2-L/2)^2}{(L/4)^2} - \frac{(\cos\theta-1)^2 +
  \sin^2\theta}{0.2^2} \right)
\end{equation}
and plot in~\cref{fig:RTE_ref_soln_fgauss} the reference solutions of
different sources and media. In \cref{fig:RTE_err}, we show the
relative $L^2$ error as a function of the number of bases for
different media and smoothness index $p$ of the source space. As in
the case above, higher regularity brings better convergence in the
error, and $\eps_2$ has limited influence on the error decay while
smaller $\eps_1$ leads to a faster decay rate.

\begin{figure}[htbp]
  \centering
  \includegraphics[width=0.2\textwidth, height = 0.1\textheight]{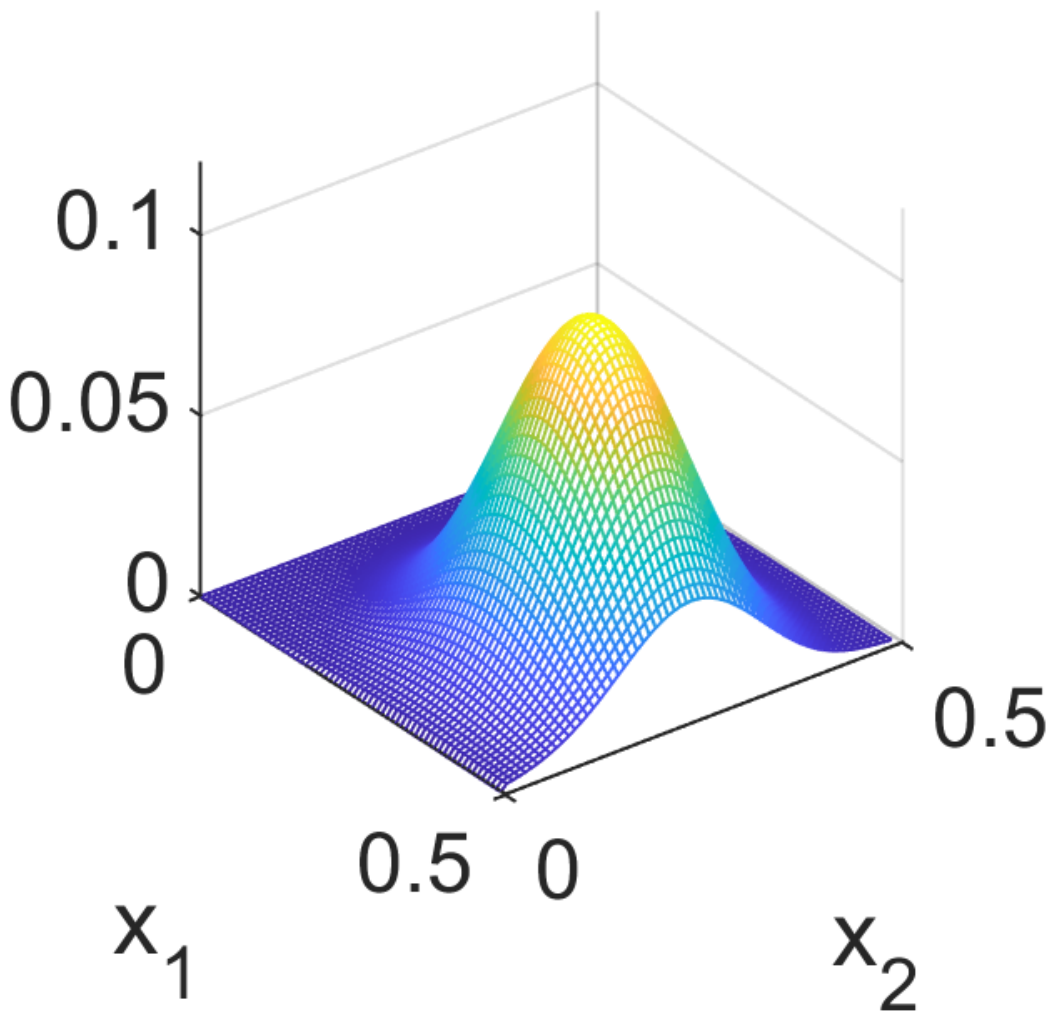}
  \includegraphics[width=0.2\textwidth, height = 0.1\textheight]{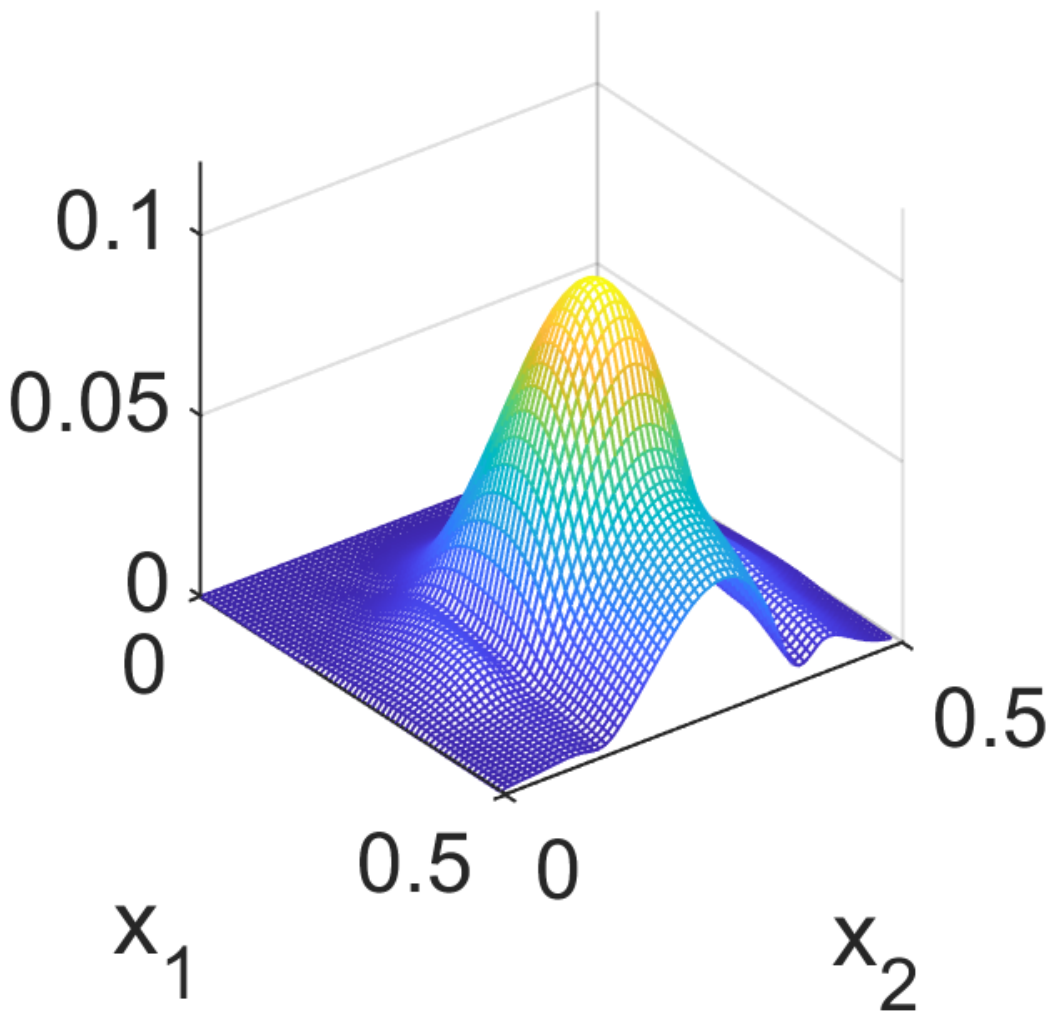}
  \includegraphics[width=0.2\textwidth, height = 0.1\textheight]{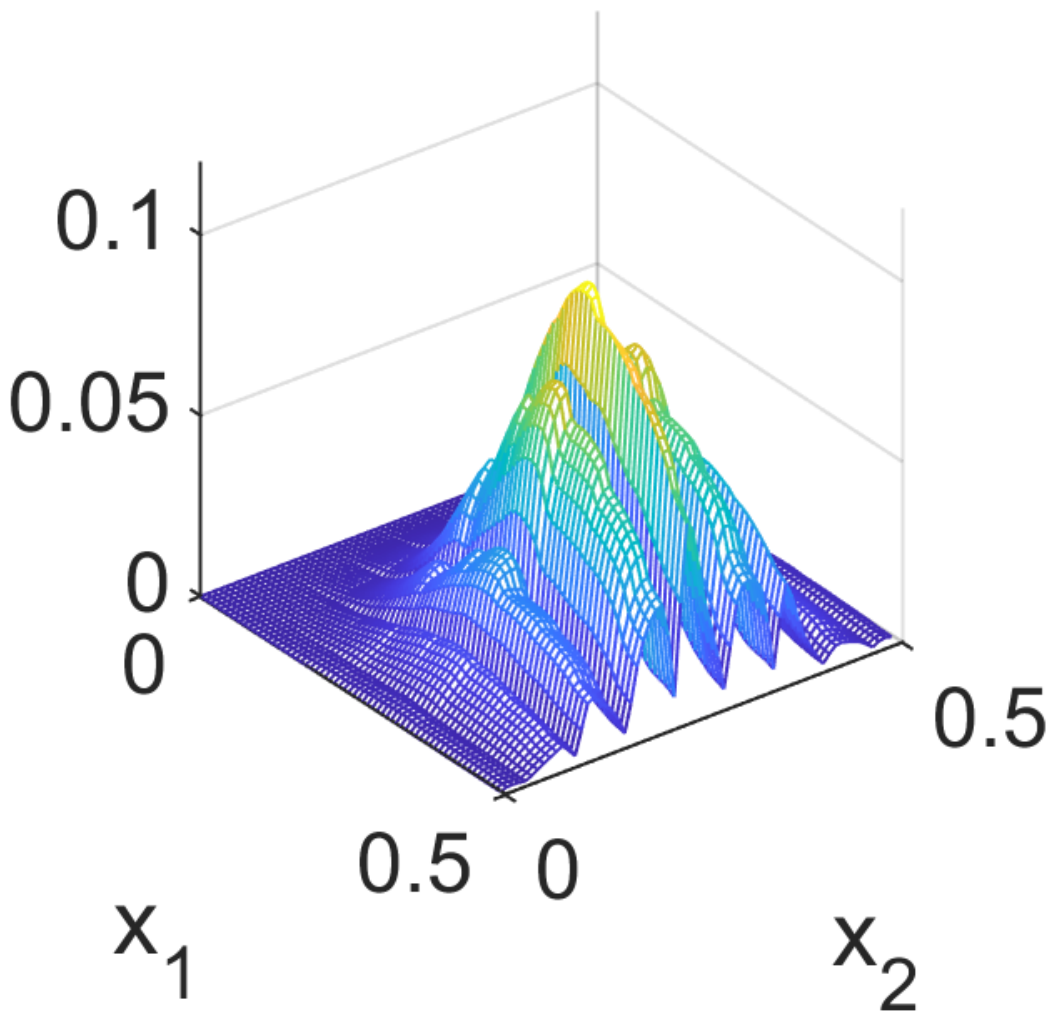}

  \includegraphics[width=0.2\textwidth, height = 0.1\textheight]{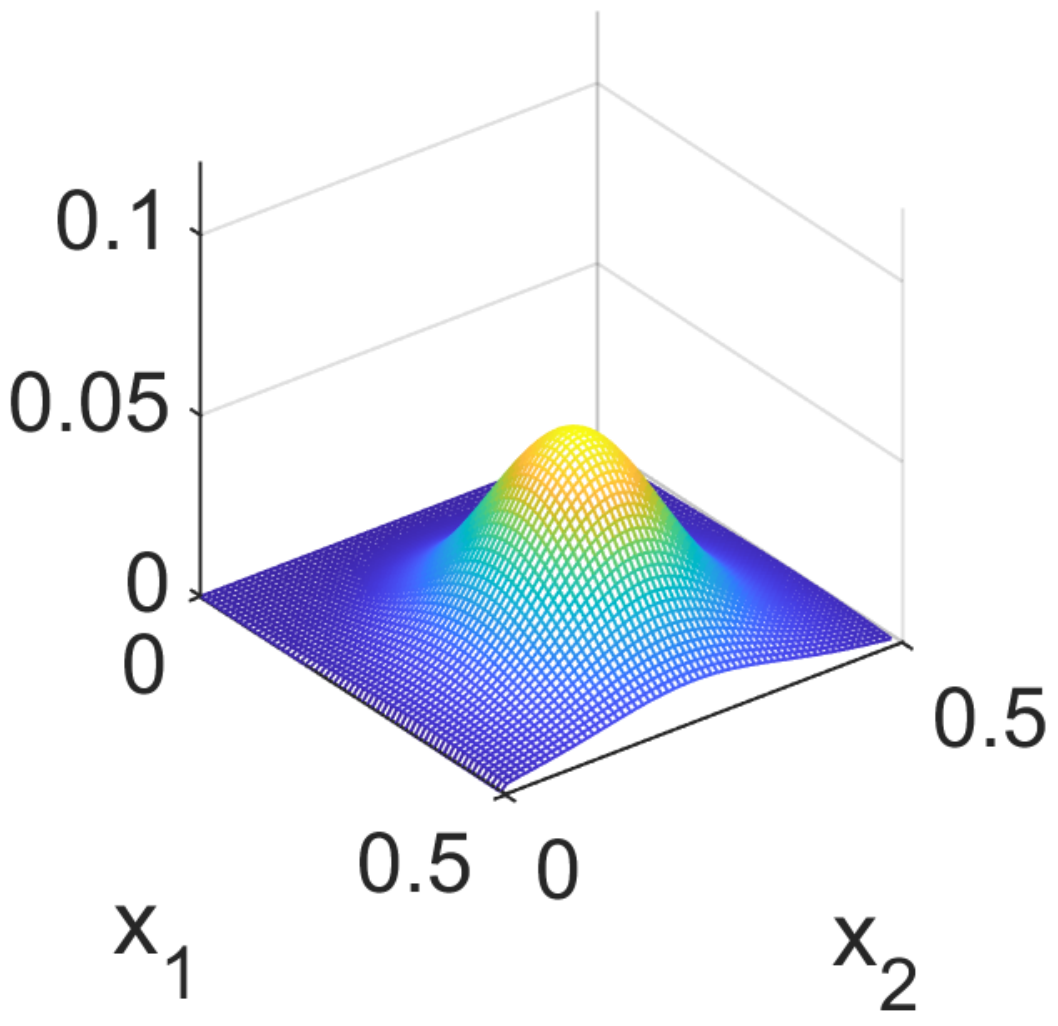}
  \includegraphics[width=0.2\textwidth, height = 0.1\textheight]{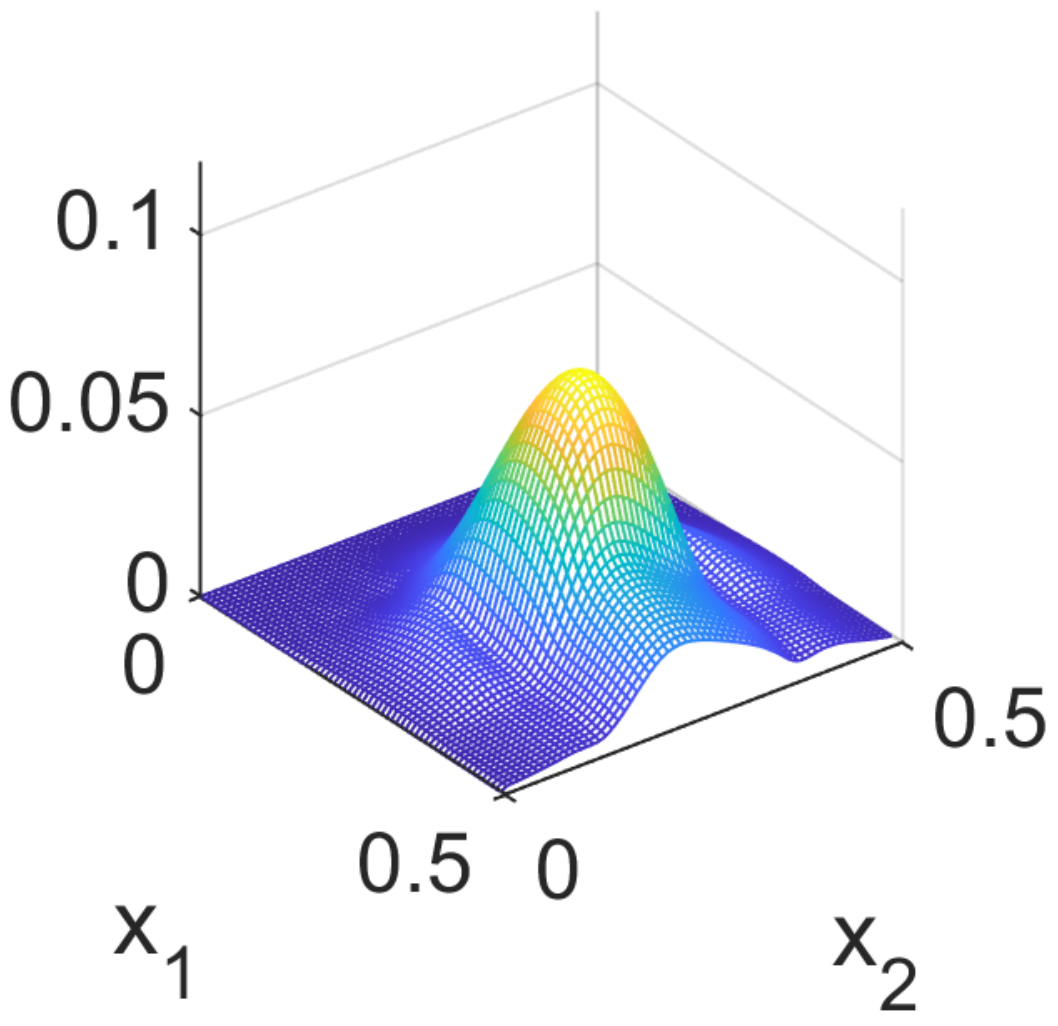}
  \includegraphics[width=0.2\textwidth, height = 0.1\textheight]{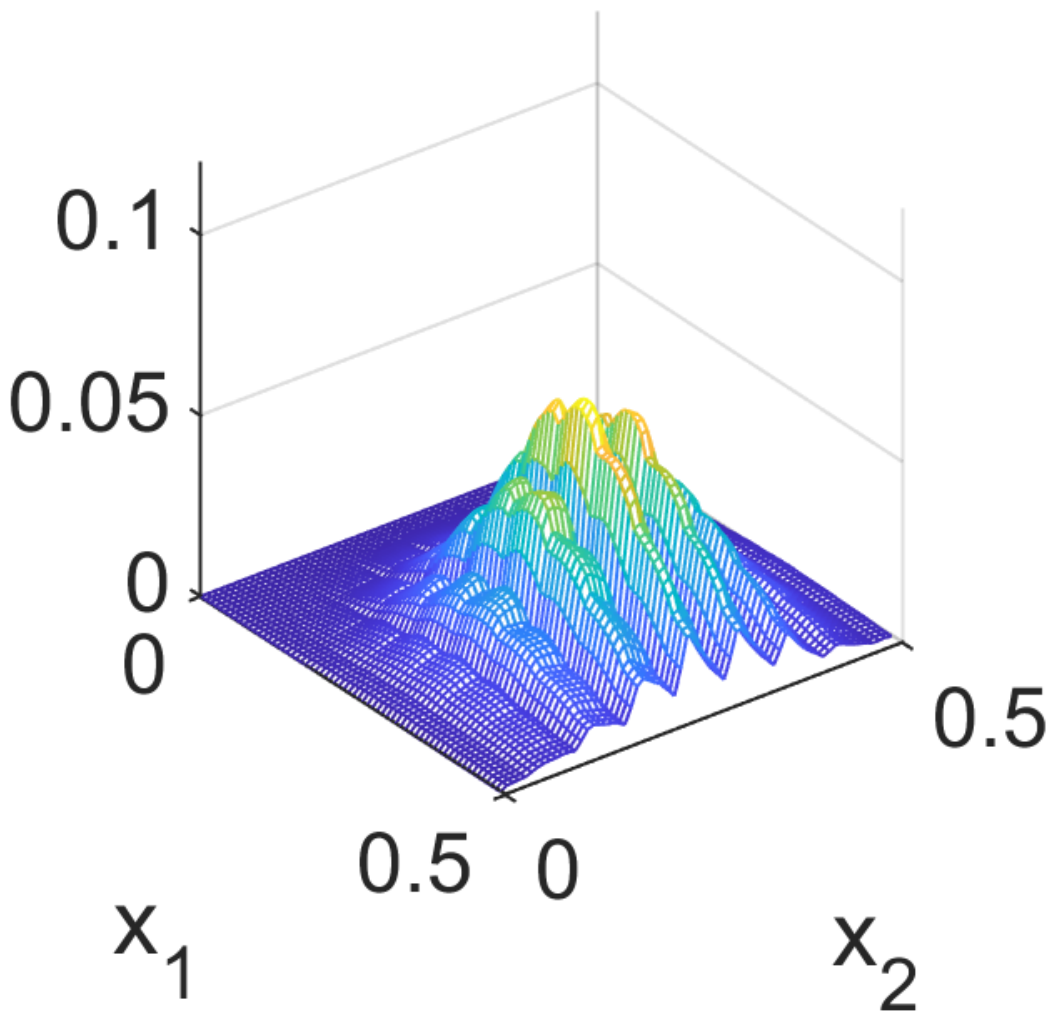}

  \includegraphics[width=0.2\textwidth, height = 0.1\textheight]{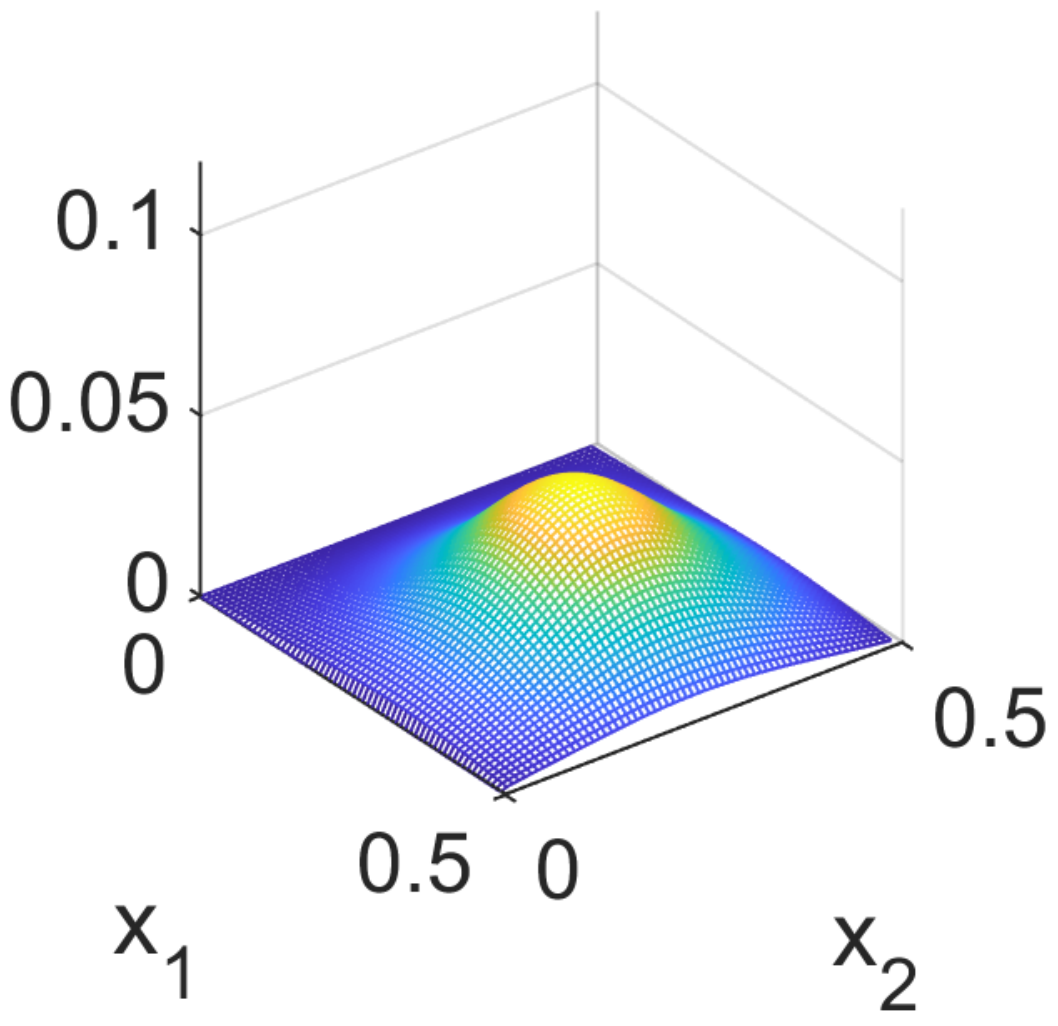}
  \includegraphics[width=0.2\textwidth, height = 0.1\textheight]{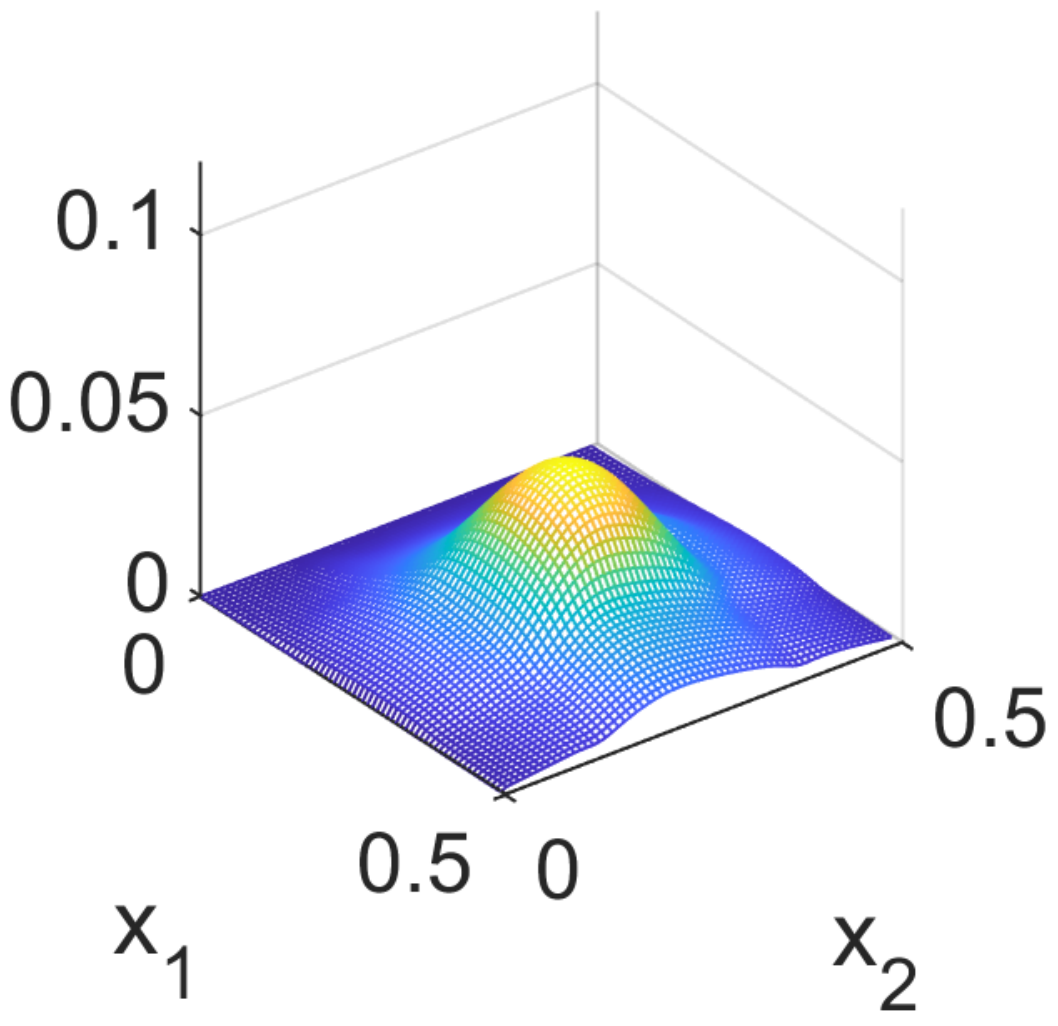}
  \includegraphics[width=0.2\textwidth, height = 0.1\textheight]{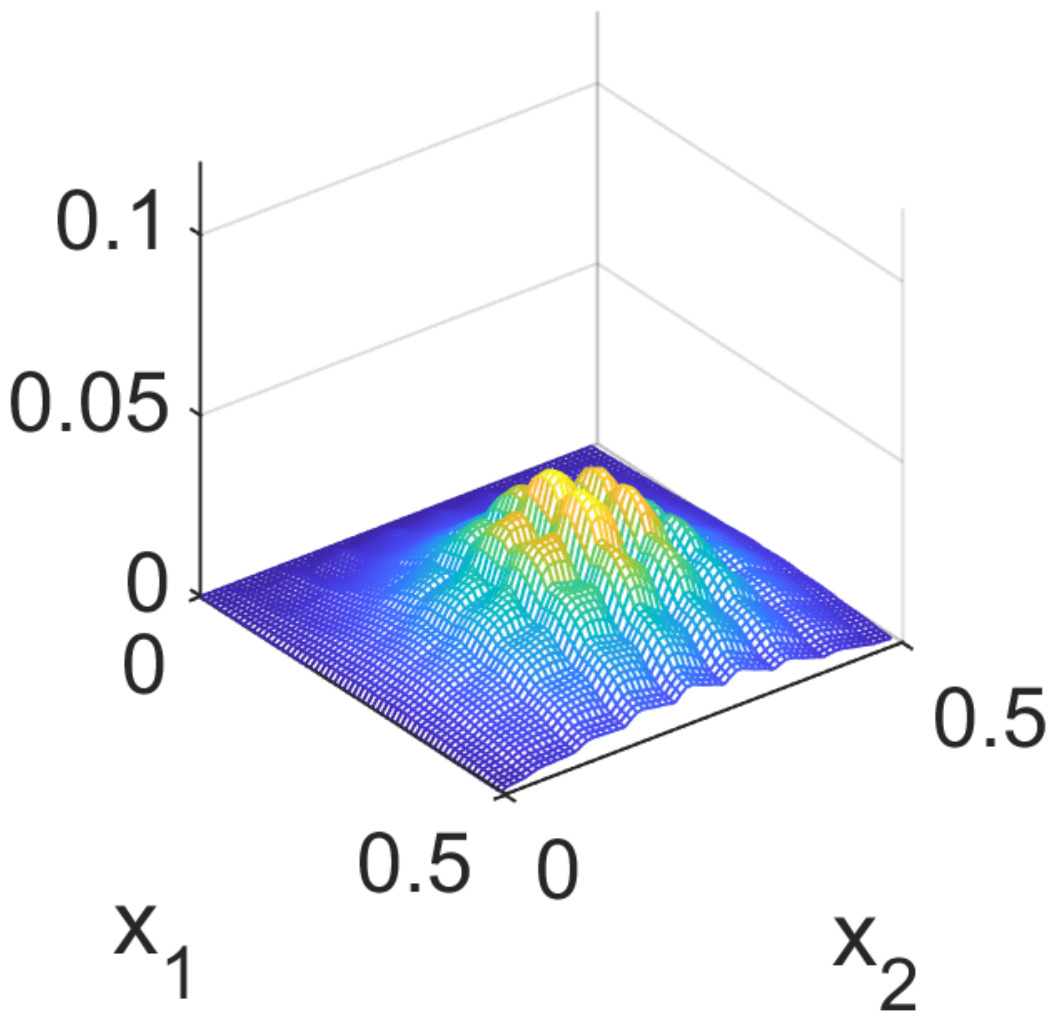}

  \caption{Reference solution at $\theta = 0$ for the RTE~\eqref{eqn:RTE} for different choices of $(\eps_1,\eps_2)$. The columns from left to right show, respectively, $\eps_2=2^{0}$, $2^{-2}$ and $2^{-4}$. The rows from top to bottom show, respectively, $\eps_1=2^{0}$, $2^{-2}$ and $2^{-4}$. The source $f(x_1,x_2,\theta)$ is defined by \eqref{def.f}.}
  \label{fig:RTE_ref_soln_fgauss}
\end{figure}

\begin{figure}[htbp]
  \centering
  \includegraphics[width=0.3\textwidth]{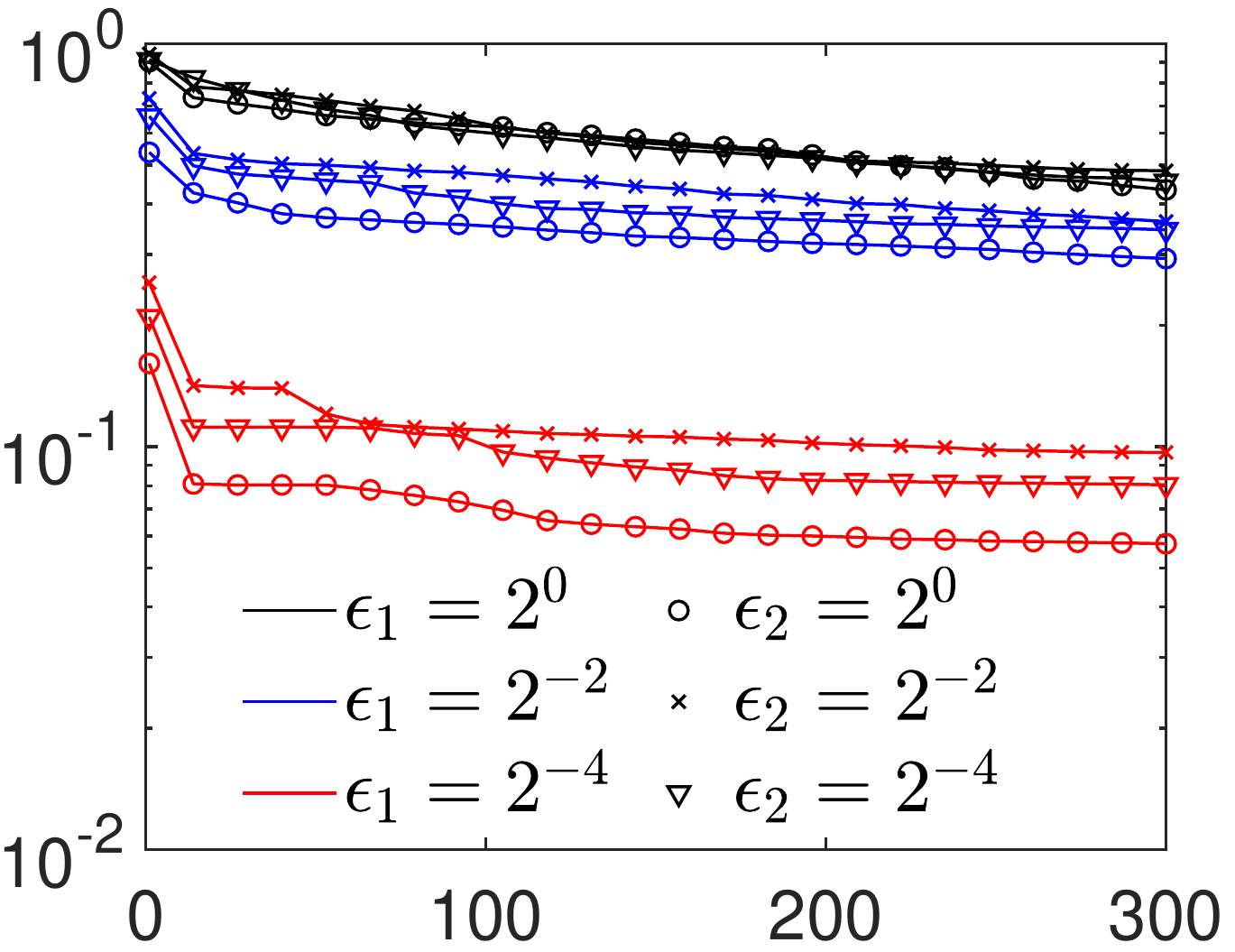}
  \includegraphics[width=0.3\textwidth]{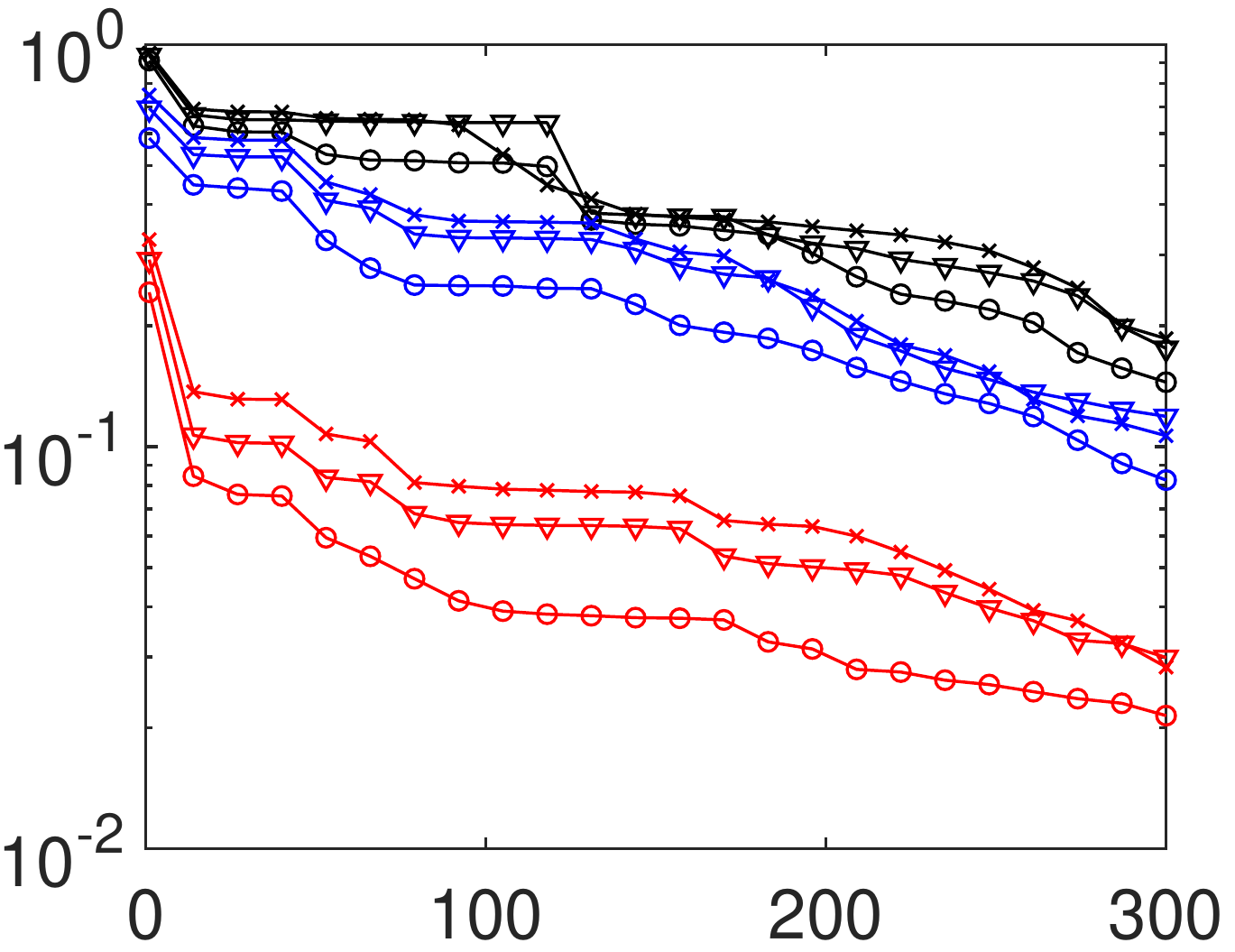}
  \includegraphics[width=0.3\textwidth]{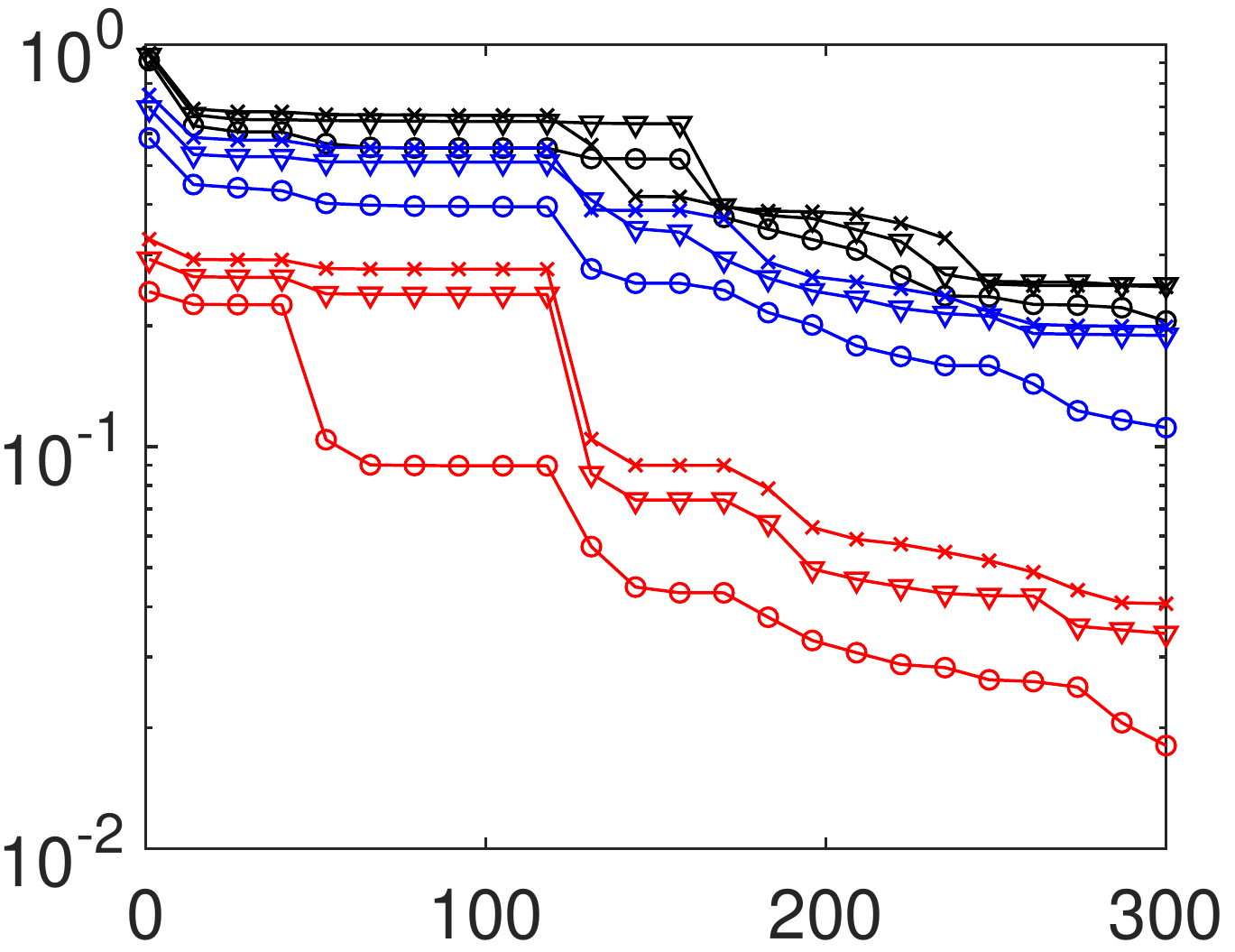}
  \caption{Relative $L^2$ error for the RTE~\eqref{eqn:RTE} with different media and smoothness indices $p=0$ (left), $p=1$ (middle) and $p=2$ (right).
    The source $f(x_1,x_2,\theta)$ is defined by \eqref{def.f}.}
  \label{fig:RTE_err}
\end{figure}

\subsection{Nonlinear systems}\label{sec:nonlinear_num}

In this section, we study~\eqref{eqn:MP1r} numerically and invert the
nonlinear operator $\mathcal{L} + \mathcal{N}$ for any source term
$f(x)$. Again, we present two examples: a semilinear elliptic equation
and a semlinear radiative transport equation. Both have strong
inhomogeneities in the media and multiscale behavior.

\subsubsection{Semilinear elliptic equation}\label{sec:semlinear_numerics}

In the first example, we consider the semilinear elliptic equation
\begin{equation}\label{eqn:semilinear_elliptic}
\begin{cases}
-\nabla \cdot (\kappa(x_1,x_2) \nabla u(x_1,x_2)) + u^3(x_1,x_2)= f(x_1,x_2),\quad& (x_1,x_2)\in\Omega\,, \\
u(x_1,x_2) = 0,\quad& (x_1,x_2)\in\partial\Omega\,,
\end{cases}
\end{equation}
with the same medium in~\eqref{eqn:medium}. As in~\Cref{sec:linear_elliptic}, to generate the reference solution, we use the finite difference method to discretize~\eqref{eqn:semilinear_elliptic} on uniform grid with mesh size $h = 2^{-6} = \tfrac{1}{64}$.

The semilinear elliptic equation has the form of~\eqref{eqn:MP1r} when the linear operator $\mathcal{L}$ and the nonlinear term $\mathcal{N}$  are defined by
\begin{equation}
    \mathcal{L} u = -\nabla \cdot (\kappa(x_1,x_2) \nabla u)\,, \quad \mathcal{N}(u) = u^3 \,.
\end{equation}
For this choice of $\mathcal{L}$ and $\mathcal{N}$, we can reuse the optimal basis obtained in the linear elliptic example in~\Cref{sec:linear_elliptic}. The fixed-point iteration (~\Cref{ALG2}) is implemented to find the coefficients of the optimal basis.

The optimal basis of the operator $\mathcal{L}$ provides accurate solution representation of the nonlinear PDEs.
We test the representation power of the basis using the source $100\sin(4\pi x_1)\sin(4\pi x_2)$.
In \cref{fig:semilinear_ref_soln} we show reference solutions for different $\eps$ as well as the relative $L^2$ error and energy error.
In all three test problems, smaller $\eps$ does not deteriorate the performance of the optimal basis, and larger $p$ produces smaller error for sources that are less oscillatory.

\begin{figure}[htbp]
  \centering
  \includegraphics[width=0.25\textwidth]{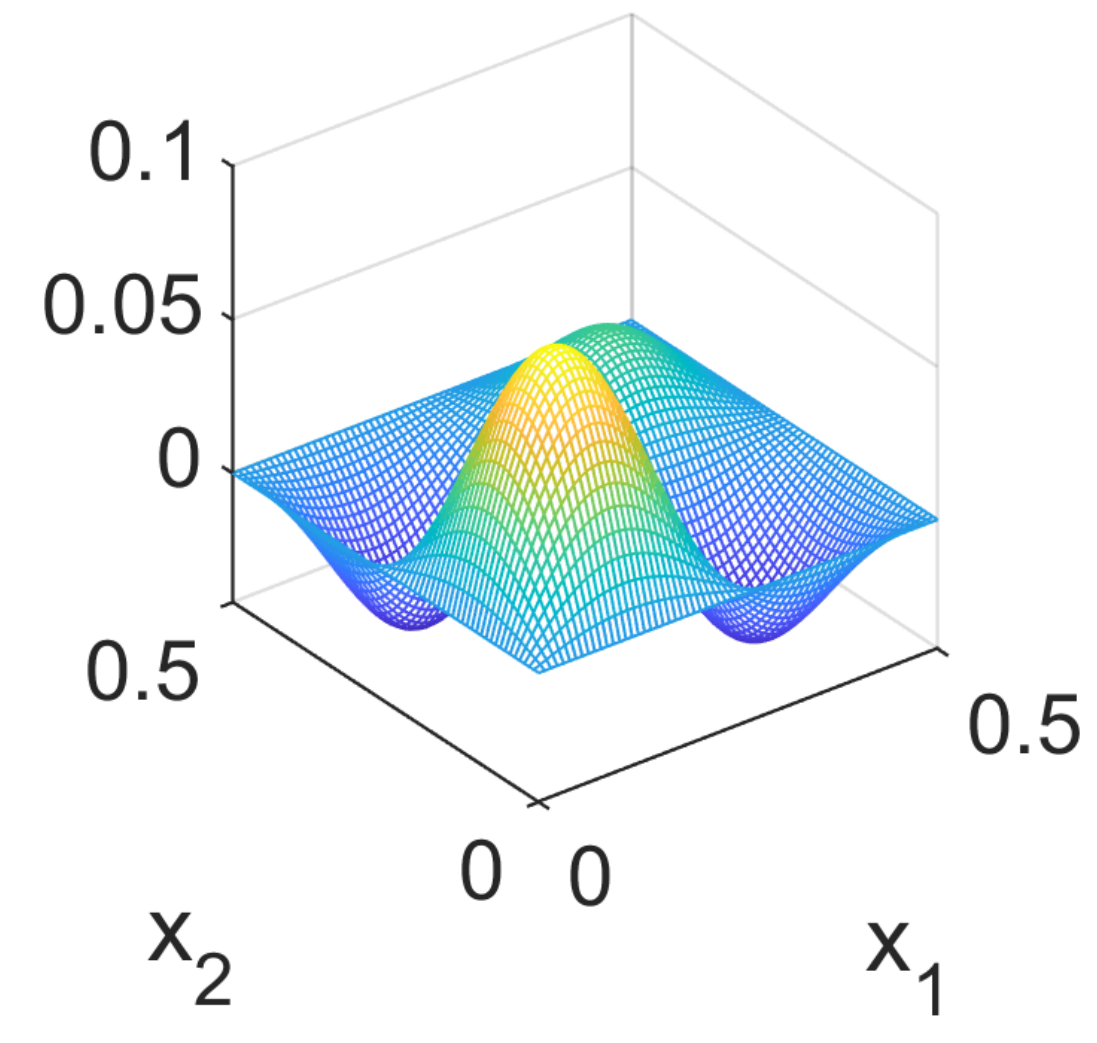}
  \includegraphics[width=0.25\textwidth]{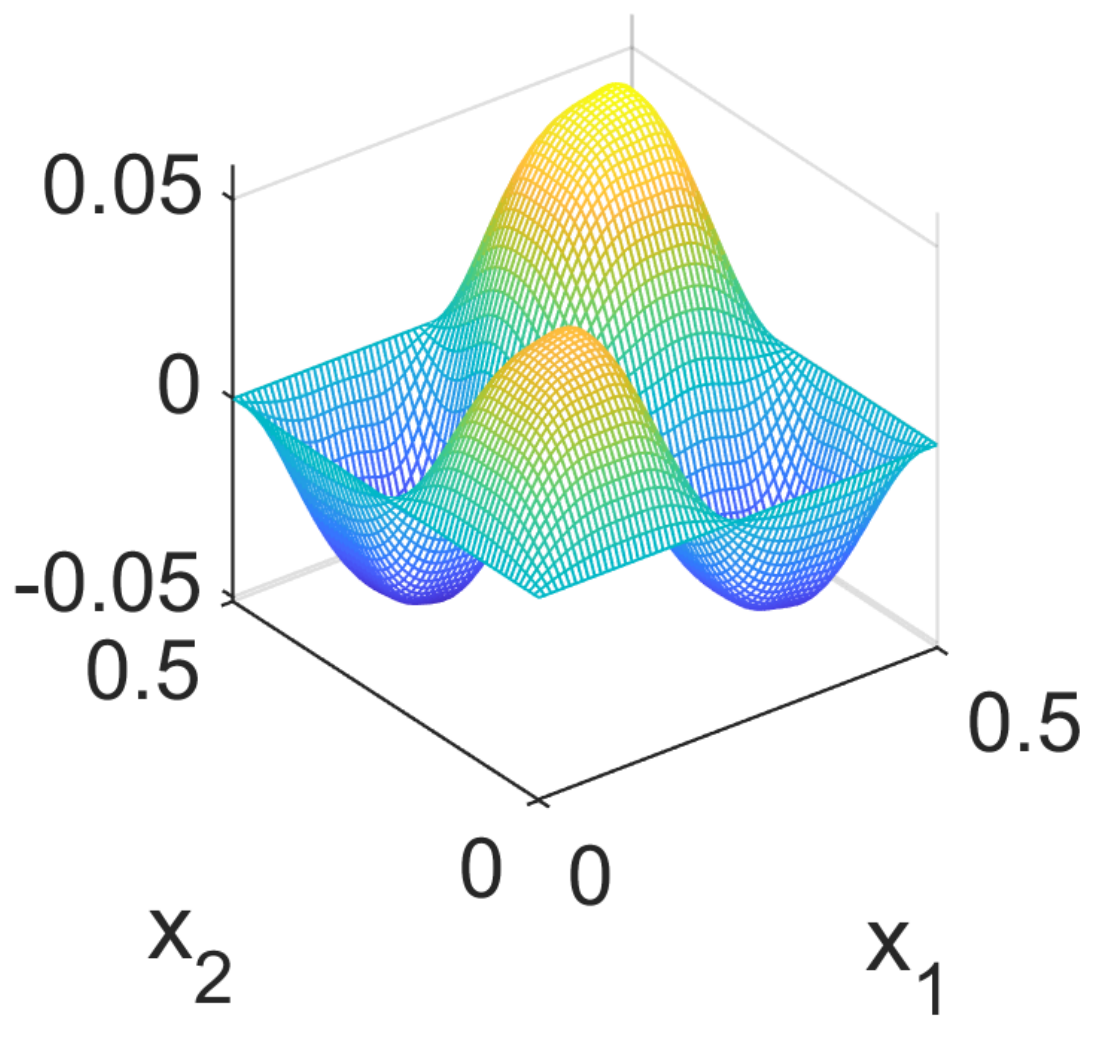}
  \includegraphics[width=0.25\textwidth]{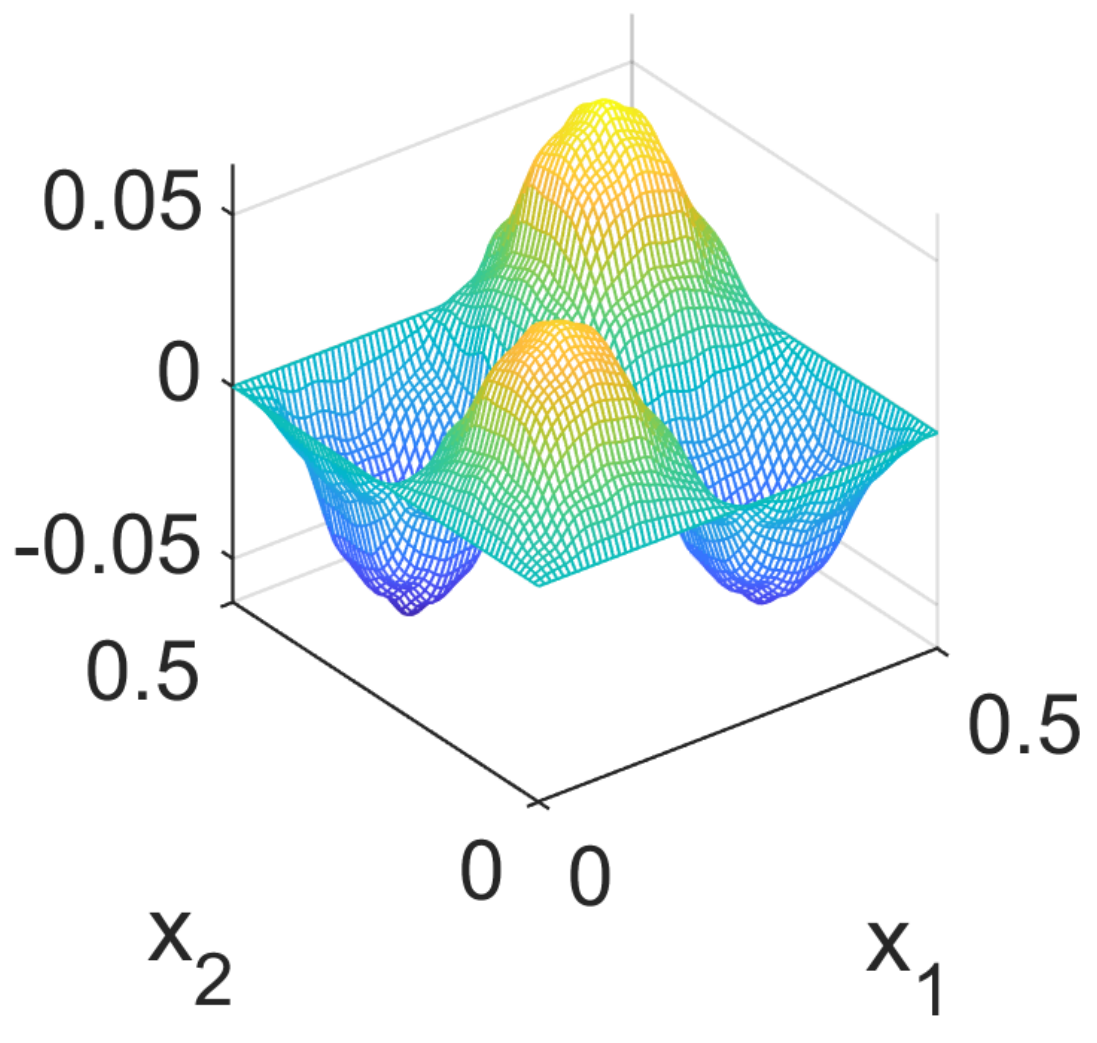}\\
  \includegraphics[width=0.25\textwidth]{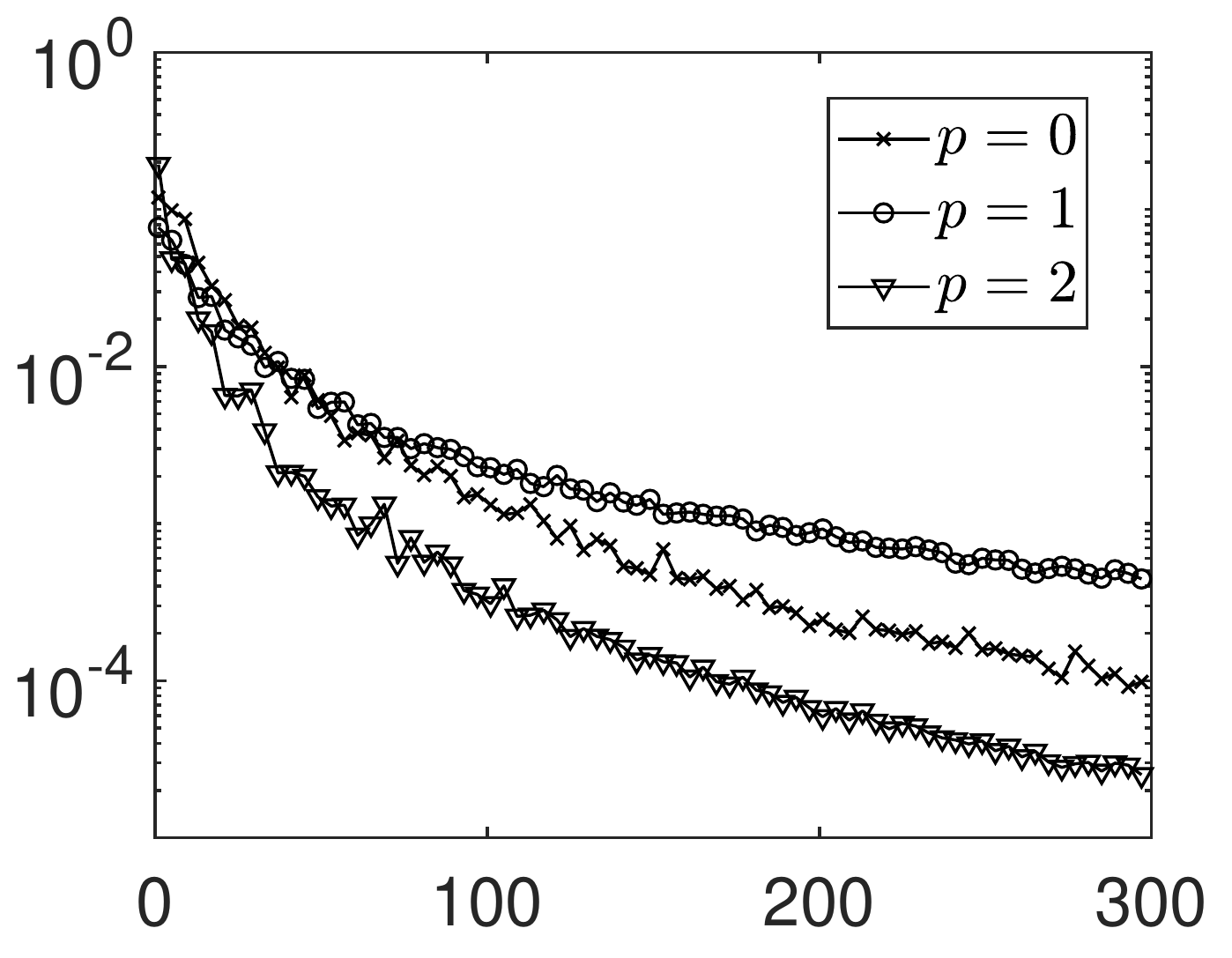}
  \includegraphics[width=0.25\textwidth]{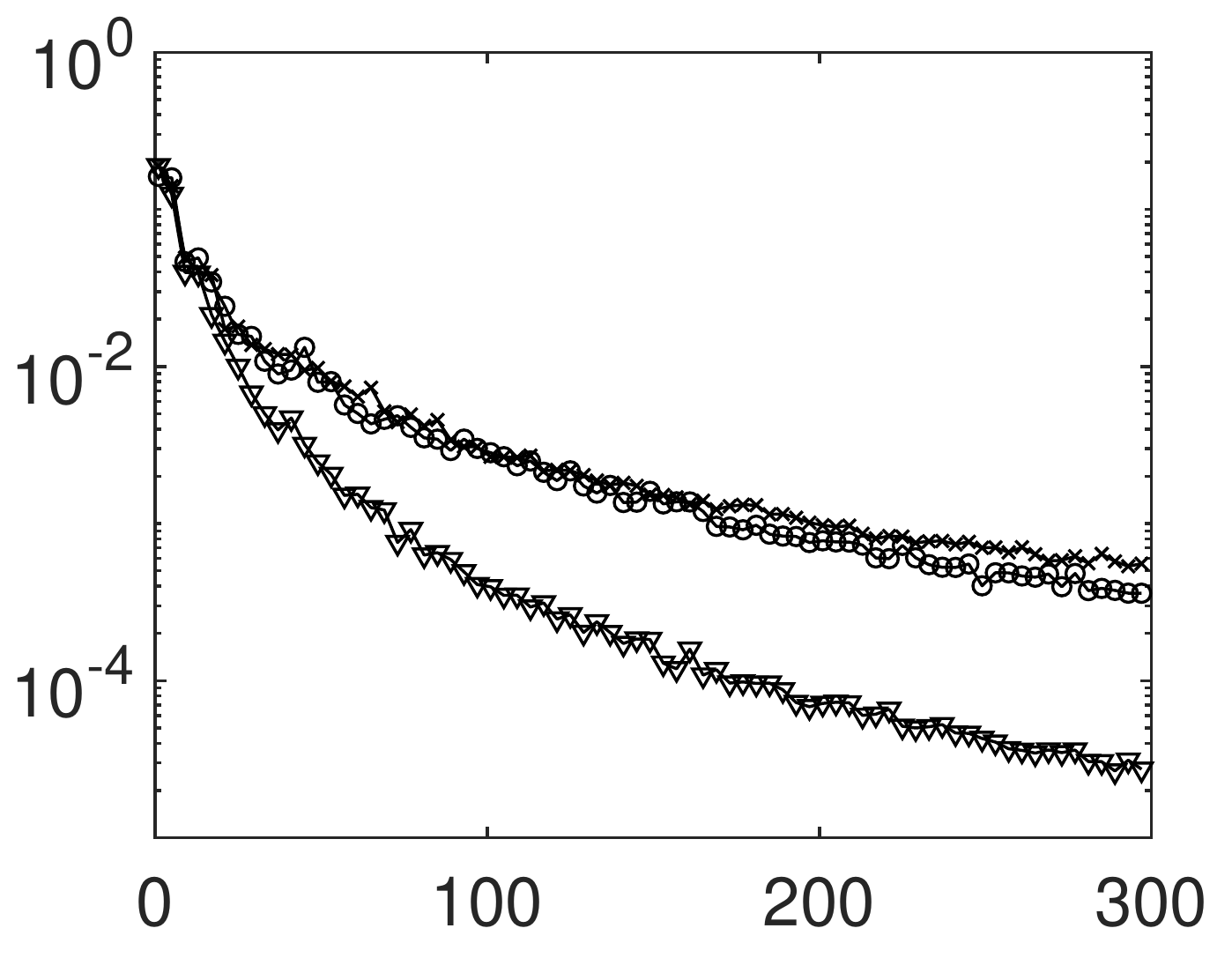}
  \includegraphics[width=0.25\textwidth]{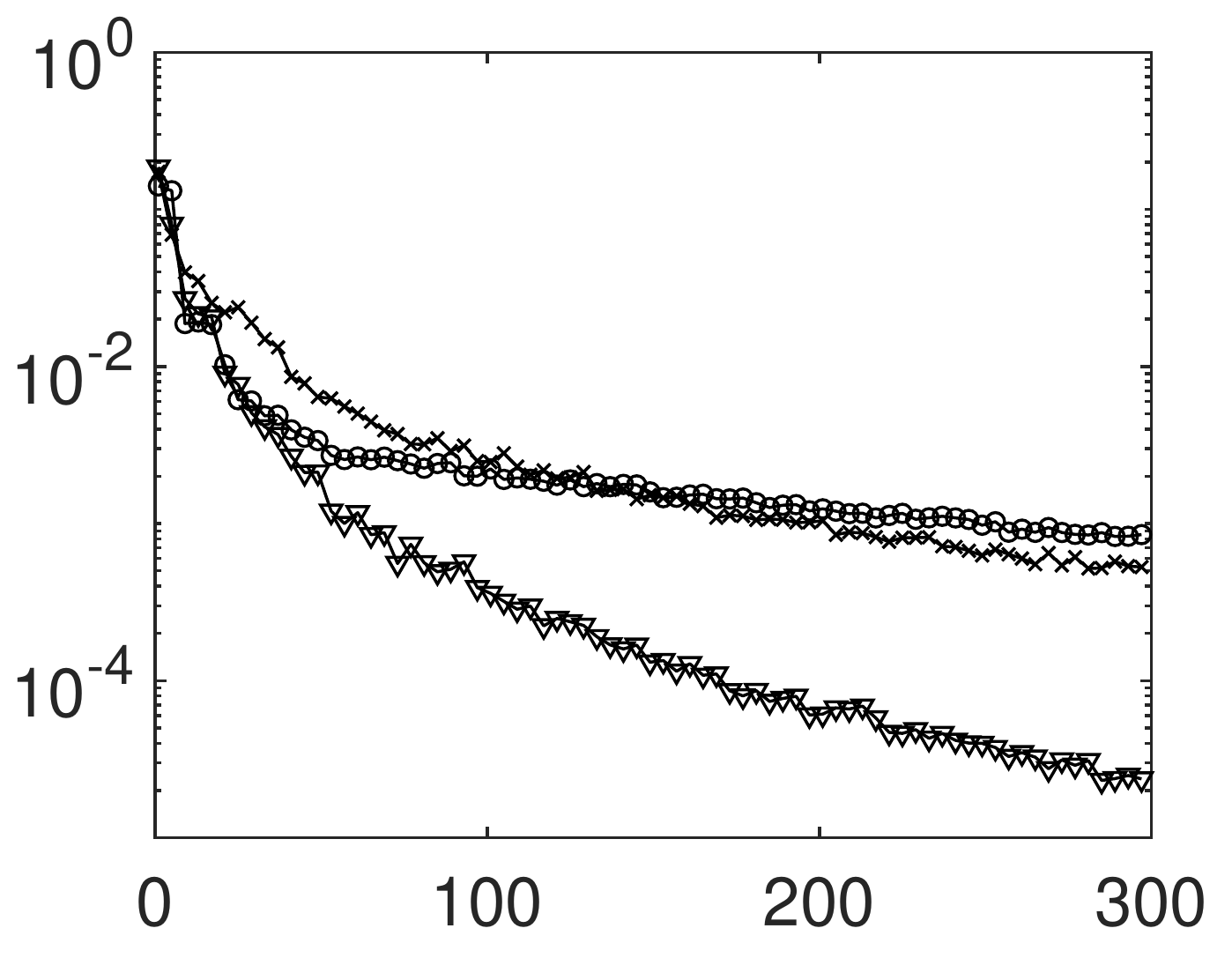}\\
  \includegraphics[width=0.25\textwidth]{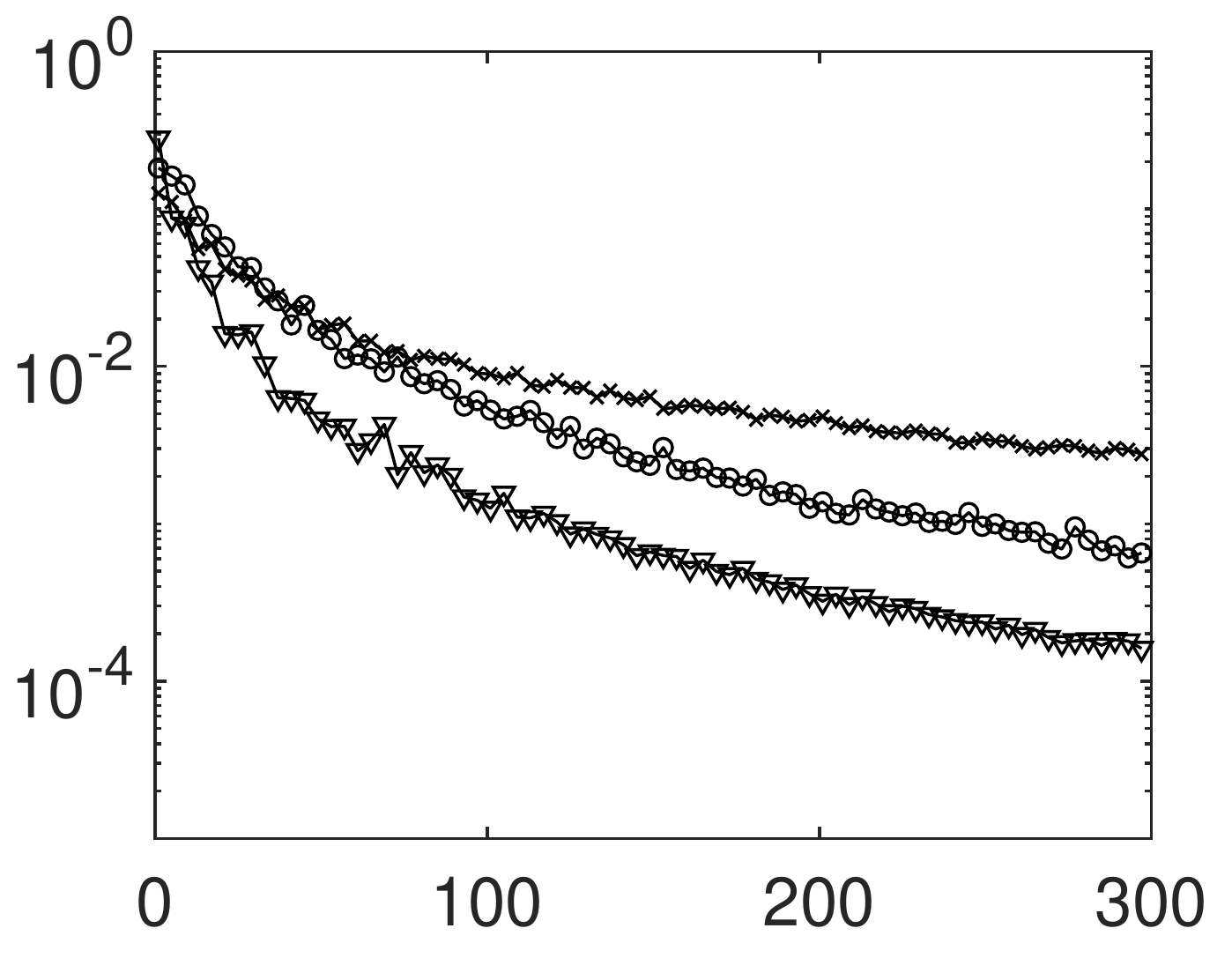}
  \includegraphics[width=0.25\textwidth]{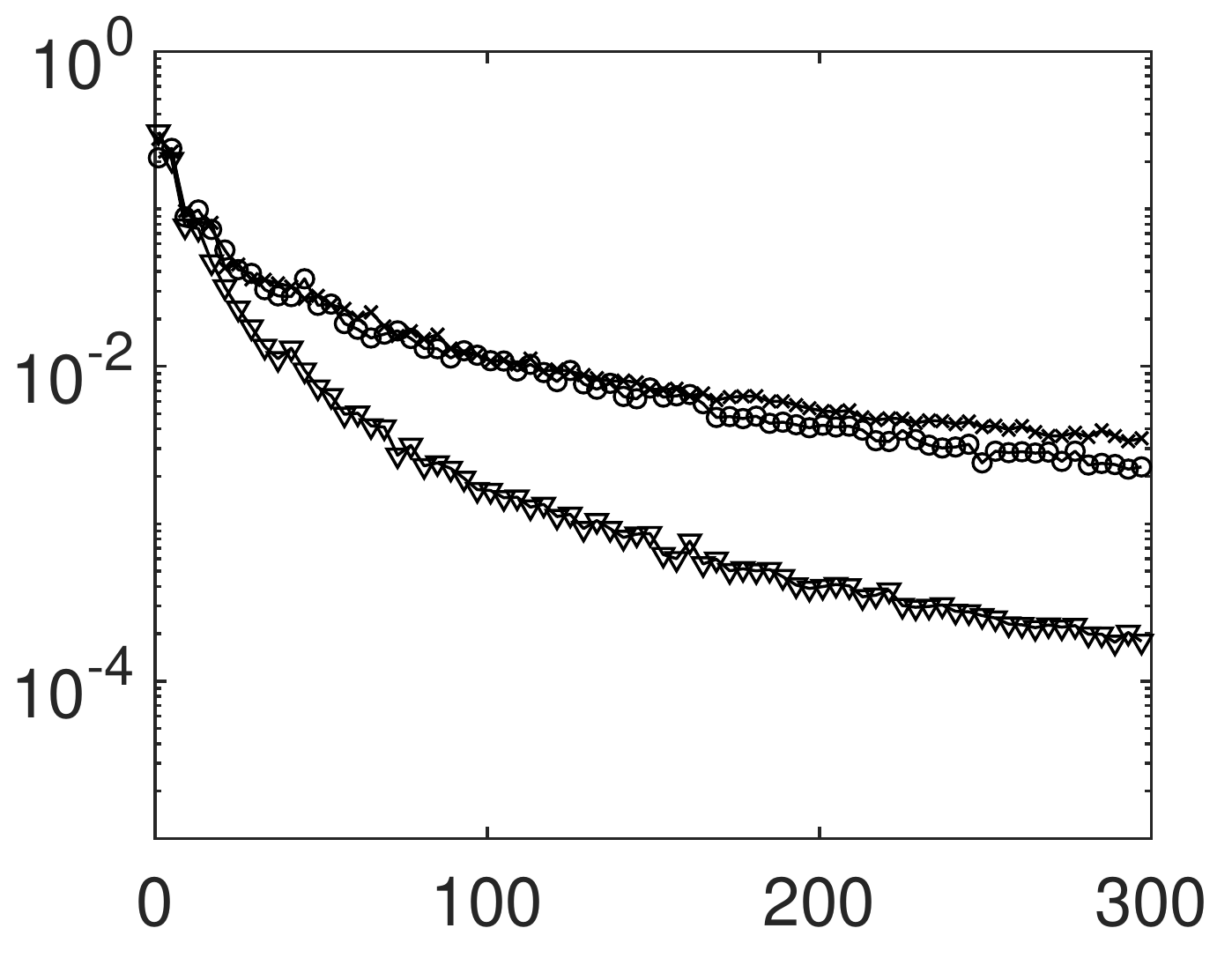}
  \includegraphics[width=0.25\textwidth]{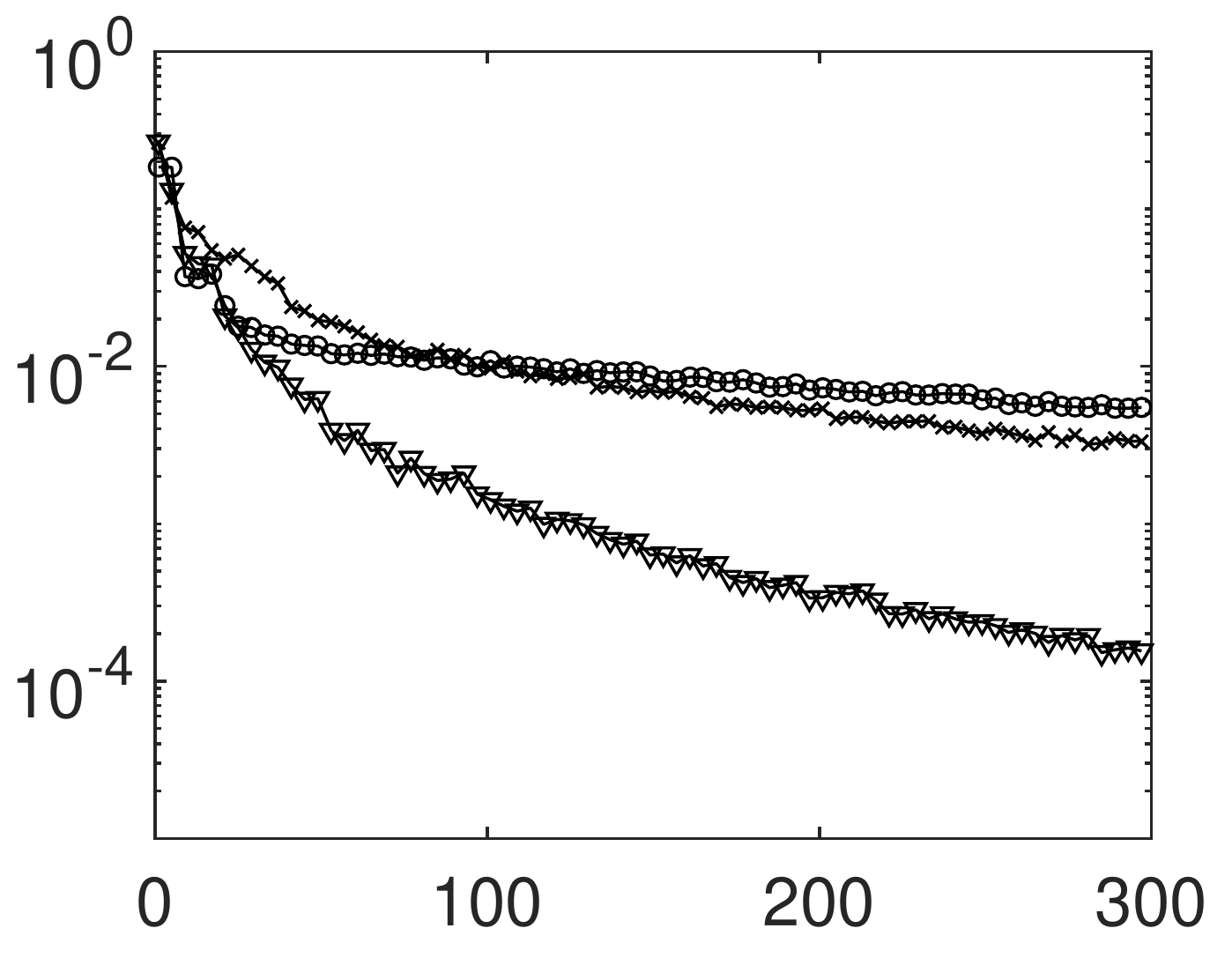}

  \caption{
  First row plots reference solutions for the semilinear elliptic equation~\eqref{eqn:semilinear_elliptic} with $\eps=2^{0}$ (left), $2^{-2}$ (middle) and $2^{-4}$ (right) using the medium~\eqref{eqn:medium}. Second row plots  the decay of $L^2$-norm as the number of basis increases with $\eps=2^{0}$ (left), $2^{-2}$ (middle) and $2^{-4}$ (right). Third row plots the decay of energy-norm with $\eps=2^{0}$ (left), $2^{-2}$ (middle) and $2^{-4}$ (right).
  The $L^2$ error achieves $10^{-4}$ with basis size $300$, in comparison to $3969$ degree of freedom}
  \label{fig:semilinear_ref_soln}
\end{figure}

\subsubsection{Semilinear radiative transport equation}\label{sec:semilinear_rte_numerics}

The last example concerns the semilinear radiative transport equation with two-photon absorption:
\begin{equation}\label{eqn:semilinearRTE}
\begin{cases}
v\cdot \nabla_x u + (\sigma_\mathrm{a}(x) + \sigma_\mathrm{s}(x) + \sigma_\mathrm{b}(x) \langle u \rangle (x)) u
&=  \sigma_\mathrm{s}(x) \mathcal{K} u(x,v)   + f(x,v),\quad x\in\Omega\,,v\in\Sb^1\,, \\
 u(x,v) &= 0,\quad (x,v)\in\Gamma_-\,,
 \end{cases}
\end{equation}
where the scattering operator $\mathcal{K}$, the scattering coefficient $\sigma_{\mathrm{s}}(x)$ and the single-photon absorption coefficient $\sigma_{\mathrm{a}}(x)$ are defined in~\eqref{eqn:scat_op}-\eqref{eqn:absp_rte}. We denote by $\langle u \rangle$ the average of $u(x,v)$ in the velocity $v$, that is,
\[
\langle u \rangle = \int_{\mathbb{S}^1} u(x,v) \rmd \mu(v)\,,
\]
where $\mu$ is the normalized spherical measure on $\Sb^{1}$. The coefficient $\sigma_{\mathrm{b}}$ is called the two-photon absorption coefficient and it models the molecule excitation process that simultaneously absorbs two photons. In the following experiments, we choose the two-photon absorption coefficient as
\begin{equation}
    \sigma_\mathrm{b}(x) = 0.1 \eps_1 (2 + 0.5\cos(x_1) + 0.5\sin(x_2) )\,.
\end{equation}

To implement~\cref{ALG2}, we separate the linear terms and the nonlinear term directly by choosing $\mathcal{L}$ and $\mathcal{N}$ in~\eqref{eqn:MP1r} as the following
\begin{equation}
    \mathcal{L} u = v\cdot \nabla_x u + (\sigma_\mathrm{a}(x) + \sigma_\mathrm{s}(x) ) u
-  \sigma_\mathrm{s}(x) \mathcal{K} u(x,v)\,, \quad \mathcal{N}(u) = \sigma_\mathrm{b}(x) \langle u \rangle (x) u \,.
\end{equation}
We thus reuse the optimal basis obtained in the RTE example in~\Cref{sec:RTE}. The fixed-point iteration of~\cref{ALG2} is again implemented.

As in~\Cref{sec:RTE}, we use the upwind method and trapezoidal rule to discretize~\eqref{eqn:semilinearRTE} with uniform grid with mesh size $h = 2^{-6} = \tfrac{1}{64}$.

In \cref{fig:semilinearRTE_err} we plot the relative $L^2$ error as a function of the number of bases for different sources, media, and settigns of $p$. The example shows that different values of $\eps_2$ do not deteriorate the performance of the optimal basis, but on average smaller $\eps_1$ gives better approximation using smaller number of basis functions. Larger values of $p$ make the source term more regular, thus producing smaller error.

\begin{figure}[htbp]
  \centering
  \includegraphics[width=0.3\textwidth]{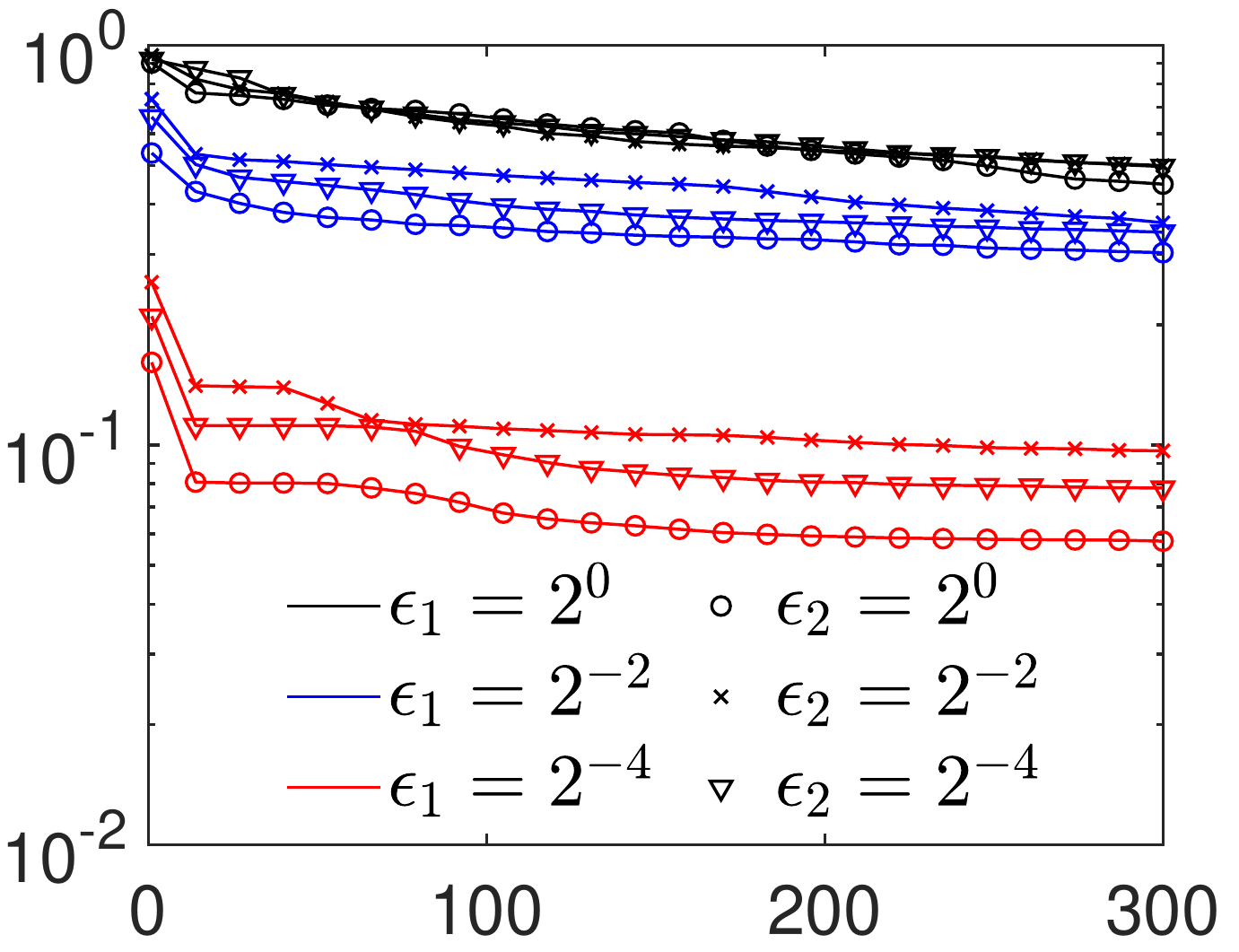}
  \includegraphics[width=0.3\textwidth]{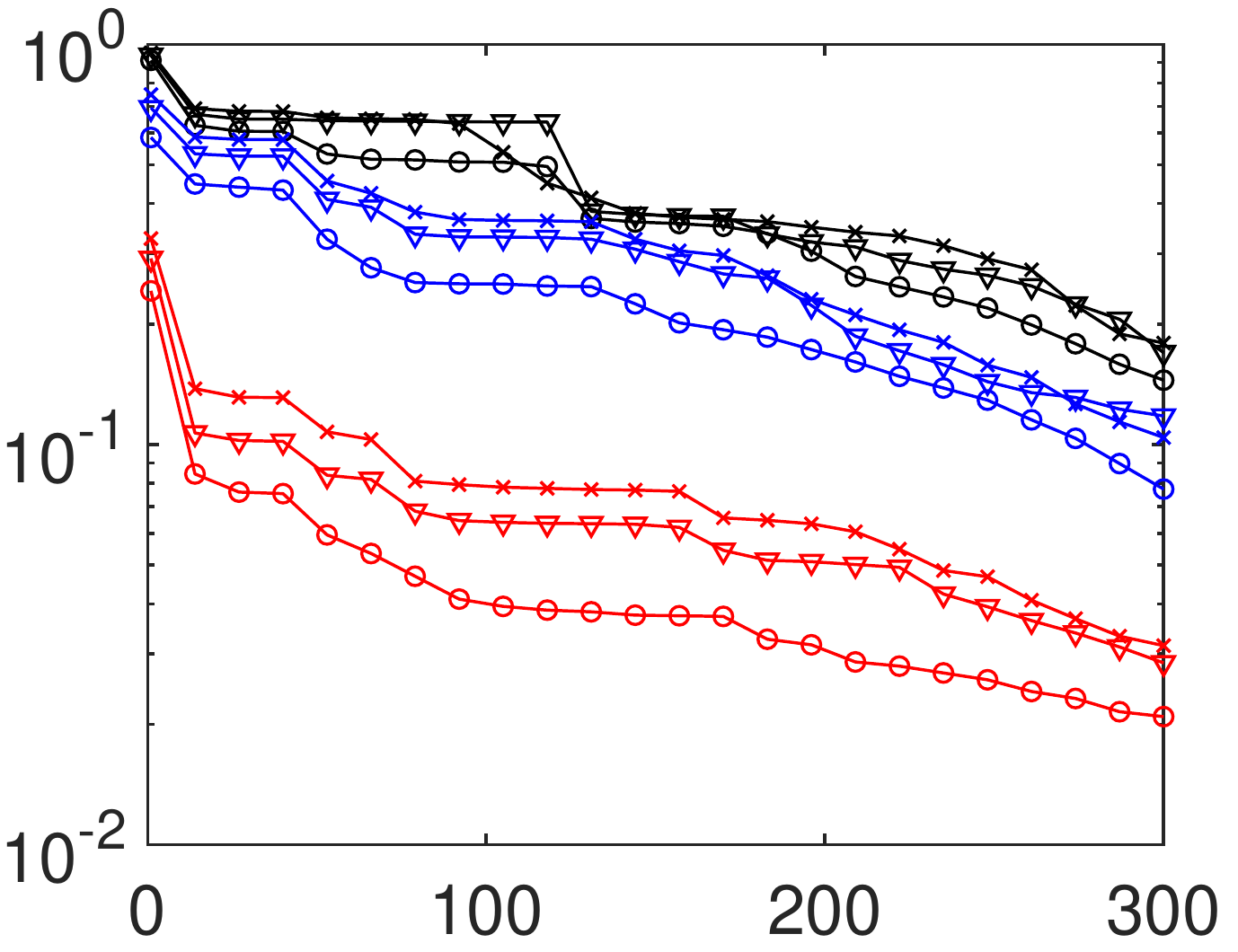}
  \includegraphics[width=0.3\textwidth]{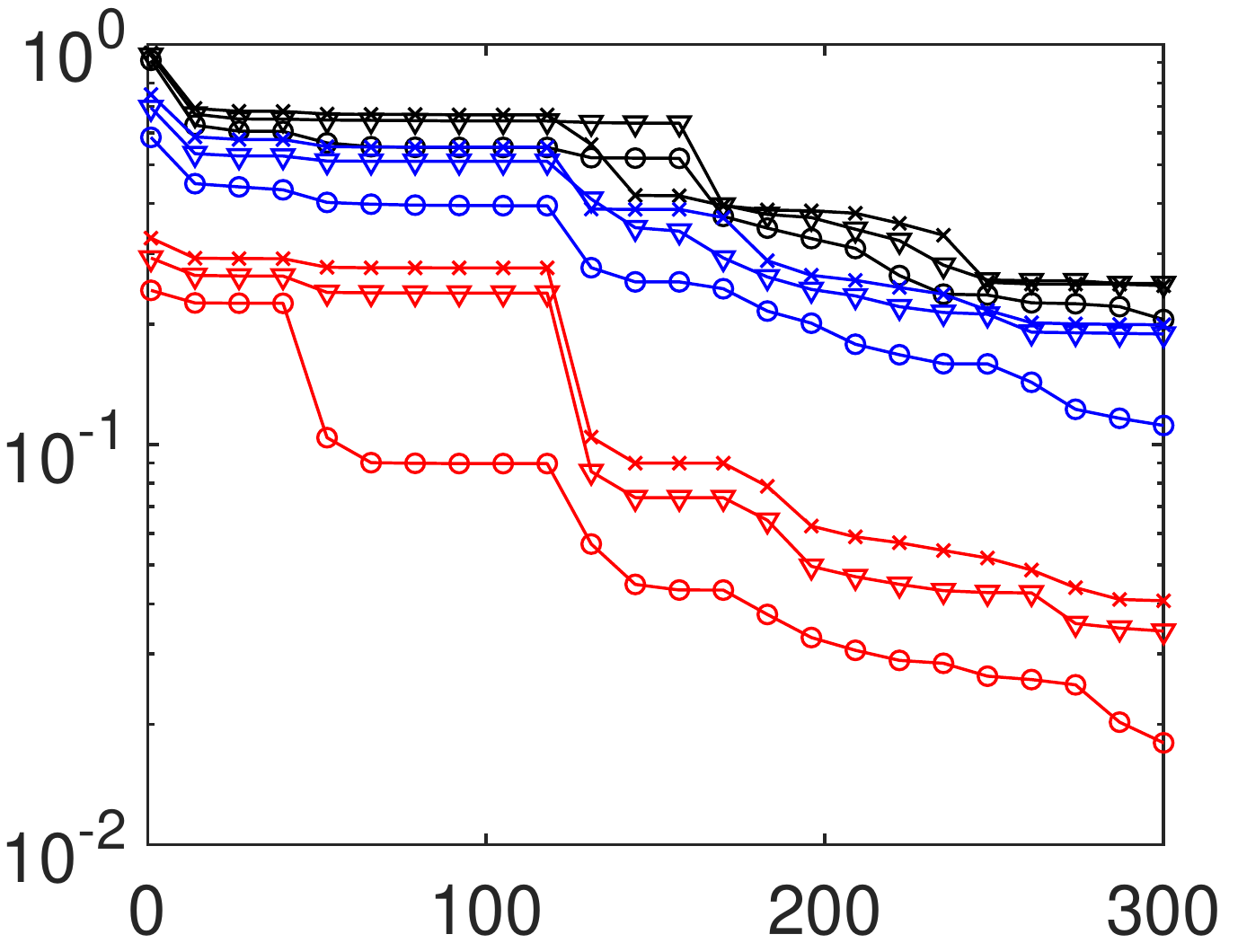}
  \caption{Relative $L^2$ error for the semilinear RTE~\eqref{eqn:semilinearRTE} with different media, sources and smoothness indices $p=0$ (left), $p=1$ (middle) and $p=2$ (right). The source is $0.1$ of the source given in~\eqref{def.f}.
  }
  \label{fig:semilinearRTE_err}
\end{figure}

\section*{Acknowledgments}
Q. Li thanks Thomas Y. Hou and Houman Owhadi for the many insights in the discussion.

\newpage

\appendix

\section{Proof of Theorem~\ref{thm:nonlinearcase2}}\label{app:thm_nonlinearcase2}
Before proving the two inequalities, we first define
\[
 c_i = \left\langle f-\mathcal{N}(u), \hat{v}_{i} \right\rangle\,.
\]
Since $u$ is the true solution, we have
\[
u=\sum^\infty_{i=1} \lambda_i c_i\hat{u}_{i},\quad
\mathcal{L}u=\sum^\infty_{i=1}c_i\hat{v}_{i} = f-\mathcal{N}\left(\sum^\infty_{i=1}\lambda_i c_i\hat{u}_{i}\right)\,.
\]
\noindent\textbf{Proof of~\eqref{eqn:firsterrorbound}:}
To prove this inequality, we first denote
\[
u_n = \sum^n_{i=1} \lambda_i \hat{c}_i \hat{u}_{i}\,,\quad\text{with $\{\hat{c}_i \}$ solving~\eqref{eqn:firstchoiceofci}}.
\]
We first formulate three sets of equations:
\[
\begin{aligned}
\mathcal{L}u+\mathcal{N}{(u)} &= f\\
\mathcal{L}\left(\sum^n_{i=1} \lambda_i c_i \hat{u}_{i}\right) +\mathcal{N}{\left(\sum^n_{i=1} \lambda_i c_i \hat{u}_{i}\right)} & = f+\mathcal{R}_\text{truncation}\\
\mathcal{L}u_n+\mathcal{N}{(u_n)} &= f +\mathcal{R}_1\\
\end{aligned}
\]
where
\[
\mathcal{R}_\text{truncation}=\sum^n_{i=1}c_i\hat{v}_{i} - f + \mathcal{N}\left(\sum^n_{i=1}\lambda_i c_i\hat{u}_{i}\right)\,,\quad
\mathcal{R}_1=\sum^n_{i=1}\hat{c}_i\hat{v}_{i} - f + \mathcal{N}\left(\sum^n_{i=1}\lambda_i \hat{c}_i\hat{u}_{i}\right)\,.
\]
According to the optimality condition~\eqref{eqn:firstchoiceofci}, we have
\[
\|\mathcal{R}_1\|_\mathcal{X}\leq\|\mathcal{R}_\text{truncation}\|_\mathcal{X}\,.
\]
Then according to the stability Assumption~\ref{assum:stability}, we have
\[
\|u_n-u\|\leq E(\|\mathcal{R}_1\|_\mathcal{X})\leq E(\|\mathcal{R}_\text{truncation}\|_\mathcal{X})\,.
\]
To analyze $\mathcal{R}_\text{truncation}$, we have:
\[
\begin{aligned}
\|\mathcal{R}_\text{truncation}\|_\mathcal{X}=&\left\|\sum^n_{i=1}c_i\hat{v}_{i} - f + \mathcal{N}\left(\sum^n_{i=1}\lambda_i c_i\hat{u}_{i}\right)\right\|_{\mathcal{X}}\\
\leq &\left\|\sum^\infty_{i=n+1}c_i\hat{v}_{i}\right\|_\mathcal{X}+\left\|\mathcal{N}\left(\sum^n_{i=1}\lambda_i c_i\hat{u}_{i}\right)-\mathcal{N}\left(\sum^\infty_{i=1}\lambda_i c_i\hat{u}_{i}\right)\right\|_{\mathcal{X}}\\
\leq &\left\|\sum^\infty_{i=n+1}c_i\hat{v}_{i}\right\|_\mathcal{X}+C\left\|\sum^\infty_{i=n+1}\lambda_i c_i\hat{u}_{i}\right\|_\mathcal{Y}\\
\leq &\left(1+C\lambda_{n+1}\right)
\mathcal{E}^{(1)}_n\,.
\end{aligned}
\]
in summary, we conclude that
\[
\|u_n-u\|_{\mathcal{Y}}\leq E\left(\left(1+C\lambda_{n+1}\right)
\mathcal{E}^{(1)}_n
\right),
\]
which proves \eqref{eqn:firsterrorbound}.

\noindent\textbf{Proof of~\eqref{eqn:seconderrorbound}}. To prove this inequality, we first denote
\[
u_n = \sum^n_{i=1} \lambda_i \hat{c}_i \hat{u}_{i}\,,\quad\text{with $\{\hat{c}_i \}$ solving \eqref{eqn:secondchoiceofci}}
\]
To proceed, we first define $u^{(3)}=u_n+u^{(2)}$ with $u^{(2)}$ solving
\begin{equation}\label{eqn:utilde}
\mathcal{L} (u^{(2)}(x))=P^{\perp}_n\left(f-\mathcal{N}\left(u_n\right)\right)\,.
\end{equation}
This gives:
\[
\left\|u_n-u\right\|_{\mathcal{Y}}\leq \left\|u^{(2)}\right\|_{\mathcal{Y}} + \left\|u^{(3)}-u\right\|_{\mathcal{Y}}\,.
\]
To control $u^{(2)}$, recall its definition~\eqref{eqn:utilde}:
\begin{equation}\label{eqn:wubound}
\|u^{(2)}\|_{\mathcal{Y}}\leq \lambda_{n+1}\left\|P^{\perp}_n\left(f-\mathcal{N}\left(u_n\right)\right)\right\|_{\mathcal{X}}=\lambda_{n+1}\mathcal{E}_{n}^{(2)}\,.
\end{equation}
To control $u^{(3)}-u$, according to stability Assumption~\ref{assum:stability},
\begin{equation}\label{eqn:bound}
\begin{aligned}
\|u^{(3)}-u\|_\mathcal{Y} & \leq
 E\left( \left\|\mathcal{L}\left(u^{(3)}\right)-\left(f-\mathcal{N}\left(u^{(3)}\right)\right)\right\|_{\mathcal{X}} \right)\\
& = E\left( \left\|\sum^n_{i=1}\hat{c}_i\hat{v}_{i}+\mathcal{L}(u^{(2)})-\left(f-\mathcal{N}\left(\sum^n_{i=1}\lambda_i\hat{c}_i\hat{u}_{i}+u^{(2)}\right)\right)\right\|_{\mathcal{X}} \right) \\
& \leq E\Bigg( \left\|\sum^n_{i=1}\hat{c}_i\hat{v}_{i}+\mathcal{L}(u^{(2)})-\left(f-\mathcal{N}\left(\sum^n_{i=1}\lambda_i\hat{c}_i\hat{u}_{i}\right)\right)\right\|_{\mathcal{X}} \\
&\qquad +\left\|\mathcal{N}\left(\sum^n_{i=1}\lambda_i\hat{c}_i\hat{u}_{i}+u^{(2)}\right)-\mathcal{N}\left(\sum^n_{i=1}\lambda_i\hat{c}_i\hat{u}_{i}\right)\right\|_{\mathcal{X}} \Bigg) \\
& \leq E\left( \left\|\sum^n_{i=1}\hat{c}_i\hat{v}_{i}-P_n\left(f-\mathcal{N}\left(\sum^n_{i=1}\lambda_i\hat{c}_i\hat{u}_{i}\right)\right)\right\|_{\mathcal{X}}+C\lambda_{n+1}\mathcal{E}_{n}^{(2)} \right)
\end{aligned}
\end{equation}
where we use \eqref{eqn:utilde} and locally Lipschitz of $\mathcal{N}$ in the last inequality. Noting from~\eqref{eqn:secondchoiceofci} that
\[
\left\|\sum^n_{i=1}\hat{c}_{i}\hat{v}_{i}-P_n\left(f-\mathcal{N}\left(\sum^n_{i=1}\lambda_i\hat{c}_{i}\hat{u}_{i}\right)\right)\right\|_{\mathcal{X}}\leq \left\|\sum^n_{i=1}c_i \hat{v}_{i}-P_n\left(f-\mathcal{N}\left(\sum^n_{i=1}\lambda_i c_i \hat{u}_{i}\right)\right)\right\|_{\mathcal{X}}\,,
\]
we have
\[
\begin{aligned}
\left\|u^{(3)}-u\right\|_{\mathcal{Y}} & \leq E\left( \left\|\sum^n_{i=1}c_i \hat{v}_{i}-P_n\left(f-\mathcal{N}\left(\sum^n_{i=1}\lambda_i c_i \hat{u}_{i}\right)\right)\right\|_{\mathcal{X}}+C\lambda_{n+1}\mathcal{E}_{n}^{(2)} \right)\\
& \leq E\left( C\left\|\sum^\infty_{i=n+1}\lambda_i c_i\hat{u}_{i}\right\|_{\mathcal{Y}}+C\lambda_{n+1}\mathcal{E}_{n}^{(2)} \right)\\
& \leq E\left( C\lambda_{n+1}\left(\left\|\sum^\infty_{i=n+1}c_i\hat{v}_{i}\right\|_{\mathcal{X}}+\mathcal{E}_{n}^{(2)}\right) \right)\\
& \leq  E\left( C\lambda_{n+1}\left(\mathcal{E}_{n}^{(1)}+\mathcal{E}_{n}^{(2)}\right) \right),
\end{aligned}\,,
\]
where we use definition of $c_i$ and $\mathcal{N}$ is locally Lipschitz in the fourth inequality.
This estimate, when combined with~\eqref{eqn:wubound} gives:
\[
\begin{aligned}
\left\|u_n-u\right\|_{\mathcal{Y}}\leq \left\|u^{(2)}\right\|_{\mathcal{Y}} + \left\|u^{(3)}-u\right\|_{\mathcal{Y}} \leq \lambda_{n+1}\mathcal{E}_{n}^{(2)} + E\left(C\lambda_{n+1}\left(\mathcal{E}_{n}^{(1)}+\mathcal{E}_{n}^{(2)}\right)\right),
\end{aligned}
\]
concluding the proof for~\eqref{eqn:seconderrorbound}.

\section{Weight matrix}\label{appendix:Pi}
The weight matrix $\mathsf{\Pi}_\mathsf{X}$ is adjusted to reflect the metric used for the space $\mathsf{X}$. We discuss how to set weight matrix for $H^p$ where $p\geq1$. Let $\mathsf{f}\in\mathsf{X}$, and  denote the grid points in $x$-direction by $0=x_0<x_1<\cdots<x_{m-1}<x_m=.5$ and the grid points in $y$-direction by $0=y_0<y_1<\cdots<y_{m-1}<y_m=.5$, we reshape $\mathsf{f}$ into a matrix with $\mathsf{f}_{ij}$ denoting its evaluation at $(x_i,y_j)$. For $k\in\mathbb{N}$, we first define the 1D $k$th order finite difference operator $\mathsf{D}^k: \Rb^{m-1}\to \Rb^{m-k-1}$ inductively, as follows:
\begin{equation}
    [\mathsf{D}^k \mathsf{w}]_i = \frac{[\mathsf{D}^{k-1} \mathsf{w}]_{i+1}-[\mathsf{D}^{k-1}\mathsf{w}]_i}{h}\,, \quad i = 1,\dots, n-k-1 \,, \quad \forall \mathsf{w}\in\Rb^{m-1} \,,
\end{equation}
with $\mathsf{D}^0 = \mathsf{I}$ being the identity matrix. As a counterpart to the differential operator $\partial_{x}^i\partial_{y}^j$, the 2D finite difference operator $\mathsf{D}^{i,j}:\Rb^{(m-1)^2}\to \Rb^{(m-1-i)(m-1-j)}$ is defined as the matrix tensor product of 1D finite difference operators
\begin{equation}\label{eqn:2D-FD}
     \mathsf{D}^{i,j} = \mathsf{D}^i \otimes \mathsf{D}^j\,.
\end{equation}
We then define the discrete Sobolev $p$-norm by
\begin{equation}
    \|\mathsf{f}\|_{p}^2 = h^2\sum_{k=0}^p \sum_{i=0}^k \| \mathsf{D}^{i,p-i} \mathsf{f} \|^2 \,, \quad \forall \mathsf{f}\in\Rb^{(m-1)^2},
\end{equation}
where $\|\cdot\|$ denotes the regular $L^2$ vector-space Euclidean norm. As such, the weight matrix $\mathsf{\Pi}_\mathsf{X}$ would be:
\begin{equation}
    \mathsf{\Pi}_\mathsf{X} = \sum_{k=0}^p \sum_{i=0}^k (\mathsf{D}^{i,p-i})^\top \cdot \mathsf{D}^{i,p-i} \,,
\end{equation}
to represent the metric for $\mathsf{X} = H^p$.

Following the notations above, we can also define the energy norm for the solution:
\begin{equation}\label{eqn:energy_norm_def}
    \|\mathsf{u}\|_\mathrm{E}^2 = h^2\|\mathsf{D}^{1,0}\mathsf{u}\|^2 + h^2\|\mathsf{D}^{0,1}\mathsf{u}\|^2 \,, \quad \forall \mathsf{u}\in\Rb^{(m-1)^2}\,.
\end{equation}

\bibliographystyle{siamplain}
\bibliography{references}

\end{document}